\documentclass[a4paper,10pt,reqno]{amsart}

\usepackage{amsmath}
\usepackage{amsthm}
\usepackage{amsfonts}
\usepackage{amssymb}
\usepackage{mathrsfs}
\usepackage[shortlabels]{enumitem}
\usepackage{graphicx}
\usepackage{subfigure}
\usepackage{color}
\usepackage[dvipsnames]{xcolor}
\usepackage{tikz}

\newtheorem{Theorem}{Theorem}[section]
\numberwithin{Theorem}{section}
\newtheorem{Proposition}[Theorem]{Proposition}

\newtheorem{Corollary}[Theorem]{Corollary}
\theoremstyle{definition}
\newtheorem{Definition}[Theorem]{Definition}
\newtheorem{Remark}[Theorem]{Remark}
\newtheorem{Example}[Theorem]{Example}

\numberwithin{equation}{section}

\newcommand{\N}{\mathbb{N}}
\newcommand{\Z}{\mathbb{Z}}

\newcommand{\R}{\mathbb{R}}

\newcommand{\C}{\mathbb{C}}

\newcommand\e{\mathrm{e}}
\newcommand\I{\mathrm{i}}
\newcommand\re{\operatorname{Re}}
\newcommand\im{\operatorname{Im}}

\newcommand{\rd}{\mathrm{d}}

\newcommand\dom{\mbox{\rm dom}}
\newcommand\ran{\mbox{\rm ran}}

\newcommand\cS{\mathcal S}
\newcommand\cT{\mathcal T}

\newcommand\cA{\mathcal A}
\newcommand\cB{\mathcal B}

\newcommand\cV{\mathcal V}
\newcommand\lbar\overline

\newcommand\eps\varepsilon
\renewcommand\epsilon\varepsilon
\renewcommand\rho\varrho
\newcommand\al\alpha
\newcommand\lm\lambda

\newcommand\ds\displaystyle

\newcommand\p\partial

\newcommand{\tolong}{\longrightarrow}

\newcommand{\s}{\stackrel{s}{\rightarrow}}
\newcommand{\w}{\stackrel{w}{\rightarrow}}

\newcommand{\dist}{\operatorname{dist}}

\newcommand{\beq}{\begin{equation}}
\newcommand{\eeq}{\end{equation}}
\newcommand{\be}{\begin{equation*}}
\newcommand{\ee}{\end{equation*}}
\newcommand{\bmat}{\begin{pmatrix}}
\newcommand{\emat}{\end{pmatrix}}

\author{Sabine B\"ogli}
\address[S.\ B\"ogli]{
Department of Mathematical Sciences,
Durham University,
Lower Mount\-joy, Stockton Road,
Durham DH1 3LE, UK}
\email{sabine.boegli@durham.ac.uk}

\author{Marco Marletta}
\address[M.\ Marletta]{
School of Mathematics,
Cardiff University,
21–-23 Senghennydd Road, Cardiff CF24 4AG, UK}
\email{MarlettaM@cardiff.ac.uk}

\thanks{
}

\usepackage[update,prepend]{epstopdf}

\title[Essential numerical ranges for linear operator pencils]{Essential numerical ranges for\\ linear operator pencils}

\begin{document}

\subjclass[2010]{47A12, 47A56, 47A58}

\keywords{Essential numerical range, numerical range, linear operator pencil, eigenvalue approximation, spectral pollution, Dirac operator, Schr\"{o}dinger operator}

\date{\today}

\begin{abstract}
We introduce concepts of essential numerical range for the linear operator pencil $\lm\mapsto A-\lm B$.  In contrast to the operator essential numerical range, the pencil essential numerical ranges are, in general, neither convex nor even connected.  The new concepts allow us to  describe the set of spectral pollution when approximating the operator pencil by projection and truncation methods.
Moreover, by transforming the operator eigenvalue problem $Tx=\lm x$ into the pencil problem $BTx=\lm Bx$ for suitable choices of $B$, we can obtain non-convex spectral enclosures for $T$ and,
in the study of truncation and projection methods, confine spectral pollution to smaller sets than with hitherto known concepts. We apply the results to various block operator matrices. In particular, Theorem \ref{thm.dirac} presents substantial improvements over previously known results for Dirac operators while Theorem \ref{thm.schroedinger} excludes spectral pollution for a class of
non-selfadjoint Schr\"{o}dinger operators which it has not been possible to treat with existing methods.
\end{abstract}

\maketitle

\section{Introduction}
One of the simplest concepts which can be used to obtain an enclosure of the spectrum of a linear operator $T$ in a Hilbert space $H$ is the {\em numerical range}:
\[ W(T) = \{ \langle Tx,x\rangle: \; x\in \dom(T), \; \| x \| = 1\}. \]
Many simple estimates of eigenvalues of differential operators, for instance, involve calculating estimates of the inner products $\langle Tx,x\rangle$,
using partial integration. The main disadvantage of $W(T)$ is its convexity, which means that $W(T)$ cannot reveal the existence of spectral gaps.

If $T$ is bounded and if one wishes to enclose only the essential spectrum of $T$, then the concept of essential numerical range $W_e(T)$ introduced by  Stampfli and Williams 
\cite{Stampfli-Williams} gives a useful refinement; see also \cite{fillmore} for a review. 
The latter gives five equivalent characterisations of the essential numerical range
for a bounded operator. For  closed, unbounded operators, we showed \cite{We} that these concepts are no longer equivalent; we settled on the singular-sequence
definition as the most useful one:
\begin{equation} W_e(T)  = \left\{ \lim_{n\rightarrow\infty}\left\langle Tx_n,x_n\right\rangle: \; x_n\in \dom(T), \; \| x_n \| = 1, \; x_n\stackrel{w}{\to}0 \right\}, \label{eq:We3.1} \end{equation}
and proved that also in the unbounded case
\beq\label{eq.We3} W_e(T)=\bigcap_{K \text{ compact}}\overline{W(T+K)}.\eeq
From~\eqref{eq.We3} it is evident that $W_e(T)$ is a closed and convex set, and we proved that it
 consists precisely of the essential spectrum of $T$  together with all possible spectral pollution which may arise by applying projection methods to 
find the spectrum of $T$ numerically. This generalises a result of Levitin and Shargorodsky \cite{levitin} for the selfadjoint case, because then our essential numerical range 
coincides with the convex hull of their {\em extended essential spectrum}.

In this paper we turn  to linear pencils $\lm \mapsto A-\lm B$, where $A$ and $B$ are operators in $H$ and $\dom(A)\subseteq\dom(B)$. There are obvious
motivations for studying pencils directly since they arise naturally in so many application areas.  However there can also be advantages in considering the reformulation of
operator problems as pencil problems. Given an operator $T$, one may consider a pencil $\lm\mapsto BT-\lm B$, in which $B$ is a suitably chosen bounded operator.
This can be regarded as an abstract generalisation of several different tricks: the multiplier trick developed by Morawetz for scattering problems \cite{Morawetz-1961};
the techniques used in the derivation of many virial theorems (see, e.g., \cite{eastham-kalf}); or the method of Descloux \cite{descloux} which takes scalar products with
respect to different bilinear forms. The success of our approach depends on being
able to replace the numerical range and essential numerical range $W(T)$ and $W_e(T)$, whose convexity may be inconvenient, by suitable  concepts of numerical range  
and essential numerical range for a pencil, whose properties should be systematically studied. Section \ref{sec.def} is devoted to these topics; the reward is reaped in Sections 
\ref{sec.approx} and \ref{sec.op}. We particularly draw the reader's attention to Theorem \ref{thm.intersection}, which shows that the abstract Morawetz trick can, in principle, 
locate the approximate point spectrum exactly; Theorem \ref{thm.schroedinger}, which establishes lack of spectral pollution for a wide class of non-selfadjoint Schr\"{o}dinger 
operators; and Theorem \ref{thm.dirac}, which substantially improves existing results for Dirac operators. 

For the operator pencil $\lm\mapsto A-\lm B$ a numerical range concept, called \emph{root domain}, was defined in \cite[Section~26]{markus} as the set of all $\lm\in\C$ such 
that  $0$ belongs to the usual operator numerical range $W(A-\lm B)$: we expand this slightly to allow all $\lambda$ such that $0\in \overline{W(A-\lm B)}$ and denote this 
set by $W(A,B)$. We also introduce a second concept of pencil numerical range, denoted $w(A,B)$: see Definition \ref{def.W} below. There are two corresponding 
concepts of essential numerical range of the pencil, denoted by $W_e(A,B)$ and $w_e(A,B)$. Our slight modification of the definition of the pencil
numerical range in  \cite[Section~26]{markus} ensures that $W_e(A,B)\subseteq W(A,B)$. 

In Section~\ref{sec.def} we study properties of, and relations between, the numerical ranges $W(A,B)$, $w(A,B)$ and the essential numerical ranges $W_e(A,B)$, $w_e(A,B)$. In the special case that $B$ is uniformly positive, we have $w(A,B)=W(B^{-\frac{1}{2}}AB^{-\frac{1}{2}})$, $w_e(A,B)=W_e(B^{-\frac{1}{2}}AB^{-\frac{1}{2}})$ and hence the sets are convex. In general, however, the pencil notions are not convex (not even connected). We establish perturbation results for $W_e(A,B)$, $w_e(A,B)$ in which we add an operator $K$ to either $A$ or $B$.
Section~\ref{sec.approx} contains spectral convergence results. We approximate both $A$ and $B$ by  
projection or domain truncation methods and confine the possible spurious eigenvalues to $w_e(A,B)$ or $W_e(A,B)$.
We apply our results to an indefinite Sturm-Liouville operator in $L^2(\R)$, previously studied in~\cite{SL1,SL2}.
In the final Section~\ref{sec.op} we transform the operator eigenvalue problem $Tx=\lm x$ into the pencil eigenvalue problem $BTx=\lm Bx$ for an arbitrary bounded operator $B$, i.e.\ we study the linear pencil $\lm\mapsto BT-\lm B$.  Whereas the operator numerical range $W(T)$ is convex, the pencil analogue $W(BT,B)$ need not be convex or even connected. It is this fact which is
responsible for allowing us to get tighter spectral enclosures by taking the intersection of $W(BT,B)$ over suitable $B$, see Theorem \ref{thm.intersection}. 
Analogously, the set of possible spectral pollution is reduced to $W_e(BT,B)$ if we approximate $BT$ and $B$ instead of $T$.
The latter is particularly effective if $T$ is a differential operator and $B$ is (the operator of multiplication with) a bounded and boundedly invertible function; then the multiplication with $B$ commutes with domain truncation.  Another important application is to $2\times 2$ block operator matrices $T$ that we multiply by $2\times 2$ matrices $B$. We compare the resulting spectral enclosures with the quadratic numerical range (see~\cite{tretter}). The theoretical results of this section are applied to Schr\"odinger, Dirac, Stokes-type and Hain-L\"ust-type operators.

We use the following notion and conventions.
The notations $\|\cdot\|$ and $\langle\cdot,\cdot\rangle$ refer to the norm and scalar product of the Hilbert space $H$.
Strong and weak convergence of elements in $H$ is denoted by $x_n\to x$ and $x_n\stackrel{w}{\to}x$, respectively.
The space $L(H)$ contains all bounded linear operators in $H$, and $C(H)$  denotes the space of all closed linear operators in $H$.
Norm and strong operator convergence in $L(H)$ is denoted by $T_n\to T$ and $T_n\s T$, respectively.
An identity operator is denoted by~$I$; scalar multiples $\lm I$ are written as $\lm$.
Analogously, the operator of multiplication with a function $V$ is again $V$.
For two operators $T$, $S$ in $H$ we say that $S$ is $T$-form bounded if the respective quadratic forms are relatively bounded, i.e.\ if there exist $\alpha',\beta'\geq 0$ such that
\beq\label{eq.formbdd}\forall\,x\in\dom(T):\quad |\langle S x,x\rangle|\leq \alpha'\|x\|^2+\beta'|\langle T x,x\rangle|.\eeq
The infimum $\beta$ of all $\beta'\geq 0$ such that there exists $\alpha'\geq 0$ satisfying~\eqref{eq.formbdd} is called the relative form bound.
The domain, range, spectrum, point spectrum, approximate point spectrum and resolvent set
of an operator $T$ are denoted by $\dom(T)$, $\ran(T)$, $\sigma(T)$, $\sigma_p(T)$, $\sigma_{\rm app}(T)$ and $\rho(T)$,
respectively, and the Hilbert space adjoint operator of $T$ is~$T^*$.
For non-selfadjoint operators there exist (at least) five different definitions for the essential spectrum which all coincide in the selfadjoint case; for a discussion see~\cite[Chapter~IX]{edmundsevans}.
Here we use
$$\sigma_e(T):=\left\{\lm\in\C:\,\exists\, (x_n)_{n\in\N}\subset\dom(T) \text{ with }\|x_n\|=1,\,x_n\stackrel{w}{\to}0,\,\|(T-\lm)x_n\|\to 0\right\},$$
which corresponds to $k=2$ in~\cite{edmundsevans}. 
For an introduction to (polynomial) operator pencils we refer to the monograph~\cite{markus}.
For the linear pencil $\lm\mapsto A-\lm B$ the spectrum is $\sigma(A,B):=\{\lm\in\C:\,0\in\sigma(A-\lm B)\}$, and $\sigma_p(A,B)$, $\sigma_{\rm app}(A,B)$, $\sigma_e(A,B)$ and $\rho(A,B)$ are defined analogously. 
Following Kato (see~\cite[Section V.3.10]{kato}), we call a linear operator $T$ in $H$ sectorial if $W(T)\subseteq\{\lm\in\C:\,|\arg(\lm-\gamma)|\leq\theta\}$ with sectoriality semi-angle $\theta\in [0,\pi/2)$ and sectoriality vertex $\gamma\in\R$.
A subspace $\Phi\subset \dom(T)$ is called a core of a closable operator $T$ if $T|_{\Phi}$ is closable with closure $\overline{T}$.
For a subset $\Omega\subset\C$ we denote its interior by ${\rm int}\,\Omega$, its convex hull by ${\rm conv}\,\Omega$, its complex conjugated set by $\Omega^*:=\{\overline{z}:\,z\in\Omega\}$, and the distance of $z\in\C$ to $\Omega$ is ${\rm dist}(z,\Omega):=\inf_{w\in \Omega}|z-w|$. Finally, $B_{r}(\lm):=\{z\in\C:\,|z-\lm|<r\}$ is the open disk of radius $r$ around $\lm\in\C$.

\section{Definitions and properties}\label{sec.def}
In this section we define numerical ranges and essential numerical ranges of the pencil $\lm\mapsto A-\lm B$ in two ways that turn out to be non-equivalent in general. We establish sufficient conditions under which they coincide and study further equivalent characterisations. The section finishes with perturbation results for pencil essential numerical ranges.

\subsection{Basic properties}

Let $A,B$ be linear operators in $H$ with $\dom(A)\subseteq\dom(B)$.
We define two (generally different) numerical ranges of the pencil $\lm\mapsto A-\lm B$.
\begin{Definition}[Numerical ranges for a pencil]\label{def.W}
We define the sets
\begin{align*}
W(A,B)&:=\big\{\lm\in\C:\,0\in \overline{W(A-\lm B)}\big\},\\
w(A,B)&:=\left\{\frac{\langle A x,x\rangle}{\langle B x,x\rangle}:\,x\in\dom(A),\,\langle Bx,x\rangle\neq 0\right\}.
\end{align*}
\end{Definition}

\begin{Remark}\label{rem.num.range}
\begin{enumerate}
\item[\rm i)]
It follows immediately that
\begin{alignat*}{5}
W(zA,B)&=z W(A,B), \quad &w(zA,B)&=z w(A,B), \quad &&z\in\C,\\
W(A,zB)&=\frac{1}{z}W(A,B), \quad &w(A,zB)&=\frac{1}{z}w(A,B), \quad &&z\in\C\backslash\{0\},
\end{alignat*}
and
that, for any $\lm\neq 0$,
 \begin{align*}
\lm\in W(A,B)\quad  & \Longleftrightarrow\quad\lm^{-1}\in W(B|_{\dom(A)},A), \\
\lm\in w(A,B)\quad  & \Longleftrightarrow\quad \lm^{-1}\in w(B|_{\dom(A)},A). 
\end{align*}
Note that $0\in W(A,B)$ if and only if $0\in\overline{W(A)}$, and $0\in w(A,B)$ implies $0\in W(A)$.

\item[\rm ii)]
Clearly, we have the spectral enclosure
$\sigma_{\rm app}(A,B)\subseteq W(A,B)$. 
For an example with $\sigma_{\rm app}(A,B)\not\subseteq w(A,B)$, let $A=B=0$; then $\sigma_{\rm app}(A,B)=\sigma_p(A,B)=\C$ but $w(A,B)=\emptyset$.

\item[\rm iii)]
If $B$ is bounded, then $W(A,B)$ is closed. This is not true in the unbounded case (see Example~\ref{ex.unifposB}).
The set $w(A,B)$ need not be closed even if $B$ is bounded; as an example, let $B=I$, then $w(A,B)=W(A)$ which is not closed in general.
\end{enumerate}
\end{Remark}

\begin{Proposition}\label{prop.num.range}
\begin{enumerate}
\item[{\rm i)}]
We have $$w(A,B)\subseteq W(A,B).$$
Moreover, 
if, in addition, 
$$0\notin \overline{W(A)}\cap \overline{W(B)}\quad\text{or} \quad W(A,B)\neq \C,$$
then $W(A,B)\subseteq \overline{w(A,B)}$.
\item[{\rm ii)}]
If $B$ is uniformly positive, then 
\begin{align*}
w(A,B)&=W\big(B^{-\frac{1}{2}}AB^{-\frac{1}{2}}\big),\\
\overline{W(A,B)}&=\overline{w(A,B)}=\overline{W\big(B^{-\frac{1}{2}}AB^{-\frac{1}{2}}\big)},
\end{align*}
and all sets are convex.
\end{enumerate}
\end{Proposition}

\begin{proof}
i)
Let $\lm\in w(A,B)$. By definition, there exists $x\in\dom(A)$ with $\lm=\langle A x,x\rangle/\langle B x,x\rangle$; without loss of generality $\|x\|=1$.
Then $0=\langle (A-\lm B)x,x\rangle \in W(A-\lm B)$ and hence $\lm\in W(A,B)$.

Now assume that there exists $\lm\in W(A,B)\backslash\overline{w(A,B)}$.
Then there is a sequence $(x_n)_{n\in\N}\subset \dom(A)$ with $\|x_n\|=1$ and $\langle (A-\lm B)x_n,x_n\rangle\to 0$.
If there exist $n_0\in\N$ and $c>0$ such that $|\langle B x_n,x_n\rangle|\geq c$ for all $n\geq n_0$, then 
$$\lm=\lim_{n\to\infty}\frac{\langle A x_n,x_n\rangle}{\langle Bx_n,x_n\rangle}\in \overline{w(A,B)},$$
a contradiction. Hence, at least on a subsequence, we have $\langle B x_n,x_n\rangle\to 0$ and thus also $\langle A x_n,x_n\rangle \to 0$. This implies $$0\in \overline{W(A)}\cap \overline{W(B)}, \quad W(A,B)=\C.$$

ii)
The first identity is a direct consequence of
$$\frac{\langle Ax,x\rangle}{\langle Bx,x\rangle}=\frac{\langle B^{-\frac{1}{2}}AB^{-\frac{1}{2}}y,y\rangle}{\|y\|^2}$$
with the one-to-one correspondence $x=B^{-\frac{1}{2}}y$ for $x\in\dom(A)$ and $y\in \dom\big(AB^{-\frac{1}{2}}\big)$.
By the assumed uniform positivity of $B$, we conclude $0\notin\overline{W(B)}$. Now claim~i) implies $\overline{W(A,B)}=\overline{w(A,B)}$.
The convexity of all sets follows from the convexity of $W\big(B^{-\frac{1}{2}}AB^{-\frac{1}{2}}\big)$.
\end{proof}

\begin{Example}\label{ex.unifposB}
In $l^2(\N)$ consider the selfadjoint operators whose representations with respect to the standard orthonormal basis of $l^2(\N)$ are diagonal:
$$A:={\rm diag}(n^2+n:\,n\in\N), \quad B:={\rm diag}(n^2:\,n\in\N).$$
Evidently $B\geq I$. Using
$$B^{-\frac{1}{2}}AB^{-\frac{1}{2}}={\rm diag}\Big(1+\frac{1}{n}:\,n\in\N\Big),$$
Proposition~\ref{prop.num.range}~ii) yields
$$w(A,B)=W\big(B^{-\frac{1}{2}}AB^{-\frac{1}{2}}\big)=(1,2].$$
Since $0\notin \overline{W(A-B)}$, we have $1\notin W(A,B)$ and hence Proposition~\ref{prop.num.range}~i) implies $W(A,B)=w(A,B)$.
Note that the numerical ranges are not closed.
\end{Example}

The following result generalises~\cite[Theorems~26.6,~26.7]{markus} for bounded pencils and also~\cite[Theorem~V.3.2]{kato} for operators (i.e.\ for $B=I$).

\begin{Theorem}\label{thm.dist.pencil}
Let $B\in L(H)$ satisfy $0\notin \overline{W(B)}$. Let $\Omega\subseteq \C\backslash W(A,B)$ be a connected set with $\Omega\cap\rho(A,B)\neq\emptyset$.
Then $\Omega\subseteq \rho(A,B)$, $W(A,B)=\overline{w(A,B)}$ and
\beq\label{eq.resnormpencil}\|(A-\lm B)^{-1}\|\leq \frac{1}{{\rm dist}(0,W(B))\,{\rm dist}(\lm,W(A,B))}, \quad \lm\in\Omega.\eeq
\end{Theorem}

\begin{proof}
The assumptions on $B$ imply, by Remark~\ref{rem.num.range}~iii) and Proposition~\ref{prop.num.range}~i), that $W(A,B)=\overline{w(A,B)}$.
Since $B$ is assumed to be bounded, by~\cite[Theorem~IV.5.17]{kato}, we conclude that $\lm\mapsto {\rm ind}(A-\lm B)$ is constant on every connected component of $\C\backslash\sigma_{\rm app}(A,B)$. 
Since $\sigma_{\rm app}(A,B)\subseteq W(A,B)$ and $\Omega\cap\rho(A,B)\neq\emptyset$, we obtain $\Omega\subseteq \rho(A,B)$.
Now let $\lm\in\Omega$. 
Then, for all $x\in\dom(A)$ with $\|x\|=1$,
\begin{align*}
\|(A-\lm B)x\|
&\geq |\langle (A-\lm B)x,x\rangle|= |\langle Bx,x\rangle|\left|\frac{\langle A x,x\rangle}{\langle B x,x\rangle}-\lm\right|\\
&\geq  {\rm dist}(0,W(B)){\rm dist}(\lm,w(A,B)),
\end{align*}
which, together with $\overline{w(A,B)}=W(A,B)$, proves~\eqref{eq.resnormpencil}.
\end{proof}

As for the numerical ranges of the pencil $\lm\mapsto A-\lm B$, we can also define two concepts of essential numerical range. The first of these, 
$W_e(A,B)$ below, involves the operator essential numerical range from equation \eqref{eq:We3.1}; the second, $w_e(A,B)$, is generally not equivalent to the first. We shall study the relationship between
the two in several propositions and examples.
\begin{Definition}[Essential numerical ranges for a pencil]\label{def.We}
We define the sets
\begin{align*}
W_e(A,B)&:=\big\{\lm\in\C:\,0\in W_e(A-\lm B)\big\},\\
w_e(A,B)&:=\left\{\lim_{n\to\infty}\frac{\langle A x_n,x_n\rangle}{\langle B x_n,x_n\rangle}:\,x_n\in\dom(A),\,\langle Bx_n,x_n\rangle\neq 0,\,\,\|x_n\|=1,\,x_n\stackrel{w}{\to}0\right\}.
\end{align*}
\end{Definition}

\begin{Remark}\label{rem.inverse}
\begin{enumerate}
\item[\rm i)]
It follows immediately that
\begin{alignat*}{5}
W_e(zA,B)&=z W_e(A,B), \quad &w_e(zA,B)&=z w_e(A,B), \quad &&z\in\C,\\
W_e(A,zB)&=\frac{1}{z}W_e(A,B), \quad &w_e(A,zB)&=\frac{1}{z}w_e(A,B), \quad &&z\in\C\backslash\{0\},
\end{alignat*}
and
that, for any $\lm\neq 0$,
\begin{align*}
\lm\in W_e(A,B)\quad &\Longleftrightarrow\quad \lm^{-1}\in W_e(B|_{\dom(A)},A),\\
\lm\in w_e(A,B)\quad &\Longleftrightarrow\quad \lm^{-1}\in w_e(B|_{\dom(A)},A).
\end{align*}
Note that $0\in W_e(A,B)$ if and only if $0\in W_e(A)$.

\item[\rm ii)]
Clearly, we have the spectral enclosure
$\sigma_e(A,B)\subseteq W_e(A,B)$. 
For an example with $\sigma_e(A,B)\not\subseteq w_e(A,B)$, let $A=B=0$; then $\sigma_e(A,B)=\C$ but $w_e(A,B)=\emptyset$.

\item[\rm iii)]
By a standard diagonal sequence argument, the set $w_e(A,B)$ is closed. 
If $B$ is bounded, then $W_e(A,B)$ is closed as well. This is not true in the unbounded case (see Example~\ref{ex.notclosed}).

\item[\rm iv)]
For the operator essential numerical range ($B=I$) it was shown in \cite[Corollary 2.5 iv)]{We} that $W_e(A)=\C$ if and only if $W(A)=\C$.
This is no longer true for general $B$. As a first counterexample, consider in $l^2(\N_0)$ the diagonal operators 
$$A:={\rm diag}(n:\,n\in\N_0),\quad B:={\rm diag}(b_n:\,n\in\N_0)$$
 with $b_0=0$ and $(b_n)_{n\in\N}\subset \C$ a bounded sequence. 
Then $W(A,B)=\C$; however $W_e(A,B)=\emptyset$ by Proposition \ref{prop.We.empty} below and the fact that $W_e(A)=\emptyset$,
see Theorem~\ref{Thm2.11}.

As a second counterexample, consider in $l^2(\N)$ the diagonal operators 
$$A_{1,1}:={\rm diag}(n:\,n\in\N),\quad B_{1,1}:={\rm diag}(\I^n:\,n\in\N).$$
Using the uniform positivity of $A_{1,1}$ we see that $\lambda\in w(A_{1,1},B_{1,1})$ if and only if
$\lambda^{-1}\in w(B_{1,1}|_{\dom(A_{1,1})},A_{1,1})$, and 
\[ w(B_{1,1}|_{\dom(A_{1,1})},A_{1,1}) = w(A_{1,1}^{-\frac{1}{2}}B_{1,1}A_{1,1}^{-\frac{1}{2}}) = \mbox{conv}\left(\{ \I^n/n \: | \: n\in {\mathbb N}\}\right).
 \]
In particular, $w(A_{1,1},B_{1,1})$ is of the form ${\mathbb C}\setminus N$ where $N$ is a bounded neighbourhood of zero.
Now choose a $2\times 2$ matrix $A_{2,2}$ such that $w(A_{2,2},I_{2\times 2}) = W(A_{2,2})$ contains $N$, and define
\[ A = \left(\begin{array}{cc} A_{1,1} & 0 \\ 0 & A_{2,2}\end{array}\right), \;\;\; B = \left(\begin{array}{cc}B_{1,1} & 0 \\ 0 & I_{2\times 2}\end{array}\right). \]
Then $w(A,B)=\C$. However $w_e(A,B)=w_e(A_{1,1},B_{1,1})$ by a direct calculation from the definitions. Furthermore, 
Theorem~\ref{Thm2.11} below implies that $W_e(A_{1,1})=\emptyset$; then since $B_{1,1}$ is bounded, Proposition \ref{prop.We.empty}
implies that $w_e(A_{1,1},B_{1,1})=\emptyset$. Hence $w_e(A,B)=\emptyset$.

\end{enumerate}
\end{Remark}

\begin{Proposition}\label{prop.ess.num.range}
\begin{enumerate}
\item[\rm i)] If
\begin{equation}\label{eq.0inWe}
0\notin W_e(A)\cap W_e(B)\quad\text{or} \quad W_e(A,B)\neq \C,
\end{equation}
then $\overline{W_e(A,B)}\subseteq w_e(A,B)$.

\item[\rm ii)]
If $B$ is bounded, then 
\beq \label{eq.Bbdd} w_e(A,B)\subseteq W_e(A,B).\eeq
If, in addition, \eqref{eq.0inWe} holds, then equality prevails in~\eqref{eq.Bbdd} and the sets are closed.
\end{enumerate}
\end{Proposition}

\begin{proof}
i)
Assume that~\eqref{eq.0inWe} is violated.
Then the proof of $W_e(A,B)\subseteq w_e(A,B)$ is analogous to the proof of the second part of Proposition~\ref{prop.num.range}~i); the only difference is that here we take the weak convergence $x_n\stackrel{w}{\to}0$ into account.
Now the claim follows from the closedness of $w_e(A,B)$, see Remark~\ref{rem.inverse}~iii).

ii)
Let $\lm\in w_e(A,B)$. By definition, there exists a sequence $(x_n)_{n\in\N}\subset\dom(A)$ with $\|x_n\|=1$, $x_n\stackrel{w}{\to}0$ and 
$$\langle Bx_n,x_n\rangle\neq 0, \quad \frac{\langle A x_n,x_n\rangle}{\langle B x_n,x_n\rangle}\tolong \lm, \quad n\to\infty.$$
Since $B$ is bounded, we obtain
\begin{align*}
|\langle (A-\lm B)x_n,x_n\rangle|
&\leq \left|\frac{\langle A x_n,x_n\rangle}{\langle B x_n,x_n\rangle}-\lm\,\right|\|B\|\tolong 0, \quad n\to\infty.
\end{align*}
Therefore, $0\in W_e(A-\lm B)$ and hence $\lm\in W_e(A,B)$.

The rest of the claim follows from claim~i).
\end{proof}

\begin{Remark}
The inequality~\eqref{eq.Bbdd} may be strict (in which case~\eqref{eq.0inWe} is violated). As an example, let $A=B$ be compact. Then $w_e(A,B)\subseteq \{1\}$ whereas $W_e(A,B)=\C$.
\end{Remark}

Next we illustrate Proposition~\ref{prop.ess.num.range}~i) 
when $W_e(A,B)$ is not closed.

\begin{Example}\label{ex.notclosed}
In $l^2(\N)$ consider the operators 
$$A:={\rm diag}\{(-1)^n n^4+\I n:\,n\in\N\}, \quad B:={\rm diag}\{n^3+\I\,(-1)^n n^2:\,n\in\N\}.$$
Let $\lm\in\R\backslash\{0\}$. Then it may be shown that $0\in W_e(A-\lm B) = {\mathbb C}$
 (see Remark~\ref{rem.inverse}~iv)) 
and hence $\lm\in W_e(A,B)$.
However, $0\notin W_e(A,B)$ since $0\notin W_e(A)=\emptyset$.
One may check that $w_e(A,B)=\R$. By Proposition~\ref{prop.ess.num.range}~i), we obtain $W_e(A,B)\subseteq w_e(A,B)$ and thus $W_e(A,B)=\R\backslash\{0\}$.
 Therefore, $W_e(A,B)$ is neither closed nor convex; it is not even connected.
\end{Example}

In the latter example the operator $B$ was $A$-bounded but not $A$-form bounded.
In the next result we consider form bounded operators.

\begin{Proposition}\label{prop.We.empty}
Assume that $W_e(A)=\emptyset$ and $B$ is $A$-form bounded with relative form bound $\beta$. Then
$$\overline{W_e(A,B)}\subseteq w_e(A,B)\subseteq\{\lm\in\C:\,|\lm|\geq \beta^{-1}\};$$
if $\beta=0$, then
$$W_e(A,B)=w_e(A,B)=\emptyset.$$
\end{Proposition}

\begin{proof}
The assumption $W_e(A)=\emptyset$ and Proposition~\ref{prop.ess.num.range}~i) imply $\overline{W_e(A,B)}\subseteq w_e(A,B)$.
Let $(x_n)_{n\in\N}\subset\dom(A)$ satisfy $\|x_n\|=1$ and $x_n\stackrel{w}{\to}0$. Fix $\eps>0$.
The relative form boundedness implies the existence of $\alpha_{\eps}\geq 0$ such that
$$|\langle B x_n,x_n\rangle|\leq \alpha_{\eps}+(\beta+\eps)|\langle A x_n,x_n\rangle|, \quad n\in\N.$$
The assumption $W_e(A)=\emptyset$ yields $|\langle A x_nx_n\rangle|\to\infty$.
Therefore, $$\limsup_{n\to\infty}\left|\frac{\langle Bx_n,x_n\rangle}{\langle A x_n,x_n\rangle}\right|\leq \beta+\eps.$$
Since $\eps>0$ was arbitrary, we arrive at $w_e(A,B)\subseteq \{\lm\in\C:\,|\lm|\geq \beta^{-1}\}$.
\end{proof}

\subsection{Equivalent characterisations and perturbation results for operators}
Before proceeding to the study of equivalent characterisations and perturbation result for pencils, for the convenience of the reader we review the corresponding 
properties for operators (the case $B=I$). The material in this section is a summary of results from \cite{We}.
\begin{Theorem}\cite[Theorem 3.1]{We}\label{thmequivdefforWe}
Let $\cV$ be the set of all finite-dimensional subspaces $V\subset H$. Define
   \begin{align*}
    W_{e1}(A)&:=\underset{V \in\cV}{\bigcap}\overline{W(A|_{V^{\perp}\cap\dom(A)})},\\
    W_{e2}(A)&:=\underset{K\in L(H)\atop {\rm rank }\,K<\infty}{\bigcap}\overline{W(A+K)},\\
    W_{e3}(A)&:=\!\underset{K\in L(H)\atop K \text{ compact}}{\bigcap}\!\overline{W(A+K)},\\
    W_{e4}(A)&:=\big\{\lm\in\C:\,\exists\,(e_n)_{n\in\N}\subset\dom(A)\text{ orthonormal with }
		                             \langle Ae_n,e_n\rangle\stackrel{n\to\infty}{\longrightarrow}\lm\big\}. \vspace{1mm}   \end{align*}
Then, in general,
\begin{align}
\label{Weinclgeneral}
&W_{e1}(A)\subseteq W_{e4}(A) \subseteq W_{e2}(A)= W_{e3}(A)=W_{\!e}(A).\\
\intertext{%
If $\,\overline{\dom(A)}=H$, then}
\label{Weincldd}
&W_{e1}(A)\subseteq W_{e4}(A)= W_{e2}(A)= W_{e3}(A)= W_{\!e}(A).
\end{align}%
If $\,\overline{\dom(A)\cap\dom(A^*)}=H$ or if $\,W(A)\neq\C$, then
\begin{equation}
\label{Weinclspecial}
W_{\!e}(A)=W_{ei}(A), \quad i=1,2,3,4.
\end{equation}
\end{Theorem}
We remark in particular the equivalence $W_{e2}(A)=W_{e3}(A)=W_{e}(A)$ in all cases. The fact that the other inclusions may be strict is shown
by examples in \cite{We}.

For the case of selfadjoint operators, $W_e$ can be found from the extended essential spectrum of Levitin and Shargorodsky.
\begin{Theorem} \cite[Theorem 3.8]{We}\label{Thm2.11} 
If $A=A^*$ is bounded define $\widehat{\sigma}_e(A) = \sigma_e(A)$; 
otherwise if $A=A^*$ is
unbounded let $\widehat{\sigma}_e(A)$ be $\sigma_e(A)$ with $+\infty$ and/or $-\infty$ added if $A$ is unbounded from above and/or from below. 
 Then
\[ W_e(A) = \mbox{\rm conv}(\widehat{\sigma}_e(A))\setminus \{-\infty,+\infty\}. \]
\end{Theorem}

The final part of our review consists of perturbation results which we shall need later. In general, the essential numerical range is not invariant under
perturbations which are only relatively compact, but not compact. The following results give additional hypotheses under which relative compactness
is sufficient for invariance.
\begin{Theorem}\cite[Theorem 4.5]{We}\label{thmABUV}
Let $T=A+\I B$ and $S=U+\I V$ with symmetric operators $A$,~$B$ and $U$, $V$ in $H$ such that
\emph{one} of the following holds:
\begin{enumerate}
 \item[\rm (i)] \hspace{0.5mm}$A$ is selfadjoint and semibounded, $U$, $V$ are $A$-compact, or
 \item[\rm (ii)] $B$ is selfadjoint and semibounded, $U$, $V$ are $B$-compact, or
 \item[\rm (iii)] $A$, $B$ are selfadjoint and semibounded, $U$ is $A$-compact and $V$ is $B$-compact.
\end{enumerate}
\noindent Then $W_{\!e}(T)=W_e(T+S)$.
\end{Theorem}
\begin{Theorem}\cite[Theorem 4.7]{We}\label{thmSqrt}
Let $T=A+\I B$ with uniformly positive $A$ and symmetric $B$ and let $A^{-1/2}S$ be $A^{1/2}$-compact, i.e.\ $A^{-1/2}SA^{-1/2}$ is compact. Then $W_{\!e}(T)=W_e(T+S)$.
In particular, if $S$ is $A$-compact and $\dom(A)\!\subset\!\dom(S) \cap \dom(S^*)$, then $W_{\!e}(T)=W_e(T+S)$.
\end{Theorem}

\subsection{Equivalent characterisations for pencils}

\begin{Theorem}\label{thm.proj}
Let $\cV$ be the set of all finite-dimensional subspaces $V\subset H$.
Assume that $A$ is densely defined
and, for every $\lm\in\C$, 
$$\overline{\dom(A-\lm B)\cap\dom((A-\lm B)^*)} = H \quad \text{or} \quad W(A-\lm B)\neq \C.$$
Then 
\begin{align*}
W_e(A,B)
&=\underset{V\in\mathcal V}{\bigcap}W(A|_{V^{\perp}\cap\dom(A)},B|_{V^{\perp}\cap\dom(A)}).
\end{align*}
\end{Theorem}

\begin{proof}
Let $\lm\in\C$.
By Theorem \ref{thmequivdefforWe} for operators we have, 
under the hypothesis $\overline{\dom(A-\lm B)\cap\dom((A-\lm B)^*)} = H$ or $W(A-\lm B)\neq\C$,
$$W_e(A-\lm B)=W_{e1}(A-\lm B)=\underset{V\in\mathcal V}{\bigcap}\overline{W\big((A-\lm B)|_{V^{\perp}\cap\dom(A)}\big)}.$$
Therefore,
\be
\begin{aligned}
W_e(A,B)
&=\bigg\{\lm\in\C:\,0\in \underset{V\in\mathcal V}{\bigcap}\overline{W\big((A-\lm B)|_{V^{\perp}\cap\dom(A)}\big)}\bigg\}\\
&= \underset{V\in\mathcal V}{\bigcap}\Big\{\lm\in\C:\,0\in \overline{W\big((A-\lm B)|_{V^{\perp}\cap\dom(A)}\big)}\Big\}\\
&=\underset{V\in\mathcal V}{\bigcap}W(A|_{V^{\perp}\cap\dom(A)},B|_{V^{\perp}\cap\dom(A)}).\qedhere
\end{aligned}
\ee
\end{proof}
\smallskip

\begin{Theorem}\label{thm.WeAB}
We have
$$
W_e(A,B)= \underset{K \text{ compact}}{\bigcap} W(A+K,B)=\underset{K\in L(H)\atop {\rm rank}\,K<\infty}{\bigcap} W(A+K,B).
$$
\end{Theorem}

\begin{proof}
Let $\lm\in\C$.
By Theorem \ref{thmequivdefforWe} we obtain
$$
W_e(A-\lm B)= W_{e3}(A-\lm B)=\underset{K \text{ compact}}{\bigcap}\overline{W(A-\lm B+K)}.
$$
The latter implies
\begin{align*}
W_e(A,B)
&= \bigg\{\lm\in\C:\,0\in \underset{K \text{ compact}}{\bigcap}\overline{W(A+K-\lm B)}\bigg\}
=\underset{K \text{ compact}}{\bigcap} W(A+K,B).
\end{align*}
The claim for finite rank operators is obtained analogously using  
\be W_{e3}(A-\lm B)=W_{e2}(A-\lm B)=\underset{K\in L(H)\atop \text{rank } K<\infty}{\bigcap}\overline{W(A-\lm B+K)}.\qedhere\ee
\end{proof}

\begin{Remark}
\begin{enumerate}
\item[\rm i)]
From Theorem~\ref{thm.WeAB} it follows immediately that for every compact or finite rank $K\in L(H)$, we have $W_e(A+K,B)=W_e(A,B)$. 

\item[\rm ii)]
In general, $w_e(A,B)$ is not even contained in the intersection of $\overline{w(A+K,B)}$ over all compact or finite rank operators $K\in L(H)$.
As an example, let $A=B$ be compact but not of finite rank. Then 
$$w_e(A,B)=\{1\}, \quad  \underset{K \text{ compact}}{\bigcap} \overline{w(A+K,B)}=\emptyset.$$
Note that in this example $w_e(A,B)$ is not invariant under compact perturbations. In fact, for $K=\alpha A$ with $\alpha\in\C\backslash\{0\}$, we have $$w_e(A+K,B)=(1+\alpha)w_e(A,B)=\{1+\alpha\}\neq w_e(A,B).$$
\end{enumerate}
\end{Remark}

\begin{Theorem}\label{thm.We.we}
\begin{enumerate}
\item[\rm i)]
Let $0\notin W_e(B)$. Then
\beq\label{eq.comppert}
\begin{aligned}
w_e(A,B)
&\subseteq \underset{K \text{ compact}}{\bigcap} \overline{w(A+K,B)}\subseteq \underset{K\in L(H)\atop \text{rank } K<\infty}{\bigcap}\overline{w(A+K,B)}\\
&\subseteq \underset{K\in L(H)\atop \text{rank } K<\infty}{\bigcap}\overline{W(A+K,B)}
=\underset{K \text{ compact}}{\bigcap} \overline{W(A+K,B)}
\end{aligned}
\eeq
and
\beq\label{eq.closint}
\overline{W_e(A,B)}\subseteq w_e(A,B).\eeq

\item[\rm ii)]
Assume that $B$ is uniformly positive.
Then 
\beq\label{eq.sqrt}
w_e(A,B)=W_e\big(B^{-\frac 1 2}AB^{-\frac 1 2}\big),
\eeq
 the sets are closed and convex and they coincide with the four intersections in~\eqref{eq.comppert}.
\end{enumerate}
\end{Theorem}


\begin{proof}
i)
To prove the first inclusion in~\eqref{eq.comppert}, let $\lm\in w_e(A,B)$. 
Then there exists a sequence $(x_n)_{n\in\N}\subset\dom(A)$ with $\langle B x_n,x_n\rangle\neq 0$, $\|x_n\|=1$, $x_n\stackrel{w}{\to} 0$ and 
$$\frac{\langle A x_n,x_n\rangle}{\langle B x_n,x_n\rangle} \tolong \lm, \quad n\to\infty.$$
Let $K$ be compact. Then $\|x_n\|=1$ and $x_n\stackrel{w}{\to}0$ imply $\langle K x_n,x_n\rangle\to 0$.
By the assumption $0\notin W_e(B)$, there exist $c>0$ and $n_0\in\N$ such that $|\langle B x_n,x_n\rangle|\geq c$ for all $n\geq n_0$.
Hence 
$$\frac{\langle K x_n,x_n\rangle}{\langle B x_n,x_n\rangle}\tolong 0, \quad \frac{\langle (A+K)x_n,x_n\rangle}{\langle B x_n,x_n\rangle}\tolong \lm, \quad n\to\infty,$$
and therefore $\lm\in w_e(A+K,B)\subseteq \overline{w(A+K,B)}$. 

The second inclusion is evident since every bounded finite rank operator is compact.

Proposition~\ref{prop.num.range}~i) implies the third inclusion.

The equality  in~\eqref{eq.comppert} follow since every bounded finite rank operator is compact, and every compact operator is the norm limit of bounded finite rank operators.

The inclusion in~\eqref{eq.closint} follows from the assumption $0\notin W_e(B)$ and Proposition~\ref{prop.ess.num.range}~i).

ii)
First we prove
\beq\label{eq.sqrtBK}\underset{K \text{ compact}}{\bigcap}\overline{W\big(B^{-\frac{1}{2}}(A+K)B^{-\frac{1}{2}}\big)}= W_e\big(B^{-\frac{1}{2}}AB^{-\frac{1}{2}}\big);\eeq
then  the closed and convex set $W_e\big(B^{-\frac{1}{2}}AB^{-\frac{1}{2}}\big)$ coincides with the  four intersections in~\eqref{eq.comppert}
 since Proposition~\ref{prop.num.range}~ii) implies 
$$\overline{W(A+K,B)}=\overline{w(A+K,B)}=\overline{W\big(B^{-\frac{1}{2}}(A+K)B^{-\frac{1}{2}}\big)}.$$
In view of (\ref{eq.comppert}) this establishes the inclusion
\begin{equation}
w_e(A,B)\subseteq W_e(B^{-\frac{1}{2}}AB^{-\frac{1}{2}}). \label{ima1}
\end{equation}
%
%
%

To prove \eqref{eq.sqrtBK}, note first that by Theorem \ref{thmequivdefforWe},
\begin{equation}\label{eq.We.rank}
W_e\big(B^{-\frac{1}{2}}AB^{-\frac{1}{2}}\big)=W_{e2}\big(B^{-\frac{1}{2}}AB^{-\frac{1}{2}}\big)
=\underset{M\in L(H)\atop \text{rank }M<\infty}{\bigcap}\overline{W\big(B^{-\frac{1}{2}}AB^{-\frac{1}{2}}+M\big)}.
\end{equation}
Take some 
\begin{equation}\label{eq.lm.We.rank}
\lm\in \underset{K \text{ compact}}{\bigcap}\overline{W\big(B^{-\frac{1}{2}}(A+K)B^{-\frac{1}{2}}\big)}
\end{equation}
and let $M\in L(H)$ have finite rank. Recall that then also $M^*\in L(H)$ has finite rank.
We show that, for an arbitrary $\eps>0$, there exists $x_{\eps}\in\dom\big(AB^{-\frac{1}{2}}\big)$ with $\|x_{\eps}\|=1$
such that
\begin{equation}\label{eq.toshow.eps}
\big|\big\langle \big(B^{-\frac{1}{2}}AB^{-\frac{1}{2}}+M\big) x_{\eps},x_{\eps}\big\rangle-\lm\, \big|<\eps;
\end{equation}
then it is easy to see that $\lm$ belongs to the set in~\eqref{eq.We.rank}.

Since $B$ is selfadjoint, it is densely defined, and hence so is the operator $B^{\frac{1}{2}}$. 
Therefore there exists a sequence $(P_n)_{n\in\N}\subset L(H)$ of orthogonal projections of finite rank with $\ran(P_n)\subset\dom\big(B^{\frac{1}{2}}\big)$ and $P_n\stackrel{s}{\to} I$.
Since strong and uniform convergence coincide on a finite-dimensional space, there exists $n_{\eps}\in\N$ such that
\begin{equation}\label{eq.proj.eps}
\big\|(P_{n_{\eps}}-I)|_{\ran(M)}\big\|<\frac{\eps}{4\|M\|}, \quad \big\|(P_{n_{\eps}}-I)|_{\ran(M^*)}\big\|<\frac{\eps}{4\|M^*\|}.
\end{equation}
Define 
$$K_{\eps}:=B^{\frac{1}{2}}P_{n_{\eps}}M P_{n_{\eps}} B^{\frac{1}{2}}, \quad \dom(K_{\eps}):=\dom\big(B^{\frac{1}{2}}\big).$$
The operator is bounded since ${\rm rank}\,P_{n_{\eps}}<\infty$ implies that $B^{\frac{1}{2}}P_{n_{\eps}}$ and $P_{n_{\eps}} B^{\frac{1}{2}}\subseteq (B^{\frac{1}{2}}P_{n_{\eps}})^*$ are bounded.
Since $K_{\eps}$ is densely defined and of finite rank, it is closable and $\overline{K_{\eps}}\in L(H)$ is compact.
By~\eqref{eq.lm.We.rank}, there exists  $x_{\eps}\in\dom\big(AB^{-\frac{1}{2}}\big)$ with $\|x_{\eps}\|=1$
such that
$$\big|\big\langle \big(B^{-\frac{1}{2}}AB^{-\frac{1}{2}}+P_{n_{\eps}}MP_{n_{\eps}}\big) x_{\eps},x_{\eps}\big\rangle-\lm\, \big|=\big|\big\langle B^{-\frac{1}{2}}(A+\overline{K_{\eps}})B^{-\frac{1}{2}} x_{\eps},x_{\eps}\big\rangle-\lm\, \big|<\frac{\eps}{2}.$$
Now~\eqref{eq.toshow.eps} follows from the latter and because~\eqref{eq.proj.eps} implies
\begin{align*}
&\big|\langle P_{n_{\eps}}M P_{n_{\eps}} x_{\eps},x_{\eps}\rangle - \langle M x_{\eps},x_{\eps}\rangle \big|\\
&\leq 
\big|\langle \left(P_{n_{\eps}}-I)M+ P_{n_{\eps}}M(P_{n_{\eps}}-I)\right) x_{\eps},x_{\eps}\rangle \big|
\leq \|(P_{n_{\eps}}-I)M\|+ \|(M(P_{n_{\eps}}-I))^*\|\\
&\leq
\big\|(P_{n_{\eps}}-I)|_{\ran(M)}\big\| \|M\|+\big\|(P_{n_{\eps}}-I)|_{\ran(M^*)}\big\| \|M^*\|<\frac{\eps}{2}.
\end{align*}

The reverse inclusion (and thus equality in~\eqref{eq.sqrtBK}) follows since, by Theorem~\ref{thmequivdefforWe},
$$W_e\big(B^{-\frac{1}{2}}AB^{-\frac{1}{2}}\big)=W_{e3}\big(B^{-\frac{1}{2}}AB^{-\frac{1}{2}}\big)
=\underset{M \text{ compact}}{\bigcap}\overline{W\big(B^{-\frac{1}{2}}AB^{-\frac{1}{2}}+M\big)}$$
and $B^{-\frac{1}{2}}KB^{-\frac{1}{2}}$ is compact for every compact $K$.

%
%
We have already seen the inclusion $w_e(A,B)\subseteq W_e(B^{-\frac{1}{2}}AB^{-\frac{1}{2}})$ in (\ref{ima1}).
The reverse inclusion (and thus equality in~\eqref{eq.sqrt}) is shown in two steps. 

In the first step we assume
${\rm int}\,W_e\big(B^{-\frac{1}{2}}AB^{-\frac{1}{2}}\big)= \mathbb C$, so for every compact operator $K$ we have 
$W_e(B^{-\frac{1}{2}}(A+K)B^{-\frac{1}{2}})=\mathbb C$.
Hence $\overline{W(B^{-\frac{1}{2}}(A+K)B^{-\frac{1}{2}})}=\mathbb C$, so by Proposition   \ref{prop.num.range} ii) it follows that $\overline{w(A+K,B)}=\mathbb C$.
Thus, by convexity of $w(A+K,B)$, which is a consequence of the uniform positivity of $B$, we also have $w(A+K,B)=\mathbb C$. Proposition \ref{prop.num.range} i)
yields $W(A+K,B)=\mathbb C$ for all compact $K$, and hence 
\[ \underset{K \text{ compact}}{\bigcap} W(A+K,B) = \mathbb C.\]
By Theorem \ref{thm.WeAB}, it follows that $W_e(A,B)=\mathbb C$ and by (\ref{eq.closint}), we have $w_e(A,B)=\mathbb C$.


In the second step we assume that $W_e\big(B^{-\frac{1}{2}}AB^{-\frac{1}{2}}\big)$ is not equal to $\C$, and hence by Remark \ref{rem.inverse} iv),
$W\big(B^{-\frac{1}{2}}AB^{-\frac{1}{2}}\big) \neq \C$. This
allows us to invoke the last part of Theorem \ref{thmequivdefforWe} to assert that $W_{e}$ coincides with $W_{e1}$. Thus, letting
$\cV$ denote the set of all finite-dimensional subspaces $V\subset H$,
\begin{equation}\label{eq.char.proj2}
\begin{aligned}
W_e\big(B^{-\frac{1}{2}}AB^{-\frac{1}{2}}\big)
&=W_{e1}\big(B^{-\frac{1}{2}}AB^{-\frac{1}{2}}\big)\\
&=\bigcap_{V\in\mathcal V}\overline{W\Big(B^{-\frac{1}{2}}AB^{-\frac{1}{2}}\big|_{V^{\perp}\cap\dom\big(AB^{-\frac{1}{2}}\big)}\Big)}.
\end{aligned}
\end{equation}
Let $\lm\in W_e\big(B^{-\frac{1}{2}}AB^{-\frac{1}{2}}\big)$.
Let $e_0\in\dom(A)$ be arbitrary with $\|e_0\|=1$.
We inductively construct an orthonormal sequence $(e_n)_{n\in\N}\subset\dom(A)$ such that, for all $n\in\N$,
\begin{equation}\label{eq.lm.approx}
\left|\frac{\langle A e_n,e_n\rangle}{\langle B e_n,e_n\rangle}-\lm\right|<\frac{1}{n};
\end{equation}
then $\lm\in w_e(A,B)$.

Let $n_0\in\N$. Assume that we have constructed orthonormal elements $e_0,\dots,e_{n_0-1}\in\dom(A)$ such that~\eqref{eq.lm.approx} is satisfied for $n=1,\dots,n_0-1$.
Let 
\begin{equation}\label{eq.orth.proj}
V_{n_0}:={\rm span}\,\big\{B^{-\frac{1}{2}}e_n:\,n=0,\dots,n_0-1\big\}.
\end{equation}
Since $\lm$ belongs to the set on the right hand side of~\eqref{eq.char.proj2}, there exists $x_{n_0}\in V_{n_0}^{\perp}\cap\dom\big(AB^{-\frac{1}{2}}\big)$
with $\|x_{n_0}\|=1$  such that 
\begin{equation}\label{eq.x0}
\big|\langle B^{-\frac{1}{2}}AB^{-\frac{1}{2}}x_{n_0},x_{n_0}\rangle-\lm\big|<\frac{1}{n_0}.
\end{equation}
Define
$$e_{n_0}:=\frac{B^{-\frac{1}{2}}x_{n_0}}{\big\|B^{-\frac{1}{2}}x_{n_0}\big\|}.$$
Obviously we have $\|e_{n_0}\|=1$.
Moreover, since 
$$\langle B^{-\frac{1}{2}}AB^{-\frac{1}{2}}x_{n_0},x_{n_0}\rangle=\frac{\langle A e_{n_0},e_{n_0}\rangle}{\langle B e_{n_0},e_{n_0}\rangle},$$ 
the inequality~\eqref{eq.x0} immediately implies~\eqref{eq.lm.approx} for $n=n_0$.
Finally,  $x_{n_0}\in V_{n_0}^{\perp}$ yields
\be\langle e_{n_0},e_n\rangle=\Big\langle \frac{B^{-\frac{1}{2}}x_{n_0}}{\big\|B^{-\frac{1}{2}}x_{n_0}\big\|}, e_n\Big\rangle
=\frac{\langle x_{n_0},B^{-\frac{1}{2}}e_n\rangle}{\big\|B^{-\frac{1}{2}}x_{n_0}\big\|}=0, \quad  n=0,\dots,n_0-1.
\qedhere\ee
\end{proof}

In the following example the characterisations coincide and are closed but not necessarily convex sets.

\begin{Example}
Let $H_1,H_2$ be infinite-dimensional Hilbert spaces.
For $c\in \C$ define
$$A:={\rm diag}(I,I), \quad B:={\rm diag}(I,c)\quad \text{in}\quad H_1\oplus H_2.$$
The operator $B$ is normal; it is selfadjoint if and only if $c\in\R$.

Let  $c\in\C\backslash (-\infty,0]$, then $0\notin W_e(B)$.
It is easy to see that $0\notin W_e(A,B)\cup w_e(A,B)$ and $W_e(B,A)=w_e(B,A)={\rm conv}\,\{1,c\}$.
Then Remark~\ref{rem.inverse}~i) implies
$$W_e(A,B)=\big\{\lm^{-1}:\,\lm\in W_e(B,A)\}=\big\{\lm^{-1}:\,\lm\in {\rm conv}\,\{1,c\}\big\}=w_e(A,B).$$
Clearly these sets are convex if and only if $c>0$, in which case $B$ is selfadjoint.
\end{Example}

\begin{Corollary}\label{prop.selfadj}
Let $A$ be selfadjoint and $B$ be uniformly positive.
Then at least one of the following holds:
\begin{enumerate}
\item[\rm (i)]
The different concepts of essential numerical range coincide,
$$\overline{W_e(A,B)}=w_e(A,B)=W_e\big(B^{-\frac{1}{2}}AB^{-\frac{1}{2}}\big);$$

\item[\rm (ii)]
the operator $B^{-\frac{1}{2}}AB^{-\frac{1}{2}}$ is bounded and 
$$w_e(A,B)=W_e\big(B^{-\frac{1}{2}}AB^{-\frac{1}{2}}\big)=\sigma_e\big(B^{-\frac{1}{2}}AB^{-\frac{1}{2}}\big)$$
consists of exactly one point.
\end{enumerate}
\end{Corollary}

\begin{proof}
The identity $w_e(A,B)=W_e\big(B^{-\frac{1}{2}}AB^{-\frac{1}{2}}\big)$ follows from Theorem~\ref{thm.We.we}~ii).

If ${\rm int}\,w_e(A,B)\neq\emptyset$ (the interior with respect to the topology in $\R$), then we use the inclusion
\[ w_e(A,B) = w_e(A+K,B) \subseteq \overline{w(A+K,B)} = \overline{W(A+K,B)}, \]
for every compact operator $K$, together with the convexity of all the sets appearing, to deduce that
$ {\rm int}\, w_e(A,B) \subseteq W(A+K,B) $ for all compact $K$, whence
\[ {\rm int}\, w_e(A,B) \subseteq \underset{K {\rm compact}}\bigcap W(A+K,B) = W_{e}(A,B). \]
It follows, again by convexity, that $w_e(A,B) = \overline{W_e(A,B)}$.

Now assume that ${\rm int}\, w_e(A,B)=\emptyset$.
If even $w_e(A,B)=\emptyset$,
then $W_e(A,B)=\emptyset$ by Theorem~\ref{thm.We.we}~i), so (i) is satisfied.
The remaining possibility is that  $w_e(A,B)\neq\emptyset$ but
${\rm int}\,w_e(A,B)=\emptyset$.
Then there exists $\lm\in\R$ such that
$$W_e\big(B^{-\frac{1}{2}}AB^{-\frac{1}{2}}\big)= w_e(A,B)=\{\lm\}.$$
By Theorem \ref{Thm2.11}, we have
$$ W_e\big(B^{-\frac{1}{2}}AB^{-\frac{1}{2}}\big)={\rm conv}\Big(\widehat{\sigma}_e\big(B^{-\frac{1}{2}}AB^{-\frac{1}{2}}\big)\Big)\backslash\{\pm\infty\}.$$
The latter can consist of exactly one point $\lm$ only if $\sigma_e\big(B^{-\frac{1}{2}}AB^{-\frac{1}{2}}\big)=\{\lm\}$
and $B^{-\frac{1}{2}}AB^{-\frac{1}{2}}$ is a bounded operator; hence (ii) is satisfied.
\end{proof}

The following example illustrates case (ii) in Corollary~\ref{prop.selfadj}. 

\begin{Example}\label{ex.caseb}
Consider the operators $A$ and $B$ of Example~\ref{ex.unifposB}.
The operator
$$B^{-\frac{1}{2}}AB^{-\frac{1}{2}}={\rm diag}\Big(1+\frac{1}{n}:\,n\in\N\Big)$$
is bounded and selfadjoint with
$$w_e(A,B)=W_e\big(B^{-\frac{1}{2}}AB^{-\frac{1}{2}}\big)=\sigma_e\big(B^{-\frac{1}{2}}AB^{-\frac{1}{2}}\big)=\{1\}.$$
Therefore we are in case (ii) of Corollary~\ref{prop.selfadj}.
One may verify that $W_e(A,B)=\emptyset$ and hence (i) is not satisfied.
\end{Example} 

Now we construct a non-selfadjoint example (but still with a uniformly positive~$B$  in order that Theorem~\ref{thm.We.we}~ii) is applicable) for which we have $$\overline{W_e(A,B)}\subsetneqq w_e(A,B)=W_e\big(B^{-\frac{1}{2}}AB^{-\frac{1}{2}}\big).$$

\begin{Example}\label{ex.line}
In $l^2(\N)$ consider the diagonal operators
$$A:={\rm diag}\big((-1)^n n^3+\I n:\,n\in\N\big), \quad B:={\rm diag}(n^2:\,n\in\N);$$
again we identify the operators with their matrix representations.
We have $B\geq I$.
Consider
$$B^{-\frac{1}{2}}AB^{-\frac{1}{2}}={\rm diag}\Big((-1)^n n+\frac{\I}{n}:\,n\in\N\Big)=T+ K$$
with
$$T:={\rm diag}\big((-1)^n n :\,n\in\N\big), \quad K:={\rm diag}\Big(\frac{\I}{n}:\,n\in\N\Big).$$
We obtain by Theorem~\ref{thm.We.we}~ii) and using that $K$ is compact,
$$w_e(A,B)=W_e\big(B^{-\frac{1}{2}}AB^{-\frac{1}{2}}\big)=W_e(T+ K)=W_e(T)=\R.$$
However, one may check that $W_e(A,B)=\emptyset$.
\end{Example}

\subsection{Perturbation results}

In the first result we assume that one of the operators $A$, $B$ has empty essential spectrum.

\begin{Theorem}\label{thm.pert.empty}
Assume that one of the following holds:
\begin{enumerate}
\item[\rm(a)]
$W_e(A)=\emptyset$ and $K$ is $A$-form bounded with relative form bound $0$;
\item[\rm(b)]
$W_e(B)=\emptyset$ and $K$ is $B$-form bounded with relative form bound $0$.
\end{enumerate}
Then
$$w_e(A+K,B)=w_e(A,B), \quad w_e(A,B+K)=w_e(A,B).$$
\end{Theorem}

\begin{proof}
Throughout this proof we use the fact that if $T$ is a linear operator with $W_e(T)=\emptyset$ and $K$ is $T$-form bounded with relative form bound $0$, then every sequence $(x_n)_{n\in\N}\subset\dom(T)$ with $\|x_n\|=1$ and $x_n\stackrel{w}{\to}0$ satisfies $|\langle Tx_n,x_n\rangle|\to\infty$ and $|\langle Kx_n,x_n\rangle|/|\langle Tx_n,x_n\rangle|\to 0$; the latter follows from Proposition~\ref{prop.We.empty} and its proof. In particular, we obtain $W_e(T+K)=\emptyset$. Further note that $-K$ is $(T+K)$-form bounded with relative form bound $0$.

We show
\begin{equation}\label{eq.pert}
w_e(A,B)\subseteq w_e(A+K,B), \quad w_e(A,B)\subseteq w_e(A,B+K);
\end{equation}
then the reverse inclusions follow from applying~\eqref{eq.pert} to $A'=A+K$, $K'=-K$ and $B'=B+K$, $K'=-K$, respectively.

Let $\lm\in w_e(A,B)$. There exists $(x_n)_{n\in\N}\subset\dom(A)$ with $\|x_n\|=1$ and $x_n\stackrel{w}{\to}0$
such that 
$$\langle B x_n,x_n\rangle\neq 0, \quad \frac{\langle A x_n,x_n\rangle}{\langle B x_n,x_n\rangle}\tolong \lm, \quad n\to\infty.$$
Note that the assumption (a) implies $|\langle K x_n,x_n\rangle/\langle A x_n,x_n\rangle |\to 0$, and the assumption (b) yields $|\langle K x_n,x_n\rangle/\langle B x_n,x_n\rangle |\to 0$.
Hence, in both cases the difference
$$\left|\frac{\langle A x_n,x_n\rangle}{\langle B x_n,x_n\rangle}-\frac{\langle (A+K) x_n,x_n\rangle}{\langle B x_n,x_n\rangle}\right|
=\left|\frac{\langle K x_n,x_n\rangle}{\langle B x_n,x_n\rangle}\right|
=\left|\frac{\langle A x_n,x_n\rangle}{\langle B x_n,x_n\rangle}\right| \left|\frac{\langle K x_n,x_n\rangle}{\langle A x_n,x_n\rangle}\right|$$
converges to $0$, so
$\langle (A+K) x_n,x_n\rangle/\langle B x_n,x_n\rangle\to \lm\in w_e(A+K,B)$.
In addition, in both cases the difference
\begin{align*}
\left|\frac{\langle A x_n,x_n\rangle}{\langle B x_n,x_n\rangle}-\frac{\langle A x_n,x_n\rangle}{\langle (B+K) x_n,x_n\rangle}\right|
&=\frac{\left|\frac{\langle A x_n,x_n\rangle}{\langle B x_n,x_n\rangle}\right| }
{\left|1+\frac{\langle B x_n,x_n\rangle}{\langle A x_n,x_n\rangle}\,\frac{\langle A x_n,x_n\rangle}{\langle K x_n,x_n\rangle}\right| }\\
&=\frac{\left|\frac{\langle A x_n,x_n\rangle}{\langle B x_n,x_n\rangle}\right| \left|\frac{\langle K x_n,x_n\rangle}{\langle B x_n,x_n\rangle}\right| }
{\left|1+\frac{\langle K x_n,x_n\rangle}{\langle B x_n,x_n\rangle}\right| }
\end{align*}
converges to $0$. This yields $\lm\in w_e(A,B+K)$.
\end{proof}

\begin{Remark}
Under the assumptions of Theorem~\ref{thm.pert.empty},
we cannot conclude
\begin{equation}\label{eq.We.not.equal}
\overline{W_e(A+K,B)}=\overline{W_e(A,B)}, \quad \overline{W_e(A,B+K)}=\overline{W_e(A,B)}.
\end{equation}
As a counterexample, consider the operators $A$ and~$B$ from Example~\ref{ex.line} and define
$K:={\rm diag}(\I\,n:\,n\in\N).$
Note that $W_e(B)=\emptyset$ and $K$ is $B$-form bounded with relative form bound $0$.
One may verify that $1\in W_e(A,B+K)=\R$.
However, as pointed out in Example~\ref{ex.line}, we have $W_e(A,B)=\emptyset$, hence the second identity in \eqref{eq.We.not.equal} is not satisfied. In addition, the first identity is not satisfied since
$$1\in W_e((B+K)|_{\dom(A)},A), \quad  W_e(B|_{\dom(A)},A)=\emptyset.$$
Note that this example also illustrates that $W_e(A,B)$ is not invariant under the relative compactness assumptions of the following Theorem~\ref{thm.pert.rel.comp}.
\end{Remark}

In the next result we do not assume that $W_e(A)=\emptyset$ or $W_e(B)=\emptyset$ but use a relative compactness argument instead.

\begin{Theorem}\label{thm.pert.rel.comp}
Let $T$ and $S$ be linear operators in $H$ with $0\notin W_e(T)$ and such that one of the following holds:
\begin{enumerate}
\item[\rm (a)]
$T=T_1+\I T_2$ and $S=S_1+\I S_2$ with symmetric operators $T_1$,~$T_2$ and $S_1$, $S_2$  such that
\begin{enumerate}
 \item[\rm (i)] \hspace{0.5mm}$T_1$ is selfadjoint and semibounded, $S_1$, $S_2$ are $T_1$-compact, or
 \item[\rm (ii)] $T_2$ is selfadjoint and semibounded, $S_1$, $S_2$ are $T_2$-compact, or
 \item[\rm (iii)] $T_1$, $T_2$ are selfadjoint and semibounded, $S_1$ is $T_1$-compact and $S_2$ is $T_2$-compact;
\end{enumerate}
\item[\rm (b)]
$T=T_1+\I T_2$ with uniformly positive $T_1$ and symmetric $T_2$ such that $T_1^{-1/2}S$ is $T_1^{1/2}$-compact, i.e.\ $T_1^{-1/2}ST_1^{-1/2}$ is compact.
\end{enumerate}
Then, for any $B$ with $\dom(T)\subseteq\dom(B)$, \beq\label{eq.pertA} w_e(T+S,B)=w_e(T,B),\eeq
and, for any $A$ with $\dom(A)\subseteq\dom(T)$,
\beq\label{eq.pertB} w_e(A,T+S)=w_e(A,T).\eeq
\end{Theorem}

\begin{proof}
First we claim that, in both cases (a) and (b),
$S$ is $T$-form bounded with relative form bound $0$ and whenever $(x_n)_{n\in\N}\subset\dom(T)$ is such that $\|x_n\|=1$, $x_n\stackrel{w}{\to}0$ and $(|\langle Tx_n,x_n\rangle|)_{n\in\N}$ is bounded, then $\langle Sx_n,x_n\rangle\to 0$.
In (a), this follows from the proof of Theorem~\ref{thmABUV}, see \cite[Theorem 4.5]{We}, and in (b) from the proof of Theorem~\ref{thmSqrt}, see \cite[Theorem 4.7]{We}.

We prove the identity~\eqref{eq.pertA}; the proof of~\eqref{eq.pertB} is analogous.
Note that,  in both cases (a) and (b), we have $0\notin W_e(T)= W_e(T+S)$ by Theorems \ref{thmABUV} and \ref{thmSqrt}, and   $-S$ is $(T+S)$-form bounded with relative form bound $0$ and whenever $(x_n)_{n\in\N}\subset\dom(T)=\dom(T+S)$ is such that $\|x_n\|=1$, $x_n\stackrel{w}{\to}0$ and $(|\langle (T+S)x_n,x_n\rangle|)_{n\in\N}$ is bounded, then $\langle -Sx_n,x_n\rangle\to 0$.
Hence it suffices to show the inclusion
$w_e(T,B)\subseteq w_e(T+S,B)$; 
the reverse inclusion follows from repeating the proof for  $T'=T+S$, $S'=-S$.

Let $\lm\in w_e(T,B)$. 
 There exists $(x_n)_{n\in\N}\subset\dom(T)$ with $\|x_n\|=1$ and $x_n\stackrel{w}{\to}0$
such that 
$$\langle B x_n,x_n\rangle\neq 0, \quad \frac{\langle T x_n,x_n\rangle}{\langle B x_n,x_n\rangle}\tolong \lm, \quad n\to\infty.$$
If there exists an infinite subset $I\subseteq\N$ such that $(\langle Tx_n,x_n\rangle)_{n\in I}$ is bounded, then the above argument implies $\langle S x_n,x_n\rangle\to 0$ as $n\in I$, $n\to\infty$. Moreover, by $0\notin W_e(T)$, we have $\liminf_{n\to\infty}|\langle Tx_n,x_n\rangle| >0$, and hence 
$$\frac{\langle Sx_n,x_n\rangle}{\langle Bx_n,x_n\rangle}=\frac{\langle Sx_n,x_n\rangle}{\langle Tx_n,x_n\rangle}\frac{\langle Tx_n,x_n\rangle}{\langle Bx_n,x_n\rangle}\tolong 0, \quad n\in I, \quad n\to\infty.$$ 
So we arrive at $\lm\in w_e(T+S,B)$.

If $|\langle Tx_n,x_n\rangle|\to\infty$, then, similarly as in the proof of Proposition~\ref{prop.We.empty} for $\beta=0$, we obtain $|\langle S x_n,x_n\rangle|/|\langle T x_n,x_n\rangle|\to 0$.
Now $\lm\in w_e(T+S,B)$ follows in an analogous way as in the proof of Theorem~\ref{thm.pert.empty}.
\end{proof}

\section{Spectral approximation and application to indefinite Sturm-Liouville operator}\label{sec.approx}
In this section we study spectral convergence of the approximation of the pencil $\lm\mapsto A-\lm B$ by 
projection or domain truncation methods.
The aim is to prove \emph{spectral exactness} of the approximation by $\lm\mapsto A_n-\lm B_n$, $n\in\N$: every $\lm\in \sigma(A,B)$ is the limit of some $\lm_n\in\sigma(A_n,B_n)$, $n\in\N$, (\emph{spectral inclusion}) and no \emph{spectral pollution} occurs, i.e.\ there is no spurious eigenvalue $\lm\notin\sigma(A,B)$ which is an accumulation point of some $\lm_n\in\sigma(A_n,B_n)$, $n\in\N$.
We prove that spectral pollution is confined to one of the essential numerical ranges.
For selfadjoint $A,B$ with one of them uniformly positive, we prove that all elements of the {\em approximate point spectrum} 
are spectrally included. We apply these results to indefinite Sturm-Liouville operators.

\subsection{Spectral approximation}

First we study the projection method.
We use the following conventions.
For any closed subspace $V\subset H$ we denote by $P_V$ the orthogonal projection in $H$ onto~$V$.
For a linear operator $T$, if $V\subset\dom(T)$ then $T_V:=P_VT|_{V}$ denotes the compression of $T$ to $V$.

\begin{Theorem}\label{thm.pencil.sigma}
Assume that $\overline{\dom(A)}=H$.
 Let $H_n\subset \dom(A)$, $n\in\N$, be finite-dimensional subspaces with $P_{H_n}\s I$ as $n\to\infty$.
Consider the following conditions:
\begin{enumerate}
\item[\rm (1)]
The subspaces are such  that
\begin{equation}\label{eq.gsr.Galerkin1}
\forall\,x\in\dom(A): \quad A_{H_n}P_{H_n}x\tolong Ax, \quad B_{H_n}P_{H_n} x\tolong Bx,\quad n\to\infty.
\end{equation}

\item[\rm (2)] We have $\overline{\dom(A^*)\cap\dom(B^*)}=H$ and the subspaces satisfy $H_n\subset\dom(A^*)\cap\dom(B^*)$ with
\begin{equation}\label{eq.gsr.Galerkin2}
\forall\,x\in\dom(A^*)\cap\dom(B^*): \quad A^*_{H_n}P_{H_n}x\tolong A^* x, \quad B^*_{H_n}P_{H_n} x\tolong B^*x,\quad n\to\infty.
\end{equation}
\end{enumerate}
Depending on which condition holds, we conclude the following:
\begin{enumerate}
\item[\rm i)] Assume that  $0\notin W_e(A)\cap W_e(B)$ or $W_e(A,B)\neq \C$.
If condition {\rm (1)} or {\rm (2)} is satisfied, then  every spurious eigenvalue belongs to $w_e(A,B)\supseteq \overline{W_e(A,B)}$.
If both {\rm (1)} and {\rm (2)} are satisfied, then for every isolated $\lambda\in\sigma(A,B)$ outside $w_e(A,B)$ there exist $\lm_n\in\sigma(A_{H_n},B_{H_n})$, $n\in\N$, such that $\lm_n\to\lm$.
\item[\rm ii)] Assume that $B$ is bounded.
If condition {\rm (1)} or {\rm (2)} is satisfied,  then  every spurious eigenvalue belongs to $W_e(A,B)\supseteq w_e(A,B)$.
If both {\rm (1)} and {\rm (2)} are satisfied, then for every isolated $\lambda\in\sigma(A,B)$ outside $W_e(A,B)$ there exist $\lm_n\in\sigma(A_{H_n},B_{H_n})$, $n\in\N$, such that $\lm_n\to\lm$.
\item[\rm iii)] Assume that $A,B$ are selfadjoint and (at least) one of them is uniformly positive.
If {\rm (1)} is satisfied, then for every
 $\lm\in\sigma_{\rm app}(A,B)$
there exist $\lm_n\in\sigma(A_{H_n},B_{H_n})$, $n\in\N$, such that $\lm_n\to\lm$.
\end{enumerate}
\end{Theorem}

\begin{proof}
We abbreviate $P_n:=P_{H_n}$, $A_n:=A_{H_n}$ and $B_n:=B_{H_n}$ for $n\in\N$.

i)
First assume that (1) holds.
Assume that there exist $\lm\in\C$, an infinite subset $I\subseteq\N$ and $\lm_n\in\sigma(A_n,B_n)=\sigma(A_n^*,B_n^*)^*$, $x_n\in H_n$, $n\in I$, with $\|x_n\|=1$, $A_n^*x_n=\overline{\lm_n}B_n^*x_n$ and $\lm_n\to\lm\notin\sigma(A,B)$.
First assume that there exists a subsequence on which $x_n\stackrel{w}{\to}x\neq 0$.
Let $y\in\dom(A)$ be arbitrary.
Then the assumption~\eqref{eq.gsr.Galerkin1} and the convergences $\lm_n\to\lm$ and $x_n\stackrel{w}{\to}x$ imply
\begin{align*}
0=&\langle (A_n^*-\overline{\lm_n} B_n^*)x_n,y\rangle
=\langle x_n, (A-\lm B)y\rangle+\langle x_n, (A_n-\lm_n B_n)P_n y-(A-\lm B)y\rangle\\
\tolong &\langle x,(A-\lm B)y\rangle.
\end{align*}
Therefore $y\mapsto \langle (A-\lm B)y,x\rangle =0$ defines a bounded linear functional on $\dom(A)$.
This implies $x\in\dom((A-\lm B)^*)$ and $(A-\lm B)^*x=0$.
Since we assumed that $x\neq 0$, we have $0\in\sigma((A-\lm B)^*)=\sigma(A-\lm B)^*$
and hence $\lm\in\sigma(A,B)$, a contradiction.
Therefore, it follows that $x_n\stackrel{w}{\to}0$.
Since $A_n^*x_n=\overline{\lm_n}B_n^*x_n$, we obtain
$$\langle A x_n,x_n\rangle -\lm_n \langle B x_n,x_n\rangle =\langle x_n, (A_n^*-\overline{\lm_n} B_n^*)x_n\rangle =0, \quad n\in I.$$
If $\langle B x_n,x_n\rangle=0$ for infinitely many $n$, then  $\langle A x_n,x_n\rangle=0$ for these $n$ and hence $W_e(A,B)=\C$ and $0\in W_e(A)\cap W_e(B)$, a contradiction. Hence, without loss of generality, $\langle B x_n,x_n\rangle \neq 0$ for all $n\in I$.
Then $$\frac{\langle A x_n,x_n\rangle}{\langle B x_n,x_n\rangle}=\lm_n\tolong \lm\in w_e(A,B), \quad n\in I, \quad n\to\infty.$$

Now we assume that (2) holds. The proof is very similar but one has to pay attention to the domains of the involved operators.

Assume that there exist $\lm\in\C$, an infinite subset $I\subseteq\N$ and $\lm_n\in\sigma(A_n,B_n)$, $x_n\in H_n$, $n\in I$, with $\|x_n\|=1$, $A_nx_n=\lm_nB_nx_n$ and $\lm_n\to\lm\notin\sigma(A,B)$.
First assume that there exists a subsequence on which $x_n\stackrel{w}{\to}x\neq 0$.
Let $y\in\dom(A^*)\cap\dom(B^*)$ be arbitrary. We use that $A^*-\overline{\lm}B^*\subseteq (A-\lm B)^*$ and $A_n^*-\overline{\lm_n}B_n^*=(A_n-\lm_nB_n)^*$.
Then the assumption~\eqref{eq.gsr.Galerkin2} and the convergences $\lm_n\to\lm$ and $x_n\stackrel{w}{\to}x$ imply
\begin{align*}
0=&\langle (A_n-\lm_n B_n)x_n,y\rangle
=\langle x_n, (A-\lm B)^*y\rangle+\langle x_n, (A_n-\lm_n B_n)^*P_n y-(A-\lm B)^*y\rangle\\
\tolong &\langle x,(A-\lm B)^*y\rangle.
\end{align*}
Therefore $y\mapsto \langle (A-\lm B)^*y,x\rangle =0$ defines a (trivial) bounded linear functional on $\dom(A^*)\cap\dom(B^*)$ and hence on the whole Hilbert space since $\overline{\dom(A^*)\cap\dom(B^*)}=H$. Since $0\notin\sigma(A-\lm B)$, the operator $A-\lm B$ is closed. 
This implies $x\in\dom(A-\lm B)$ and $(A-\lm B)x=0$.
Since we assumed that $x\neq 0$, we have  $\lm\in\sigma(A,B)$, a contradiction.
Therefore, it follows that $x_n\stackrel{w}{\to}0$.
Since $A_nx_n=\overline{\lm_n}B_nx_n$, we obtain
$$\langle A x_n,x_n\rangle -\lm_n \langle B x_n,x_n\rangle  =0, \quad n\in I.$$
In completely the same way as in the previous case we arrive at $\lm\in w_e(A,B)$.
This proves the claim about spectral pollution. Note that $w_e(A,B)\supseteq \overline{W_e(A,B)}$ by 
Proposition~\ref{prop.ess.num.range}~i).

Now we turn to spectral inclusion.
Assume that both (1) and (2) hold and take an  isolated $\lambda\in\sigma(A,B)$ outside $w_e(A,B)$. Then there exists $\eps>0$ such that 
$B_{\eps}(\lm)\cap\sigma(A,B)=\{\lm\}$ and
$B_{\eps}(\lm)\cap w_e(A,B)=\emptyset$.
Choose $\delta\in(0,\eps)$. Assume that there exists an infinite subset $I\subset\N$ with ${\rm dist}(\lm,\sigma(A_n,B_n))\geq\delta$, $n\in I$. 
Define $\Gamma:=\partial B_{{\delta}/2}(\lm)$. Then the corresponding Riesz projections are
\begin{align*}
 P_{\Gamma}:=\frac{1}{2\pi\I}\int_{\Gamma} (A-zB)^{-1}\,\rd z, \quad  P_{\Gamma,n}:=\frac{1}{2\pi\I}\int_{\Gamma} (A_n-zB_n)^{-1}\,\rd z=0, \quad n\in I.
\end{align*}
Let $x\in H$ be arbitrary. For every $n\in I$ define the function $f_n:\Gamma\to [0,\infty)$ by $f_n(z):=\|(A-zB)^{-1}x-(A_n-zB_n)^{-1}P_nx\|$.
Then $$\|P_{\Gamma}x-P_{\Gamma,n}P_nx\|\leq \frac{1}{2\pi}\int_{\Gamma} f_n (z)\,\rd |z|, \quad n\in I.$$

Next we prove
\begin{equation}\label{eq:resunifbdd}
\sup_{z\in\Gamma}\,\sup_{n\in I}\|(A_n-zB_n)^{-1}\|<\infty.
\end{equation}
Assume that the latter is false. Then there exists an infinite subset $I_2\subseteq I$ such that for every $n\in I_2$ there are $z_n\in \Gamma$ and $x_n\in H_n$ with $\|x_n\|=1$, $\|(A_n-z_nB_n)x_n\|<1/n$.
Since $\Gamma$ is compact and $H$ is weakly compact, we can extract another infinite subset $I_3\subseteq I_2$ so that $(z_n)_{n\in I_2}$ converges to some $z\in\Gamma$ and $(x_n)_{n\in I_2}$ converges weakly in $H$ to some $x\in H$.
Let $y\in\dom(A^*)\cap\dom(B^*)$. 
Since $A^*-\overline{z}B^*\subseteq (A-z B)^*$, the assumption \eqref{eq.gsr.Galerkin2} and $z_n\to z$ imply
$\|(A-zB)^*y-(A_n-z_nB_n)^*P_ny\|\to 0$.
Now we estimate
\begin{align*}
|\langle x,(A-zB)^*y\rangle|
&\leq |\langle x_n,(A-zB)^*y\rangle|+|\langle x-x_n,(A-zB)^*y\rangle|,\\
|\langle x_n,(A-zB)^*y\rangle|
&\leq |\langle(A_n-z_nB_n)x_n,P_ny\rangle|+\|x_n\|\|(A-zB)^*y-(A_n-z_nB_n)^*P_ny\|\\
&\leq \|(A_n-z_n B_n)x_n\|\|y\|+\|(A-zB)^*y-(A_n-z_nB_n)^*P_ny\|,
\end{align*}
which implies $\langle x,(A-zB)^*y\rangle=0$ using the convergences above.
Analogously as above for spurious eigenvalues, we arrive at $z\in\sigma(A,B)$ if $x\neq 0$ or, if $x=0$, then
$$z=\lim_{n\in I_2\atop n\to\infty}\frac{ \langle Ax_n,x_n\rangle}{\langle Bx_n,x_n\rangle}\in w_e(A,B),$$
 which are both contradictions. This proves \eqref{eq:resunifbdd}.

Now we show that $f_n(z)\to 0$, $n\to\infty$, for every $z\in\Gamma$.
To this end, let $z\in\Gamma$. 
Define $y:=(A-zB)^{-1}x\in\dom(A)$. Then the assumptions imply $\|(A-zB)y-(A_n-zB_n)P_ny\|\to 0$ as $n\to\infty$.
Hence
\begin{align*}
f_n(z)&=\|y-(A_n-zB_n)^{-1}P_n(A-zB)y\|\\
&\leq \|y-P_ny\|+\|(A_n-zB_n)^{-1}\|\|(A-zB)y-(A_n-zB_n)P_ny\|.
\end{align*}
With \eqref{eq:resunifbdd} we obtain $f_n(z)\to 0$ as $n\to\infty$.

Note that, by \eqref{eq:resunifbdd}, $f_n(z)$ is unformly bounded in $n\in\N$ and $z\in\Gamma$. 
Lebesgue's dominated convergence theorem implies $\|P_{\Gamma}x-P_{\Gamma,n}P_nx\|\to 0$ as $n\to\infty$.
Hence $P_{\Gamma,n}P_n \s P_{\Gamma}, \,n\to\infty,$ and so we obtain $P_{\Gamma}=0$, a contradiction to 
$\lm\in B_{{\delta}/2}(\lm)\cap\sigma(A,B)\neq\emptyset$.
Therefore, there exists $n_{\delta}\in\N$ such that ${\rm dist}(\lm,\sigma(A_n,B_n))<\delta$, $n\geq n_{\delta}$.
Since $\delta$ can be chosen arbitrarily small, we finally obtain ${\rm dist}(\lm,\sigma(A_n,B_n))\to 0$, $n\to\infty$.

ii)
We proceed as in i). The only difference occurs at the point where we have
$$\langle A x_n,x_n\rangle -\lm_n\langle B x_n,x_n\rangle =0, \quad n\in I,\quad \lm_n\tolong\lm, \quad n\to\infty.$$
Since now $B$ is assumed to be bounded, we have
$$\langle (A-\lm B) x_n,x_n\rangle\tolong 0\in W_e(A-\lm B), \quad n\in I,\quad  n\to\infty.$$
Hence $\lm\in W_e(A,B)$.
Note that $W_e(A,B)\supseteq w_e(A,B)$ by Proposition~\ref{prop.ess.num.range}~ii).

iii)
Let $\lm\in\sigma_{\rm app}(A,B)$ and let $\eps>0$ be arbitrary.
Then there exists $x_{\eps}\in\dom(A)$ with $\|x_{\eps}\|=1$ and $$\|(A-\lm B)x_{\eps}\|<\eps.$$
If $\lm\in\sigma(A_n,B_n)$ for all sufficiently large $n\in\N$, the claim follows immediately. Now assume that there exist infinitely many
 $n\in\N$ such that $\lm\in\rho(A_n,B_n)$.

First assume that $B$ is uniformly positive, $B\geq c$ for some $c>0$.
Then $A_n$ is selfadjoint and $B_n\geq c $ for all $n\in\N$.
Note that $\sigma(A_n,B_n)=\sigma\big(B_n^{-\frac{1}{2}}A_nB_n^{-\frac{1}{2}}\big)$.
Since 
\begin{align*}
\|(A_n-\lm B_n)P_nx_{\eps}\|
&\geq \frac{\|P_nx_{\eps}\|}{\|(A_n-\lm B_n)^{-1}\|}
\geq  \frac{\big\|P_nx_{\eps}\big\|}{\big\|B_n^{-\frac{1}{2}}\big\|^2\big\|(B_n^{-\frac{1}{2}}A_nB_n^{-\frac{1}{2}}-\lm)^{-1}\big\|}\\
&\geq c\|P_nx_{\eps}\|\,{\rm dist}\Big(\lm,\sigma\Big(B_n^{-\frac{1}{2}}A_nB_n^{-\frac{1}{2}}\Big)\Big)\\
&=c\|P_nx_{\eps}\|\,{\rm dist}(\lm,\sigma(A_n,B_n)),
\end{align*}
we obtain, using the assumption~\eqref{eq.gsr.Galerkin1} and $P_n\s I$,
$$\limsup_{n\to\infty}{\rm dist}(\lm,\sigma(A_n,B_n))\leq 
\limsup_{n\to\infty} \frac{\|(A_n-\lm B_n)P_nx_{\eps}\|}{c\|P_nx_{\eps}\|}
\leq \frac{\eps}{c}.$$
Since $\eps>0$ can be chosen arbitrarily small, we obtain ${\rm dist}(\lm,\sigma(A_n,B_n))\to 0$ as $n\to\infty$.

Now assume that $A$  is uniformly positive, $A\geq c$ for some $c>0$.
Then $B_n$ is selfadjoint and $A_n\geq c $ for all $n\in\N$.
Since $\lm=0$ is not possible, we obtain $\|(B-\lm^{-1}A)x_\eps\|<\eps|\lm|^{-1}$.
We proceed analogously as in the previous case, with the role of $A$ and $B$ being interchanged and $\lm$ replaced by $\lm^{-1}$,
to arrive at
$$\limsup_{n\to\infty}{\rm dist}(\lm^{-1},\sigma(B_n,A_n))
\leq \limsup_{n\to\infty}\frac{1}{|\lm|}\frac{\|(A_n-\lm B_n)P_nx_\eps\|}{c\|P_nx_\eps\|}\leq\frac{\eps}{|\lm|c}.$$
Hence there exist $\mu_n\in\sigma(B_n,A_n)$, $n\in\N$, with $\mu_n\to\lm^{-1}$. In particular $\mu_n\neq 0$ for all sufficiently large $n$, and therefore $\lm_n:=\mu_n^{-1}\in\sigma(A_n,B_n)$ satisfy $\lm_n\to\lm$.
\end{proof}

\begin{Remark}\label{rem.approx}
\begin{enumerate}
\item[\rm i)]
Assume that there exists $\lm_0\in\rho(A,B)$ with
$\lm_0\in \rho(A_{H_n},B_{H_n})$, $n\in\N$, and $\sup_{n\in\N}\|(A_{H_n}-\lm_0B_{H_n})^{-1}\|<\infty.$
Then the assumption~\eqref{eq.gsr.Galerkin1} implies $(A_{H_n}-\lm_0B_{H_n})^{-1}P_{H_n}\s (A-\lm_0B)^{-1}$ as $n\to\infty$, see~\cite[Theorem~3.1]{Boegli-chap1}.

\item[\rm ii)]
If, in addition to the assumptions of claim~iii), $B$ is $A$-bounded with relative bound~$0$, then $\sigma(A,B)=\sigma_{\rm app}(A,B)$. The latter follows since, by~\cite[Corollary~1]{Hess-Kato}, $(A-\lm B)^*=A-\overline{\lm}B$ for all $\lm\in\C$.

\item[\rm iii)] In claim~iii) it is not enough to assume that  $A$ or $B$ is strictly positive. As a counterexample, let $A=B:={\rm diag}(n^{-1}:\,n\in\N\}$ in $l^2(\N)$. It is easy to see that $\sigma_{\rm app}(A,B)=\C$. However, if we truncate the pencil to $H_n:={\rm span}\{e_j:\,j=1,\dots,n\}$, then $\sigma(A_{H_n},B_{H_n})=\{1\}$ for all $n\in\N$. So every $\lm\in\C\backslash\{1\}$ is not approximated.
\end{enumerate}
\end{Remark}

\begin{Example}\label{ex.op.as.pencil}
In $H:=l^2(\N)\oplus l^2(\N)$ define $T:={\rm diag}(S,S)$ and $J:={\rm diag}(I,-I)$ with $S:={\rm diag}(n:\,n\in\N)$, identified with its matrix representation with respect to the standard orthonormal basis of $l^2(\N)$.
Note that we have the equivalence $\lm\in\sigma(T,J)$ if and only if $1/\lm\in\sigma(J|_{\dom(T)},T)$;
here we use that $\lm=0$ need not be considered since $T\geq I$.

Since $W_e(T)=\emptyset$, Proposition~\ref{prop.We.empty} implies $W_e(T,J)=\emptyset$.
Therefore, by Theorem~\ref{thm.pencil.sigma}~ii) applied to $A=T$, $B=J$, no spurious eigenvalues occur if we use a
projection method of the pencil $\mathcal{L}(\lm):=T-\lm J$.
Together with Theorem~\ref{thm.pencil.sigma}~iii)  we conclude spectral exactness of the projection method.

Note that $J^{-1}=J$, and $JT$ is selfadjoint with $\sigma(JT)=\sigma(T,J)=\Z\backslash\{0\}$. However, by~\cite[Theorem~4.5]{We}, if we apply the projection method to $JT$, spectral pollution can be arranged to occur 
at \emph{any} point in $\R\backslash\sigma(JT)$ since, by Theorem~\ref{Thm2.11}, 
$$W_e(JT)={\rm conv}\,(\widehat{\sigma}_e(JT))\backslash\{\pm\infty\}=\R.$$
\end{Example}

The following theorem shows that arbitrary compact subsets of $W_e(A,B)$ can be filled with spurious eigenvalues.

\begin{Theorem}\label{thmdescloux.pencil}
Assume that the following holds:
  \begin{enumerate}[label=\rm{(\alph{*})}] 
\item $A$ is densely defined, and $0\notin \overline{W(A)}\cap \overline{W(B)}$ or $W(A,B)\neq \mathbb C$;
\item for every $\lm\in\C$, $\overline{\dom(A-\lm B)\cap \dom((A-\lm B)^*)}=H$
or $W(A-\lm B)\neq\C$. 
\end{enumerate}
Let $V_n\subset \dom(A)$, $n\in\N,$ be  finite-dimensional subspaces such that $P_{V_n}\s I$.
Then, for any compact subset $\Omega\subseteq W_e(A,B)$, there exist finite-dimensional subspaces $H_n\subset \dom(A),\,n\in\N,$ with $V_n\subseteq H_n$ satisfying the following properties:
  \begin{enumerate}[label=\rm{\roman{*})}] 
\item every $\lm\in\Omega\backslash\sigma(A,B)$ is a spurious eigenvalue, $$\sup_{\lm\in\Omega}\,{\rm dist}(\lm,\sigma(A_{H_n}, B_{H_n}))\tolong 0, \quad n\to\infty.$$
\item If $\Omega\subset{\rm int}\, W_e(A,B)$ is a finite set, then $$\sigma(A_{H_n},B_{H_n})=\sigma(A_{V_n},B_{V_n})\cup\Omega, \quad n\in\N.$$
\end{enumerate}
\end{Theorem}

\begin{proof}
First we derive a general argument for an arbitrary $\lm\in\Omega\subseteq W_e(A,B)$; it is the generalisation of~\cite[Lemma 6.6]{We} from operators to pencils. 
Let $V\subset\dom(A)$ be a finite-dimensional subspace and let $\eps>0$.
Define
$$U:={\rm span}\big(V\cup \ran(A|_V)\cup\ran(B|_V)\big).$$
Then ${\rm rank}\, P_U<\infty$.
 By assumption~(b) and Theorem~\ref{thm.proj},
we obtain 
$$\lm\in W\big(A|_{U^{\perp}\cap\dom(A)},B|_{U^{\perp}\cap\dom(A)}\big).$$
By the assumption~(a) and Proposition~\ref{prop.num.range}~i), we conclude 
$$\lm\in \overline{w\big(A|_{U^{\perp}\cap\dom(A)},B|_{U^{\perp}\cap\dom(A)}\big)};$$ 
under the assumptions of claim~ii) we can omit the closure.
Hence there exists $\mu\in B_{\eps}(\lm)$ ($\mu=\lm$ in claim~ii)) and a normalised $x\in U^{\perp}\cap\dom(A)\subseteq V^{\perp}\cap\dom(A)$ such that $ 0 = \langle (A-\mu B)x,x\rangle$; it follows that if $V_x:=V\oplus{\rm span}\{x\}$ then $(A-\mu B)_{V_x}$ admits the triangular representation
$$ (A-\mu B)_{V_x} = \left(\begin{array}{cc} (A-\mu B)_V &  T  \\ 0 & 0 \end{array}\right) $$
and therefore
$\mu\in \sigma(A_{V_x},B_{V_x}).$

Let $n\in\N$. There exists a finite open covering  
$\{D_{k;n}:\,k=1,\dots,N_n\}$ of $\Omega$ by open disks $D_{k;n}:=B_{1/n}(c_{k;n})$ with centres $c_{k;n}$ and equal radius $1/n$.
By applying the above argument inductively $N_n$ times with $\eps=1/n$, we construct orthonormal elements 
$x_{1;n},\dots,x_{N_n;n}\in V_n^{\bot}\cap\dom(A)$  
and points $\mu_{k;n}\in B_{1/n}(c_{k;n})$, $k=1,\ldots,N_n$, such that 
$$H_n:=V_n\oplus {\rm span}\{x_{1;n}\}\oplus\dots\oplus {\rm span}\{x_{N_n;n}\}$$ satisfies 
$$\{\mu_{1;n},\dots,\mu_{N_n;n}\}\subseteq \sigma(A_{H_n},B_{H_n}), \quad k=1,\dots,N_n.$$
By construction of the disks $D_{k;n},\,k=1,\dots,N_n$, we have 
$$\sup_{\lm\in\Omega}\,\dist(\lm,\sigma(A_{H_n},B_{H_n}))\leq \sup_{\lm\in\Omega}\,\min_{k=1, \dots,N_n}\,(|\lm-\mu_{k;n}|)\leq \frac{2}{n}\tolong 0, \quad n\to\infty.$$

In ii), with $\Omega=\{\mu_1,\dots,\mu_N\}$, we apply the general argument inductively $N$ times to construct orthonormal elements $x_{1;n},\dots,x_{N;n}\in V_n^{\perp}\cap\dom(A)$ such that $H_n:=V_n\oplus{\rm span}\{x_{1;n}\}\oplus\dots\oplus{\rm span}\{x_{N;n}\}$ satisfies $\sigma(A_{H_n},B_{H_n})=\sigma(A_{V_n},B_{V_n})\cup\Omega$.

%
\end{proof}

Now we approximate a differential operator pencil via domain truncation. To this end, let $\Omega\subseteq\R^d$ be a domain and let $\Omega_n\subset\Omega$, $n\in\N$, be bounded, nested subdomains that exhaust $\Omega$. We consider two differential expressions $\tau_1$ and $\tau_2$ and associated operators whose actions on
appropriate domains are determined always by these same expressions.
The following spectral convergence results are similar to the ones in \cite[Theorems VIII.23--25]{reedsimon1} for selfadjoint operators, where also a common core assumption as in (a) is used.

\begin{Theorem}\label{thm.trunc}
Let $A$, $B$ be realisations of $\tau_1$, $\tau_2$, respectively, in $L^2(\Omega)$ such that $B$ is $A$-bounded 
and $B^*$ is $A^*$-bounded. For $n\in\N$ let $A_n$, $B_n$ be realisations of $\tau_1$, $\tau_2$, respectively, in $L^2(\Omega_n)$ such that 
$\dom(A_n)\subseteq\dom(B_n)$ and $\dom(A_n^*)\subseteq\dom(B_n^*)$.
Assume that
\begin{enumerate}
\item[\rm(a)] 
there exists a core $\Phi\subseteq\dom(A^*)$ of $A^*$ such that for all $f\in\Phi$ there exists $n_f\in\N$ for which the restriction $f|_{\Omega_n}$ lies in $\dom(A_n^*)$, $n\geq n_f$;
\item[\rm(b)]
the quadratic forms $a$ and $b$ associated with $A$ and $B$ are closable with $\dom(\overline{a})\subseteq \dom(\overline{b})$ and, for each $n\in\mathbb N$ and any $f_n$ in $\dom(A_n)$, the extension by zero of $f_n$ to $L^2(\Omega)$ lies in $\dom(\overline{a})\subseteq \dom(\overline{b})$; denoting this extension also by $f_n$, assume further
that $\langle A_n f_n,f_n \rangle = \overline{a}[f_n]$ and $\langle B_n f_n,f_n\rangle = \overline{b}[f_n]$. 
\item[\rm(c)] The spectra $\sigma(A_n,B_n)$, $\sigma(A_n^*,B_n^*)$ consist entirely of eigenvalues.
\end{enumerate}
Then the following holds:
\begin{enumerate}
\item[\rm i)] Assume that  $0\notin W_e(A)\cap W_e(B)$ or $W_e(A,B)\neq \C$.
Then  every spurious eigenvalue belongs to $w_e(A,B)\supseteq \overline{W_e(A,B)}$.
If, in addition, 
there exists a core $\widetilde\Phi\subseteq\dom(A)$ of $A$ such that for all $f\in\widetilde\Phi$ there exists $n_f\in\N$ for which the restriction $f|_{\Omega_n}$ lies in $\dom(A_n)$, $n\geq n_f$,
 then for every isolated $\lambda\in\sigma(A,B)$ outside $w_e(A,B)$ there exist $\lm_n\in\sigma(A_n,B_n)$, $n\in\N$, such that $\lm_n\to\lm$.
\item[\rm ii)] Assume that $B$ is bounded.
Then  every spurious eigenvalue belongs to $W_e(A,B)\supseteq w_e(A,B)$.
If, in addition, there exists a core $\widetilde\Phi$ as in {\rm i)}, then for every isolated $\lambda\in\sigma(A,B)$ outside $W_e(A,B)$ there exist $\lm_n\in\sigma(A_n,B_n)$, $n\in\N$, such that $\lm_n\to\lm$.
\item[\rm iii)] Assume that $A,B$ are selfadjoint, (at least) one of them is uniformly positive and $A_n,B_n$ are selfadjoint as well.
If  there exists  a core $\widetilde\Phi$ as in {\rm i)}, then for every
 $\lm\in\sigma_{\rm app}(A,B)$
there exist $\lm_n\in\sigma(A_{H_n},B_{H_n})$, $n\in\N$, such that $\lm_n\to\lm$.
\end{enumerate}
\end{Theorem}

\begin{Remark}
Typically, $\tau_1$ will be an elliptic differential operator and $\tau_2$ will be either a multiplication operator or an elliptic operator whose
order is less than that of $\tau_1$. The hypotheses concerning cores and extensions will usually be satisfied if the domains of the $A_n$ and $B_n$ are equipped with 
suitable boundary conditions. For instance, if $\tau_1$ is an operator of order $2\nu$, $\nu\in \mathbb N$, then the traces of functions in $\dom(A_n)$ on the boundary
of $\Omega_n$ should vanish from order $0$ up to order $\nu-1$.

\end{Remark}

\begin{proof}[Proof of Theorem {\rm  \ref{thm.trunc}}]
The proof is similar to the one of Theorem~\ref{thm.pencil.sigma}, with (1) and (2) replaced by the assumptions on the cores $\Phi$ and $\widetilde\Phi$.

i) and ii) Assume that there exist $\lambda\in\mathbb C$ and an infinite index set $I\subseteq\mathbb N$ and $\lambda_n
\in\mathbb C$, $n\in I$, such that  $0$ is an eigenvalue of each $(A_n-\lambda_nB_n)$.  Let $f_n$ be the 
normalised eigenfunctions in $L^2(\Omega_n)$. Since the $\lambda_n$ are supposed
to form a polluting sequence we assume that $\lambda_n \rightarrow\lambda$ where $\lambda\not\in\sigma(A,B)$.
Suppose that on some subsequence, $f_n\w f\neq 0$. Let $g\in \Phi$. We have, for $n\geq n_g$, by assumption~(a), 
\begin{align*}
 0 &= \langle (A_n-\lambda_n B_n)f_n,g\rangle 
= \langle f_n,(A_n^*-\overline{\lambda_n}B_n^*)g|_{\Omega_n}\rangle \\
&= \langle f_n, (\tau_1^*-\overline{\lambda_n} \tau_2^*)g\rangle
 =  \langle f_n,(A^*-\overline{\lambda_n} B^*)g\rangle \\
&= \langle f_n,(A^*-\overline{\lambda}B^*)g\rangle
 + (\lambda-\lambda_n)\langle f_n, B^* g\rangle\\
& \longrightarrow
\langle f ,(A^*-\overline{\lambda} B^*)g\rangle=
 \langle f ,(A-\lambda B)^*g\rangle, 
\end{align*}
in which we have abused notation to use the symbol $f_n$ to mean the extension by zero of $f_n$ to $L^2(\Omega)$ in
the second and third lines, and we used $A^*-\overline{\lambda}B^*\subseteq (A-\lm B)^*$. Thus $g\mapsto  \langle (A-\lambda B)^*g,f\rangle =0 $ is a (trivial) bounded linear functional on $\Phi$ and
 hence on the whole Hilbert space. Since $0\not\in \sigma(A-\lambda B)$, the operator $A-\lambda B$ is closed. Hence $f\in \dom(A-\lambda B)$ and 
 \[ (A-\lambda B)f = 0. \]
 This contradicts the assumption that $\lambda\not\in\sigma(A,B)$, and so $f_n\w 0$. 
 
 We now know that 
 \[ 0 = \langle (A_n-\lambda_nB_n)f_n,f_n\rangle = (\overline{a}-\lambda\overline{b})[f_n].\] 
 By definition of the closed forms $\overline{a}$ and
 $\overline{b}$ there exist functions $h_n\in \dom(A)$, with $\| h_n \| = 1$, $\| h_n - f_n \| \rightarrow 0$, implying $h_n\w 0$, such that
 \[ 0 = \lim_{n\rightarrow\infty}(\overline{a}[h_n] - \lambda_n \overline{b}[h_n]) = \lim_{n\rightarrow\infty}\left\{ \langle Ah_n,h_n\rangle - \lambda_n
  \langle Bh_n,h_n \rangle\right\}. \]
If, on any subsequence, $\langle Bh_n,h_n\rangle$ tends to zero, then so must $\langle Ah_n,h_n\rangle$, and hence $0$ lies both in
$W_e(A)$ and in $W_e(B)$, and $W_e(A,B)=\mathbb C$. Therefore $\langle Bh_n,h_n\rangle$ is bounded away from zero. We can
therefore divide by $\langle Bh_n,h_n\rangle$ and obtain
\[ 0 = \lim_{n\rightarrow\infty}\left\{ \frac{ \langle Ah_n,h_n\rangle}{\langle Bh_n,h_n\rangle} - \lambda_n \right\} \]
and, since $\lambda_n\rightarrow\lambda$, deduce $\lambda\in w_e(A,B)$. This proves the claim about spurious eigenvalues in i).

 If $B$ is bounded then $\langle Bh_n,h_n \rangle$ is bounded, and so 
 \[  \lim_{n\rightarrow\infty}\left\{ \langle Ah_n,h_n\rangle - \lambda  \langle Bh_n,h_n \rangle\right\} = 0, \]
 giving $\lambda \in W_e(A,B)$, which proves the claim about spurious eigenvalues in ii).

Given that the assumption on the core $\widetilde\Phi$ holds, we claim that every isolated $\lm\in\sigma(A,B)$ outside $w_e(A,B)$ (in i)) or $W_e(A,B)$ (in ii)) is the limit of some $\lm_n\in\sigma(A_n,B_n)$, $n\in\N$.
This is proved analogously as in Theorem~\ref{thm.pencil.sigma}; instead of \eqref{eq.gsr.Galerkin1} and \eqref{eq.gsr.Galerkin2}
we apply the assumptions on the cores $\Phi$ and $\widetilde\Phi$, in the same way as above.
 
 Finally, if $B$ is uniformly positive, the proof follows that of part iii) of Theorem~\ref{thm.pencil.sigma} provided we make the important 
 observation that, because of the hypotheses on our domain truncation, the operators $B_n$ have a lower bound which is not less than the lower bound for $B$
 (`domain monotonicity'). The same applies to the case that $A$ is uniformly positive.

\end{proof}

\subsection{Application to indefinite Sturm-Liouville operator}

Indefinite Sturm-Liouville operators were studied both as pencil problem and as selfadjoint operators in Krein spaces, see e.g.\ \cite{SL1,SL2} and the references therein. We establish spectrally exact approximations of the operator pencil, both for projection and interval truncation methods.
For uniformly positve $T$ and interval trunction, spectral exactness was proved in 
\cite{MZ}. Here we give a short and elegant proof using essential numerical ranges, and we extend the result to the projection method (see Theorem \ref{thm.schroed}). In addition, we can prove, for the first time, spectral exactness for interval truncation even if the potential $V$ tends to zero at infinity (see Theorem \ref{thm.schroed2}). 

Let $-\infty<a\leq b<\infty$ and let $J\in L^{\infty}(\R)$ be real-valued with
$$J|_{(-\infty,a)}\equiv -1, \quad J|_{(b,\infty)}\equiv 1.$$
In particular, if $a=b=0$, then $J$ is the sign function. If however $a<b$, then $J$ may have more than one sign change.
With another real-valued potential $V\in L^{\infty}(\R)$, consider the differential expression
$$(\tau f)(x):=-f''(x)+V(x) f(x).$$
In $L^2(\R)$ define the selfadjoint operator
\begin{alignat*}{3}
(Tf)(x)&:=(\tau f)(x), \quad& \dom(T)&:=W^{2,2}(\R),
\end{alignat*}
and $J$ is selfadjoint and bounded as multiplication operator in $L^2(\R)$.

\begin{Proposition}\label{prop.schroed}
\begin{enumerate}
\item[\rm i)] If $\lim_{|x|\to\infty}V(x)=0$, then $W_e(T,J)=W(T,J)=\C$ and $\sigma_e(T,J)=\R$.
\item[\rm ii)] If there exist $m_{-},m_{+}>0$ such that $\lim_{x\to \pm\infty}V(x)= m_{\pm}$, then
\begin{align*}
W_e(T,J)&=w_e(T,J)=(-\infty,-m_{-}]\,\dot\cup\, [m_{+},\infty)=\sigma_e(T,J).
\end{align*}
\end{enumerate}
\end{Proposition}

\begin{proof}
First we calculate $\sigma_e(T,J)$; the essential numerical ranges require separate proofs for i) and ii).

Let $$V_0(x):=\begin{cases}m_{-}, &x\leq 0,\\ m_{+}, &x>0,\end{cases} \quad J_0(x):=\begin{cases}-1, &x\leq 0,\\ 1, &x>0,\end{cases}$$
and define $K:=V-V_0$, $T_0:=T-K$.
Then $T_0$ is selfadjoint with $T_0\geq\min\{m_{-},m_{+}\}$, the operators $K$, $J-J_0$ are $T_0$-compact and $T=T_0+K$.
Hence $$\sigma_e(T,J)=\{\lm\in\C:\,0\in\sigma_e(T-\lm J)=\sigma_e(T_0-\lm J_0)\}=\sigma_e(T_0,J_0)
=\sigma_e(J_0^{-1}T_0).$$
Note that $J_0^{-1}=J_0$.
Let $S_0$ denote the direct sum 
of two Schr\"odinger operators, one on $(-\infty,0]$ with expression ${\rm d}^2/{\rm d}x^2 - V_0(x)$ and Dirichlet condition at $0$
(leading to essential spectrum $(-\infty,-m_-]$), the other with expression $-{\rm d}^2/{\rm d}x^2 +V_0(x)$ on $[0,\infty)$ and Dirichlet condition at $0$ (leading to essential spectrum $[m_+,\infty)$). We observe that for any $\lambda\in\rho(J_0 T_0)$  the difference
$(S_0-\lambda )^{-1} - (J_0 T_0 - \lambda )^{-1}$
has rank at most $2$ (by a variation-of-parameters calculation). Therefore $\sigma_e(J_0T_0)=\sigma_e(S_0)$. 
Since the essential spectrum of a direct sum is the union of  essential spectra of both operators, we obtain $\sigma_e(S_0)=(-\infty,-m_-]\cup [m_+,\infty)$.

This concludes the proof for the essential spectrum and we turn to the (essential) numerical ranges.

i)
It suffices to find a sequence $(f_n)_{n\in\N}\subset\dom(T)$ such that $\|f_n\|=1$, $f_n\stackrel{w}{\to} 0$, $\langle J f_n,f_n\rangle=0$
and $\langle T f_n,f_n\rangle \to 0$; then $W_e(T,J)=\C$ and hence $W(T,J)=\C$.

The assumption $\langle J f_n,f_n\rangle =0$ is satisfied if $f_n\in\dom(T)$ is 
symmetric around $\frac{a+b}{2}$
 with ${\rm supp}f_n\cap(a,b)=\emptyset$, so we restrict our attention to such functions.
Using integration by parts, we obtain
$$\langle T f_n,f_n\rangle =\|f_n'\|^2+\langle V f_n,f_n\rangle.$$ 
Let $\phi\in C_0^{\infty}(\R)$ be an even function that satisfies 
$$\phi(x)\in [0,1], \quad {\rm supp}\, \phi\subset (-2,-1)\cup (1,2),\quad \|\phi\|=1.$$
Define, for $n\in\N$ such that $n^2\geq\frac{b-a}{2}$,
$$f_n(x):=\frac{1}{n}\phi\left(\frac{x-\frac{a+b}{2}}{n^2}\right), \quad x\in\R.$$
Then  $f_n\in C_0^{\infty}(\R)\subset\dom(T)$ is symmetric around $\frac{a+b}{2}$ and satisfies 
\begin{equation}\label{eq.cond.fn}
\begin{aligned}
{\rm supp}\, f_n&\subset \left(\frac{a+b}{2}-2n^2,\frac{a+b}{2}-n^2\right)\cup\left(\frac{a+b}{2}+n^2,\frac{a+b}{2}+2n^2\right)\\
&\subset (-\infty,a)\cup(b,\infty),\\
 \|f_n\|&=1, \quad \|f_n'\|=\frac{\|\phi'\|}{n^2}.
\end{aligned}
\end{equation}
Note that $\|f_n'\|\to 0$. In addition, $\lim_{|x|\to\infty}V(x)=0$ together with the first claim in~\eqref{eq.cond.fn} imply $\langle V f_n,f_n\rangle\to 0$; hence $\langle Tf_n,f_n\rangle\to 0$. Moreover,  the first claim in~\eqref{eq.cond.fn} yields $f_n\stackrel{w}{\to}0$.

ii)
%
Theorem~\ref{thm.We.we}~ii) implies 
$$w_e(J|_{\dom(T_0)},T_0)=W_e\big(T_0^{-\frac{1}{2}}J|_{\dom(T_0)}T_0^{-\frac{1}{2}}\big).$$
Since the set on the right hand side is closed, and the closure of $T_0^{-\frac{1}{2}}J|_{\dom(T_0)}T_0^{-\frac{1}{2}}$
is the selfadjoint bounded operator $T_0^{-\frac{1}{2}}JT_0^{-\frac{1}{2}}$,
we obtain using~\cite[Corollary~5.1]{salinas},
$$w_e(J|_{\dom(T_0)},T_0)=W_e\big(T_0^{-\frac{1}{2}}JT_0^{-\frac{1}{2}}\big)={\rm conv}\,\sigma_e\big(T_0^{-\frac{1}{2}}JT_0^{-\frac{1}{2}}\big)=\Big[-\frac{1}{m_{-}},\frac{1}{m_{+}}\Big].$$
The last equality follows from $\sigma_e\big(T_0^{-\frac{1}{2}}JT_0^{-\frac{1}{2}}\big)=\sigma_e(J|_{\dom(T_0)},T_0)=\sigma_e(T_0,J)^{-1}$.
Now we make use of the equivalence in Remark~\ref{rem.inverse}~i),
\begin{align*}
\lm\in w_e(T_0,J)\quad&\Longleftrightarrow\quad \frac{1}{\lm}\in w_e(J|_{\dom(T_0)},T_0);
\end{align*}
note that $\lm=0$ need not be considered since $T_0\geq\min\{m_{-},m_{1}\}$.
We apply the perturbation result in Theorem~\ref{thm.pert.rel.comp}~(a) to obtain 
$$w_e(T,J)=w_e(T_0,J)=(-\infty,-m_{-}]\,\dot\cup\, [m_{+},\infty).$$
Now we obtain $w_e(T,J)= W_e(T,J)$
 by Proposition~\ref{prop.ess.num.range}~ii) and using $0\notin W_e(T_0)=W_e(T)$ by Theorem~\ref{thmABUV}.
\end{proof}

Under the assumptions of  Proposition~\ref{prop.schroed}~ii),  spectral exactness prevails if we approximate the pencil using projection or domain truncation methods.

\begin{Theorem}\label{thm.schroed}
Assume that there exist $m_{-},m_{+}>0$ such that $\lim_{x\to \pm\infty}V(x)= m_{\pm}$.
\begin{enumerate}
\item[\rm i)] 
Let $H_n\subset W^{2,2}(\R)$, $n\in\N$, be finite-dimensional subspaces with $P_{H_n}\s I$ as $n\to\infty$.
Assume that $$\forall\,f\in W^{2,2}(\R):\quad T_{H_n}P_{H_n}f\tolong Tf, \quad n\to\infty.$$
Then the approximation of the pencil $\lm\mapsto T-\lm J$ by $\lm\mapsto T_{H_n}-\lm J_{H_n}$, $n\in\N$, 
is free of spectral pollution; it is even spectrally exact if $T$ is uniformly positive.

\item[\rm ii)]
Define
\begin{alignat*}{3}
(T_nf)(x)&:=(\tau f)(x), \quad& \dom(T_n)&:=\big\{f\in W^{2,2}(-n,n):\,f(\pm n)=0\big\},\\
(J_nf)(x)&:=J(x)f(x), \quad &\dom(J_n)&:=L^2(-n,n).
\end{alignat*}
Then the approximation of the pencil $\lm\mapsto T-\lm J$ by $\lm\mapsto T_n-\lm J_n$, $n\in\N$, 
is free of spectral pollution; it is even spectrally exact if $T$ is uniformly positive.
\end{enumerate}
\end{Theorem}

\begin{proof}
By  Proposition~\ref{prop.schroed}~ii), we have 
$$W_e(T,J)=(-\infty,-m_-]\,\dot\cup\, [m_+,\infty)=\sigma_e(T,J)\subseteq\sigma(T,J).$$
Now claim~i) follows from Theorem~\ref{thm.pencil.sigma}~ii),~iii) and Remark~\ref{rem.approx}~ii).
Analogously, claim~ii) is obtained with Theorem~\ref{thm.trunc} using that $\Phi=C_0^{\infty}(\R)$ is a core of $T=T^*$ and $\dom(T_n)\subset W^{1,2}(\R)$ (by extending every function by zero outside $[-n,n]$) for all $n\in\N$. 
\end{proof}

In the next result we make use of the fact that the domain truncation process commutes with multiplication with a bounded and boundedly invertible function.

\begin{Theorem}\label{thm.schroed2}
Let $a<b$ and let $B_{\varphi}$ be the bounded and continuous function
$$B_{\varphi}(x):=\begin{cases}
\e^{\I\,\varphi}, & x\in (-\infty,a],\\
\e^{\I\,t \varphi}, & x\in (a,b),\,t=\frac{b-x}{b-a},\\
1, & x\in [b,\infty).
\end{cases}$$
Then
$$\underset{\varphi\in(-\pi,0)\cup (0,\pi)}{\bigcap}W_e(B_\varphi T,B_{\varphi}J)\subseteq \R.$$
If $\lim_{|x|\to\infty}V(x)=0$, then the above sets coincide and equal $\sigma_e(T,J)$; 
in this case interval truncation as in Theorem{\rm~\ref{thm.schroed}~ii)} is spectrally exact. 
\end{Theorem}

\begin{proof}
The assumption $V\in L^{\infty}(\R)$ implies the existence of $v>0$ such that $V(x)\geq -v$ for almost every $x\in\R$.
We prove 
$$W_e(B_\varphi T,B_{\varphi}J)\subseteq \begin{cases}
\left\{\lm\in\C:\,\im\,\lm\leq v|\sin\varphi|\right\}, &\varphi\in \left(-\pi,-\frac{2\pi}{3}\right],\\[1mm]
\left\{\lm\in\C:\,\im\,\lm\geq - v\sin\varphi \right\}, &\varphi\in \left[\frac{2\pi}{3},\pi\right).
\end{cases}$$
Let $\varphi\in (-\pi,0)\cup (0,\pi)$, $\lm\in\C$ and $f\in\dom(T)$. 
Then, using integration by parts,
\begin{align*}
\langle (B_{\varphi}T-\lm B_{\varphi}J)f,f\rangle&=
\langle -B_{\varphi}f''+B_{\varphi}(V-\lm J) f,f\rangle \\
&=\langle B_{\varphi} f',f'\rangle + \langle B_{\varphi}' f',f\rangle +\langle B_{\varphi}V f,f\rangle-\lm \langle B_{\varphi}J f,f\rangle.
\end{align*}
The quadratic form $f\mapsto {\rm e}^{-\I\,\varphi/2}\langle B_{\varphi}f',f'\rangle$ is sectorial with sectoriality vertex $0$ and semi-angle $|
\varphi|/2<\pi/2$.
Note that 
\beq\label{eq.bddforms}
|\langle B_{\varphi}V f,f\rangle-\lm \langle B_{\varphi}J f,f\rangle|\leq (\|V\|_{\infty}+|\lm|\|J\|_{\infty}) \|f\|^2, \quad n\in\N.
\eeq
Moreover, since ${\rm supp}\,B'_{\varphi}=[a,b]$, we have, for any $\eps>0$,
\begin{equation}\label{eq.indefangle}
\begin{aligned}
| \langle B_{\varphi}' f',f\rangle|
&= \left|\int_a^b   B_{\varphi}' f'\overline{f}\,\rd x\right|
\leq \|B_{\varphi}'\|_{\infty}\left(\int_a^b |f'|^2\,\rd x\right)^{\frac{1}{2}}\left(\int_a^b |f|^2\,\rd x\right)^{\frac{1}{2}}\\
&\leq  \|B_{\varphi}'\|_{\infty} \left(\eps  \int_a^b |f'|^2\,\rd x +\frac{1}{4\eps}\|f\|^2\right)\\
&\leq  \|B_{\varphi}'\|_{\infty} \left(\frac{\eps}{|\cos(\varphi/2)|}\re\big({\rm e}^{-\I\,\varphi/2}\langle B_{\varphi}f',f'\rangle\big)   +\frac{1}{4\eps}\|f\|^2\right).
\end{aligned}
\end{equation}
Choosing $\eps>0$ sufficiently small, the estimates~\eqref{eq.bddforms},~\eqref{eq.indefangle} and~\cite[Theorem~VI.1.33]{kato} imply that $f\mapsto {\rm e}^{-\I\,\varphi/2}\langle (B_{\varphi}T-\lm B_{\varphi}J)f,f\rangle$ is sectorial.

Now let $(f_n)_{n\in\N}\subset \dom(T)$ with $\|f_n\|=1$, $f_n\stackrel{w}{\to} 0$ and $\langle (B_{\varphi}T-\lm B_{\varphi}J)f_n,f_n\rangle\to 0$.
The above considerations imply that the sequences $(\langle B_{\varphi}f_n',f_n'\rangle)_{n\in\N}$, $(\langle B_{\varphi}'f_n',f_n\rangle)_{n\in\N}$, $(\langle B_{\varphi}V f_n,f_n\rangle)_{n\in\N}$ and $(\langle B_{\varphi}Jf_n,f_n\rangle)_{n\in\N}$ are bounded. By passing to a subsequence, there exist $c_1,c_2,c_3,c_4\in\C$ with $c_1+c_2+c_3-\lm c_4=0$ such that, in the limit $n\to\infty$,
$$\langle B_{\varphi}f_n',f_n'\rangle\tolong c_1, \quad \langle B_{\varphi}'f_n',f_n\rangle\tolong c_2, \quad \langle B_{\varphi}V f_n,f_n\rangle\tolong c_3, \quad \langle B_{\varphi} Jf_n,f_n\rangle\tolong c_4.$$
The boundedness of $(\langle B_{\varphi}f_n',f_n'\rangle)_{n\in\N}$ together with the (quasi-)sectoriality of $B_{\varphi}$ implies that $(\|f_n'\|)_{n\in\N}$ is a bounded sequence.
By the Rellich-Kondrachov theorem and $f_n\stackrel{w}{\to}0$, we obtain $\|f_n|_{[a,b]}\|_{L^2(a,b)}\to 0$.
Now the first line in~\eqref{eq.indefangle} implies that $c_2=0$.
In addition, we obtain
$$\lm=\frac{c_1+c_3}{c_4}=\frac{c_1+z}{w}$$
with
\begin{align*}
z&:=\lim_{n\to\infty}\left({\rm e}^{\I\,\varphi}  \int_{-\infty}^a V|f_n|^2\,\rd x+\int_b^{\infty} V|f_n|^2\,\rd x\right),\\
w&:=\lim_{n\to\infty}\left( -{\rm e}^{\I\,\varphi} \int_{-\infty}^a|f_n|^2\,\rd x+\int_b^{\infty} |f_n|^2\,\rd x\right).
\end{align*}
Note that  $c_1\in {\rm conv}\big([0,\infty)\cup {\rm e}^{\I\,\varphi}\,[0,\infty)\big)$, $z\in {\rm conv}\big([-v,\infty)\cup {\rm e}^{\I\,\varphi}\, [-v,\infty)\big)$ and $w\in {\rm conv}\{-{\rm e}^{\I\,\varphi},1\}$. Thus there exist $s,t\in [0,1]$ and $u_1=-{\rm e}^{\I\,\varphi}v+\alpha_1\in (-{\rm e}^{\I\,\varphi} v +[0,\infty))$, $u_2=-v+{\rm e}^{\I\,\varphi}\alpha_2\in (-v+{\rm e}^{\I\,\varphi}\,[0,\infty))$ such that
$$c_1+z=s u_1+(1-s)u_2, \quad w=-t{\rm e}^{\I\,\varphi} +(1-t).$$
First we assume that $\varphi\in\left[\frac{2\pi}{3},\pi\right)$. Then $\sin\varphi>0$.
We estimate, using $\alpha_1,\alpha_2\geq 0$ and $s,t\in [0,1]$,
\begin{align*}
\im\,\lm
&=\im\,\frac{(c_1+z)\overline{w}}{|w|^2}=\frac{-\re(c_1+z)\im\, w+\im(c_1+z)\re\, w}{|w|^2}\\
&=\frac{(-t +(2t-1)s)v
+s  t\alpha_1+(1-s)(1-t)\alpha_2
}{1-2t(1-t) (1+\cos\varphi)}\,\sin\varphi\\
&\geq \frac{-t +(2t-1)s }{1-2t(1-t) (1+\cos\varphi)}\,v\sin\varphi\\
&\geq \begin{cases}-\frac{1-t }{1-2t(1-t) (1+\cos\varphi)}\,v\sin\varphi, & t\in [0,\frac 1 2],\\ -\frac{t  }{1-2t(1-t) (1+\cos\varphi)}\,v\sin\varphi, &t\in [\frac 1 2,1].\end{cases}
\end{align*}
For a fixed $\varphi\in\left[\frac{2\pi}{3},\pi\right)$, the latter bound is a function of $t$; note that it is symmetric with respect to the point $t=1/2$, so we  consider $t\in [1/2,1]$.
An easy calculation using $1+\cos\varphi\leq \frac{1}{2}$ reveals that the minimum of the function is attained for
$t=1$; we arrive at  $\im\,\lm\geq -v\sin\varphi$.
The bound for $\varphi\in\left(-\pi,-\frac{2\pi}{3}\right]$ is obtained analogously.
Now the intersection of all $W_e(B_\varphi T,B_{\varphi}J)$ is contained in $\R$ because $v\sin\varphi\to 0$ as $\varphi\to \pm \pi$.

If $\lim_{|x|\to\infty}V(x)=0$, then $\sigma_e(T,J)=\R$ by Proposition \ref{prop.schroed} i). Hence domain truncation is free of spectral pollution
 by Theorem~\ref{thm.trunc} ii) and since $\sigma(T,J)=\sigma(B_{\varphi}T,B_{\varphi}J)$ and $\sigma(T_n,J_n)=\sigma(B_{\varphi;n}T_n,B_{\varphi;n}J_n)$, $n\in\N$, where $B_{\varphi;n}:=B_{\varphi}|_{[-n,n]}$ and $J_n:=J|_{[-n,n]}$.
In addition, for every non-real $\lm\in\sigma(T,J)$ there exists $\varphi\in (-\pi,0)\cup(0,\infty)$ so that $\lm\notin W_e(B_{\varphi}T,B_{\varphi}J)$. Hence Theorem~\ref{thm.trunc} ii) implies that $\lm$ is the limit of some $\lm_n\in\sigma(T_n,J_n)=\sigma(B_{\varphi;n}T_n,B_{\varphi;n}J_n)$, $n\in\N$.
For real $\lm\in\sigma(T,J)$, i.e.\ for $\lm\in\sigma_e(T,J)=\R$, we prove spectral inclusion as follows. 

Let $\lm\in (0,\infty)$; the proof is analogous for $\lm\in (-\infty,0)$, and the case $\lambda=0$ follows 
from either of the previous two using a diagonal sequence argument. Define the differential expression
$$\tau:=-\frac{{\rm d}^2}{{\rm d}x^2}+V,$$
which is in limit point case at $\pm\infty$.
Because $\left.J\right|_{(-\infty,a)}=-1$ and $V(x)\to 0$ as $x\to -\infty$, for each $\mu>0$ there exists a unique (up to scalar multiplication) solution of
$$
(\tau-\mu J )u^-(\mu,\cdot)=0
$$
 with $u^-(\mu,\cdot)\in L^2(-\infty,c)$ for some (and hence all) $c\in\R$. Since $V(x)\to 0$ as $x\to -\infty$, this solution has only finitely many zeros in each interval 
 $(-\infty,c)$ and so, in particular, it is the principal solution (see \cite{MZ}) of the differential equation on $(-\infty,b]$.  
 We may assume without loss of generality that $u^{-}(\lambda,b)\neq 0$; if this were not true then we could simply increase the 
 value of $b$, and still have $\left. J\right|_{[b,\infty)} = 1$ but with $u^-(\lambda,b)\neq 0$, and hence $u^-(\mu,b)\neq 0$ for all 
 $\mu$ in a neighbourhood of $\lambda$. 

Consider now the finite-interval approximations $u_n^-$ to $u^-$ defined as solutions of the boundary value problems
\[ (\tau-\mu J )u_n^-(\mu,\cdot) = 0 \;\;\; \mbox{in $(-n,b)$}; \;\;\; u_n^-(\mu,-n) = 0; \;\;\; u_n^-(\mu,b) = u^-(\mu,b). \]
By \cite{MZ} these exist and
\[ \lim_{n\rightarrow\infty} (u_n^-(\mu,b),(u_n^-)'(\mu,b)) = (u^-(\mu,b),(u^-)'(\mu,b)), \]
the limit being locally uniform in $\mu$.

Now consider the  unique solution $u_n^+(\mu,\cdot)$ of the initial value problem
$$(\tau-\mu J )u_n^+(\mu,\cdot)=0; \;\;\; u_n^+(\mu,n)=0; \;\;\; (u_n^+)'(\mu,n)=1;$$
here $'$ denotes differentiation with respect to the second variable. 
Denote by $S^+$ the realisation of $\tau$ in $L^2(b,\infty)$ with Dirichlet boundary condition $f(b)=0$.
Then $S^+$ is selfadjoint with $\sigma_e(S^+)=[0,\infty)$.
If we denote by $S_{n}^+$ the realisation of $\tau$ in $L^2(b,n)$ with the boundary conditions $f(b)=0$, $f(n)=0$, this operator is selfadjoint as well.
Since $\lm\in\sigma_e(S^+)$ and since the spectral approximation of $S^+$ by $S_{n}^+$ is well known to be spectrally exact \cite{Bailey-1993},
for each $\eps>0$ there exists $n_\eps\in\N$ such that, for all $n\geq n_\eps$, there are two Dirichlet eigenvalues in $[\lm-\eps,\lm+\eps]$, i.e.\ there are $\lm-\eps\leq\mu_n^{(1)}<\mu_n^{(2)}\leq \lm+\eps$ with
$$u_n^+(\mu_n^{(1)},b)=0,\quad u_n^+(\mu_n^{(2)},b)=0,\quad u_n^+(\mu,b)\neq 0, \quad \mu\in (\mu_n^{(1)},\mu_n^{(2)}).$$
Thus the Titchmarsh-Weyl function
$$\mu \mapsto m_n^+(\mu):=\frac{(u_n^+)'(\mu,b)}{(u_n^+)(\mu,b)},$$
being Nevanlinna \cite{MR2077204}, is continuous  and strictly increasing on $(\mu_n^{(1)},\mu_n^{(2)})$ with singularities at the endpoints,
and 
\[ \lim_{\mu\searrow \mu_n^{(1)}}m_n^+(\mu) = -\infty, \;\;\; \lim_{\mu\nearrow \mu_n^{(2)}}m_n^+(\mu) = +\infty. \]
Correspondingly, the function 
\[ \mu\mapsto m_n^-(\mu) := \frac{(u_n^-)'(\mu,b)}{(u_n^-)(\mu,b)} \]
is continuous on $[\mu_n^{(1)},\mu_n^{(2)}] \subseteq [\lambda-\epsilon,\lambda+\epsilon]$ for all sufficiently small $\epsilon$, 
since $u_n^-(\mu,b) = u^-(\mu,b)$ and $u^-(\mu,b) \neq 0$ for $\mu$ in a neighbourhood of $\lambda$. Hence by the
intermediate value theorem applied to $m_n^+(\mu)-m_n^-(\mu)$, 
 there exists $\lm_n\in (\mu_n^{(1)},\mu_n^{(2)})$ with 
\begin{equation}\label{eq:robin1}
\frac{(u_n^+)'(\lm_n,b)}{(u_n^+)(\lm_n,b)}=\frac{(u_n^-)'(\lm_n,b)}{(u_n^-)(\lm_n,b)}.
\end{equation}
The function
\[ u_n(x) := \left\{ \begin{array}{ll} u_n^-(\lambda_n,x), & x < b, \\
 u_n^+(\lambda_n,x)\frac{u_n^-(\lambda_n,b)}{u_n^+(\lambda_n,b)}, & x\geq b,
 \end{array}\right. \]
 is therefore an eigenfunction of $T_n-\lm_n J_n$ with eigenvalue $0$. Since $\lambda_n$ is $\epsilon$-close to $\lambda$ and $\epsilon$ can be
 arbitrarily small, we have proved the spectral inclusion.
\end{proof}

\section{Operator spectral problem transformed into pencil problem}\label{sec.op}
As seen in Example~\ref{ex.op.as.pencil}, the set of spectral pollution might be smaller (even empty) when the operator eigenvalue problem $Tx=\lm x$ is transformed into the pencil eigenvalue problem $Ax=\lm Bx$ with $A:=BT$. 
In this section we explore this idea further. In the same way we establish tight enclosures of the spectrum of $T$ by taking the intersection of numerical ranges $W(BT,B)$ for suitable $B$.

\subsection{Abstract results for operators and diagonal $2\times 2$ block operator matrices}

The following result gives a (not necessarily connected) spectral enclosure in terms of numerical ranges.
Note that the sets in~\eqref{eq.W.pencil.spec} coincide if $\sigma(T)=\sigma_{\rm app}(T)$.

\begin{Theorem}\label{thm.intersection}
Let $T\in C(H)$.
\begin{enumerate}
\item[\rm i)]
The approximate point spectrum and spectrum are related to numerical ranges by 
\begin{equation}\label{eq.W.pencil.spec}
\sigma_{\rm app}(T)\subseteq  \underset{B\in L(H)}{\bigcap}W(BT,B)\subseteq \underset{B\in L(H)\atop 0\in\rho(B)}{\bigcap}W(BT,B)\subseteq \sigma(T),
\end{equation}
and
\begin{equation}\label{eq.We.pencil.spec}
\sigma_{e}(T)= \underset{B\in L(H)}{\bigcap}W_e(BT,B).
\end{equation}

\item[\rm ii)]
For $\Lambda\subseteq L(H)\backslash\{0\}$ let  
$$ \Omega\subseteq\C\backslash \underset{B\in \Lambda}{\bigcap}W(BT,B) $$
be a connected set. If $\Omega\cap\rho(T)\neq\emptyset$, then $\Omega\subseteq\rho(T)$ and
\begin{align*}
\|(T-\lm)^{-1}\|
&\leq\inf_{B\in\Lambda}\frac{\|B\|}{{\rm dist}(0,W(B))\,{\rm dist}(\lm,W(BT,B))},
\quad \lm\in\Omega.
\end{align*}
\end{enumerate}
\end{Theorem}

\begin{proof}
i)
To prove the first inclusion in~\eqref{eq.W.pencil.spec}, let $\lm\in\sigma_{\rm app}(T)$. Then there exists a normalised sequence $(x_n)_{n\in\N}\subset\dom(T)$
such that $\|(T-\lm)x_n\|\to 0$. Let $B\in L(H)$. Then $|\langle (BT-\lm B)x_n,x_n\rangle|\leq \|B\| \|(T-\lm)x_n\|\to 0$ and hence $\lm\in W(BT,B)$.

The second inclusion in~\eqref{eq.W.pencil.spec} is trivial.

Now take $\lm\in\C$ such that $\lm\in W(BT,B)$ for all $B\in L(H)$ with $0\in\rho(B)$.
We use the following well-known equivalence, which is a consequence of von Neumann's theorem \cite[Theorem V.3.24]{kato}:
\begin{equation}\label{eq.equiv.spec}
\lm\in\sigma(T)\quad\Longleftrightarrow\quad 0\in\sigma\big(((T-\lm)^*(T-\lm))^{1/2}\big)\cup\sigma\big(((T-\lm)(T-\lm)^*)^{1/2}\big).
\end{equation}
With $|T-\lm|:=((T-\lm)^*(T-\lm))^{1/2 }$ let $T-\lm=U|T-\lm|$ be the polar decomposition of $T-\lm$.
By~\cite[Section VI.2.7]{kato}, the operator $|T-\lm|$ is selfadjoint and non-negative, and $U:\ran(|T-\lm|)\to\ran(T-\lm)$ is isometric.
By continuity, it can be extended to an isometric operator on $\overline{\ran(|T-\lm|)}$, and then further extended to a bounded operator $U\in L(H)$ by setting $Ux:=0$, $x\in\ran(|T-\lm|)^{\perp}$. Then $U^*\in L(H)$ with $U^*Ux=x$ for all $x\in\ran(|T-\lm|)$.
We set $B:=U^*\in L(H)$.
Then 
\beq BT-\lm B=U^*(T-\lm)=U^*U|T-\lm|=|T-\lm|.\label{eq.Uadjoint}\eeq
 Assume that $\lm\in\rho(T)$. Then the equivalence~\eqref{eq.equiv.spec} yields $0\in\rho(|T-\lm|)$. Therefore $\ran(|T-\lm|)=H=\ran(T-\lm)$ and hence $U\in L(H)$ is unitary.
Thus $0\in\rho(B)$.
Since $\lm\in W(BT,B)$ by the choice of $\lm$,~\eqref{eq.Uadjoint} implies that
$$0\in \overline{W(|T-\lm|)}={\rm conv}\,\sigma(|T-\lm|)\subseteq [0,\infty).$$
 Therefore $0\in\sigma(|T-\lm|)$ and hence~\eqref{eq.equiv.spec} implies $\lm\in\sigma(T)$. 

The inclusion $\sigma_{e}(T)\subseteq \underset{B\in L(H)}{\bigcap}W_e(BT,B)$ is shown analogously as the first inclusion in~\eqref{eq.W.pencil.spec}; we use in addition that the sequence $(x_n)_{n\in\N}$ converges weakly to~$0$.

To prove the reverse inclusion (and thus equality in \eqref{eq.We.pencil.spec}), we choose $\lm\in\C$ such that $\lm \in W_e(BT,B)$ for all $B\in L(H)$. We proceed analogously as above (i.e.\ we use the polar decomposition of $T-\lm$ and set $B:=U^*\in L(H)$) to arrive at~\eqref{eq.Uadjoint}.
Since $\lm\in W_e(BT,B)$ by the choice of $\lm$, and using Theorem \ref{Thm2.11}, we obtain 
$$0\in W_e(|T-\lm|)={\rm conv}\,\widehat{\sigma}_e(|T-\lm|)\backslash\{\infty\}\subseteq [0,\infty).$$
 Therefore $0\in\sigma_e(|T-\lm|)$
and hence $0\in \sigma_e(|T-\lm|^2)\subseteq W_e(|T-\lm|^2)$.
This implies the existence of a sequence $(x_n)_{n\in\N}\subset\dom(|T-\lm|^2)\subseteq\dom(|T-\lm|)=\dom(T-\lm)$ with $\|x_n\|=1$, $x_n\stackrel{w}{\to} 0$ and $$\|(T-\lm)x_n\|^2=\langle |T-\lm|^2x_n,x_n\rangle\tolong 0, \quad n\to\infty.$$
Therefore, $\lm\in\sigma_e(T)$.

ii) The claim follows from the first inclusion in~\eqref{eq.W.pencil.spec} and Theorem~\ref{thm.dist.pencil} and its proof;
note that $\|(T-\lm)x\|\geq \|(BT-\lm B)x\| \|B\|^{-1}$ and hence $\|(T-\lm)^{-1}\|\leq \|B\| \|(BT-\lm B)^{-1}\|$ for $\lm\in\Omega$ and $B\in\Lambda$.
\end{proof}

The following result can be used for approximations of selfadjoint operators to remove spurious eigenvalues in gaps of the (essential) spectrum. Note that if $P$ is the spectral projection $\chi_{(-\infty,\gamma]}(T)$ for some $\gamma\in\R$, then $T$ admits a diagonal block operator representation as in Theorem~\ref{thm.2x2} below; however, in general $\chi_{(-\infty,\gamma]}(T)$ is unknown.

\begin{Proposition}\label{prop.decomp}
Let $T\in C(H)$ be selfadjoint. Let $P$ be an orthogonal projection in $H$ with $\ran(P)\subseteq\dom(T)$ and define $B:=I-2P=-P+(I-P)$. 
Assume that
\begin{equation}\label{eq.decomp.semibdd}
a:=\sup\, W(T|_{\ran(P)})<\infty, \quad b:=\inf\, W(T|_{\ran(P)^{\perp}\cap\dom(T)})>-\infty.
\end{equation}
\begin{enumerate}
\item[\rm i)] If $a<b$, then 
\begin{align*}
w_e(BT,B)&=W_e(BT,B)\subseteq \overline{w(BT,B)}=W(BT,B)\\
&\subseteq \big\{\lm\in\C:\,\re\,\lm\in (-\infty,a]\,\dot\cup\, [b,\infty)\big\}.
\end{align*}
\item[\rm ii)] If $W_e(T|_{\ran(P)})=\emptyset$, then 
\begin{align*}
w_e(BT,B)&=W_e(BT,B)\\
&\subseteq \big\{\lm\in\C:\,\re\,\lm\geq \min W_e(T|_{\ran(P)^{\perp}\cap\dom(T)})\big\};
\end{align*}
if, in addition, $W_e(T|_{\ran(P)^{\perp}\cap\dom(T)})=\emptyset$, then 
$$w_e(BT,B)=W_e(BT,B)=\emptyset.$$
\item[\rm iii)] If $W_e(T|_{\ran(P)})\neq \emptyset$ and $W_e(T|_{\ran(P)^{\perp}\cap\dom(T)})\neq\emptyset$, define
$$a_e:=\max\, W_e(T|_{\ran(P)})\leq a, \quad b_e:=\min\, W_e(T|_{\ran(P)^{\perp}\cap\dom(T)})\geq b.$$
If $a_e<b_e$, then
$$w_e(BT,B)=W_e(BT,B)\subseteq \big\{\lm\in\C:\,\re\,\lm\in (-\infty,a_e]\,\dot\cup\, [b_e,\infty)\big\}.$$
\end{enumerate}
\end{Proposition}

\begin{proof}
i)
Let $\lm \in W(BT,B)$.
There exists a normalised sequence $(x_n)_{n\in\N}\subset \dom(T)$ such that $\langle (BT-\lm B)x_n,x_n\rangle\to 0$.
Define $u_n:=Px_n$, $v_n:=(I-P)x_n$ for all $n\in\N$. 
Note that 
\begin{align*}
&\langle (BT-\lm B)x_n,x_n\rangle\\
&=-\langle (T-\re\,\lm)u_n,u_n\rangle-\langle (T-\re\,\lm)v_n,u_n\rangle+\langle (T-\re\, \lm)u_n,v_n\rangle\\
&\quad +\langle (T-\re\,\lm)v_n,v_n\rangle
-\I\,\im\,\lm\,\langle B x_n,x_n\rangle \\
&=-\langle (T-\re\,\lm)u_n,u_n\rangle+2\I\, \im \langle (T-\re\,\lm)u_n,v_n\rangle+\langle (T-\re\,\lm)v_n,v_n\rangle\\
&\quad -\I\,\im\,\lm\,\langle B x_n,x_n\rangle.
\end{align*}
Taking the real part on both sides of the latter equation, and using that $B$ is selfadjoint, yields $-\langle (T-\re\,\lm)u_n,u_n\rangle+\langle (T-\re\,\lm)v_n,v_n\rangle\to 0$.
Define the diagonal block operator matrix
$$\cA:={\rm diag}\,(-P(T-\re\,\lm), (I-P)(T-\re\,\lm)) \quad \text{in}\quad \ran(P)\oplus (\ran(P)^{\perp}\cap\dom(T)).$$
Then $\| (u_n,v_n)^t\|=1$ and  $\langle \cA (u_n,v_n)^t,(u_n,v_n)^t\rangle \to 0$.
Therefore
$$0\in \overline{W(\cA)}={\rm conv}\, \big(-(\overline{W(T|_{\ran(P)})}-\re\,\lm)\cup (\overline{W(T|_{\ran(P)^{\perp}\cap\dom(T)})}-\re\,\lm)\big).$$
Then it is easy to see that $\re\,\lm\in (-\infty,a]\,\dot\cup\, [b,\infty)$. In particular, $W(BT,B)\neq\C$.
The equalities $w_e(BT,B)=W_e(BT,B)$ and $\overline{w(BT,B)}=W(BT,B)$ follow from Propositions~\ref{prop.ess.num.range}~ii),~\ref{prop.num.range}~i) and Remark~\ref{rem.num.range}~iii). 

ii)
If $W_e(BT,B)\neq\emptyset$, let $\lm\in W_e(BT,B)$ and proceed as in i). Note that, in addition, we have $x_n\stackrel{w}{\to} 0$ and hence $u_n\stackrel{w}{\to}0$ and $v_n\stackrel{w}{\to}0$.
We obtain 
\begin{equation}\label{eq.sum.conv}
-(\langle Tu_n,u_n\rangle-\re\,\lm\|u_n\|^2)+\langle Tv_n,v_n\rangle-\re\,\lm\|v_n\|^2\tolong 0.
\end{equation} 
Since $\big(-(\langle Tu_n,u_n\rangle-\re\,\lm\|u_n\|^2)\big)_{n\in\N}$ and $\big(\langle Tv_n,v_n\rangle-\re\,\lm\|v_n\|^2\big)_{n\in\N}$ are both bounded from below by~\eqref{eq.decomp.semibdd}, 
the convergence in \eqref{eq.sum.conv} implies that  both sequences are bounded. Therefore there exist an infinite subset $I\subseteq\N$ and $c\in\R$ such that 
$$\langle Tu_n,u_n\rangle-\re\,\lm\|u_n\|^2\tolong c, \quad \langle Tv_n,v_n\rangle-\re\,\lm\|v_n\|^2\tolong c, \quad n\in I, \quad n\to\infty.$$
The assumption $W_e(T|_{\ran(P)})=\emptyset$ implies that $u_n\to 0$
and hence $\|v_n\|\to 1$ as $n\in I$, $n\to\infty$.
Then $\widehat v_n:=v_n/\|v_n\|$ satisfies 
$$\widehat v_n\stackrel{w}{\tolong}0, \quad \langle T\widehat v_n,\widehat v_n\rangle-\re\,\lm\tolong c, \quad n\in I, \quad n\to\infty.$$
Thus $\re\,\lm+c \in  W_e(T|_{\ran(P)^{\perp}\cap\dom(T)})$; in particular, $W_e(T|_{\ran(P)^{\perp}\cap\dom(T)})\neq \emptyset$.
Assume that $c\neq 0$.
Then there exists an infinite subset $\widehat I\subseteq I$ such that $u_n\neq 0$ for all $n\in\widehat I$. Moreover, we have
$$\frac{\langle T u_n,u_n\rangle}{\|u_n\|^2}\tolong \begin{cases} \infty, & c>0,\\ -\infty, & c<0,\end{cases}\qquad n\in \widehat I, \quad n\to\infty.$$
By the assumption~\eqref{eq.decomp.semibdd}, we conclude $c< 0$.
Hence $\re\,\lm\geq \min W_e(T|_{\ran(P)^{\perp}\cap\dom(T)})$.
So we proved in particular that $W_e(BT,B)\neq\C$. Now  Proposition~\ref{prop.ess.num.range}~ii) implies that $W_e(BT,B)=w_e(BT,B)$.

iii)
Let $\lm\in W_e(BT,B)$ and proceed as in ii).
There exist an infinite subset $J\subseteq I$ and $\alpha\in [0,1]$ such that
$$\|u_n\|^2\tolong \alpha, \quad \|v_n\|^2\tolong 1-\alpha, \quad n\in J, \quad n\to\infty.$$
If $\alpha=0$, the arguments in ii) imply $\re\,\lm\geq b_e$. Analogously $\alpha=1$  yields $\re\,\lm\leq a_e$. 
It is left to consider the case $\alpha\in (0,1)$.
Then, for $n\in J$ sufficiently large, $u_n\neq 0$ and $v_n\neq 0$ and hence, in the limit $n\to\infty$, $\widehat u_n:=u_n/\|u_n\|$, $\widehat v_n:=v_n/\|v_n\|$ satisfy
$$\widehat u_n\stackrel{w}{\tolong}0, \quad \widehat v_n\stackrel{w}{\tolong}0, \quad  \langle T\widehat u_n,\widehat u_n\rangle-\re\,\lm\tolong \frac{c}{\alpha}, \quad \langle T\widehat v_n,\widehat v_n\rangle -\re\,\lm\tolong \frac{c}{1-\alpha}.$$
Then $$\re\,\lm\in \left(-\frac{c}{\alpha}+W_e(T|_{\ran(P)})\right)\cap\left(-\frac{c}{1-\alpha}+W_e(T|_{\ran(P)^{\perp}\cap\dom(T)})\right).$$
By the hypothesis $a_e<b_e$, this is not possible if $c=0$. If however $c\neq 0$, then  $c/\alpha$ and $c/(1-\alpha)$ have the same sign, and so we obtain $\re\,\lm<  a_e$ (if $c> 0$) or $\re\,\lm>b_e$ (if $c<0$).

The equality $W_e(BT,B)=w_e(BT,B)$ follows analogously as in ii).
\end{proof}

In the next result $T$ may be non-selfadjoint, but we assume that it admits a diagonal block operator representation.

\begin{Theorem}\label{thm.2x2}
Let $H_1$, $H_2$ be two infinite-dimensional Hilbert spaces, and let $\cT={\rm diag}(T_1,T_2)$ in $H_1\oplus H_2$. 
Assume that 
$$a:=\sup\,\re\, W(T_1)<\infty, \quad  b:=\inf\, \re\,W(T_2)>-\infty.$$
Define $\cB:={\rm diag}(-I,I)$ in $H_1\oplus H_2$.

\begin{enumerate}
\item[\rm i)]
If $a<b$, then 
\begin{align*}
w_e(\cB\cT,\cB)&=W_e(\cB\cT,\cB)\subseteq \overline{w(\cB\cT,\cB)}=W(\cB\cT,\cB)\\
&\subseteq \{\lm\in\C:\,\re\,\lm\in (-\infty,a]\,\dot\cup\,[b,\infty)\}.
\end{align*}
If, in addition, $\cT$ is selfadjoint, then
\begin{align*}
w_e(\cB\cT,\cB)&=W_e(\cB\cT,\cB)\\
&={\rm conv}(\widehat\sigma_e(T_1))\backslash\{-\infty\}\,\dot\cup\, {\rm conv}(\widehat\sigma_e(T_2))\backslash\{\infty\},\\
\overline{w(\cB\cT,\cB)}&=W(\cB\cT,\cB)={\rm conv}\,\sigma(T_1)\,\dot\cup\, {\rm conv}\,\sigma(T_2).
\end{align*}
\item[\rm ii)]
Assume that at least one of $-T_1$, $T_2$ is sectorial.
If $W_e(T_1)=W_e(T_2)=\emptyset$, then $$w_e(\cB\cT,\cB)=W_e(\cB\cT,\cB)=\emptyset.$$
If $W_e(T_1)=\emptyset$ and $W_e(T_2)\neq \emptyset$, let $\gamma\in\C$ and $-\pi/2\leq \theta_-\leq \theta_+\leq \pi/2$ be such that
$$ W(-T_1-\gamma)\subseteq \{\lm\in\C:\,\arg(\lm)\in [\theta_-,\theta_+]\}.$$
Then
\begin{align*}
w_e(\cB\cT,\cB)&=W_e(\cB\cT,\cB)\\
&\subseteq \{\lm_1+\lm_2:\,\arg(\lm_1)\in [\theta_-,\theta_+],\,\lm_2\in W_e(T_2)\}.
\end{align*}
 If $W_e(T_1)\neq \emptyset$ and $W_e(T_2)\neq\emptyset$, define
$$a_e:=\max\,\re\, W_e(T_1)\leq a, \quad b_e:=\min\,\re\, W_e(T_2)\geq b.$$
If $a_e<b_e$, then
$$w_e(\cB\cT,\cB)=W_e(\cB\cT,\cB)\subseteq \big\{\lm\in\C:\,\re\,\lm\in (-\infty,a_e]\,\dot\cup\, [b_e,\infty)\big\}.$$

\item[\rm iii)]
Assume that $\cT=\mathcal U+\I \mathcal V$ with a selfadjoint $\mathcal U={\rm diag} (U_1,U_2)$ and symmetric $\mathcal V={\rm diag}(V_1,V_2)$ such that
$U_1\leq a<0<b\leq U_2$.
Let $\mathcal K$ be a block operator matrix in $H_1\oplus H_2$ with $\dom(\cT)\subseteq\dom(\mathcal K)$ such that  $(\cB\mathcal U)^{-\frac 1 2}\cB\mathcal K (\cB\mathcal U)^{-\frac 1 2}$ is compact.
Then 
$$W_e(\cB(\cT+\mathcal K),\cB)=w_e(\cB(\cT+\mathcal K),\cB)=w_e(\cB\cT,\cB)=W_e(\cB\cT,\cB).$$

\end{enumerate}
\end{Theorem}

\begin{proof}
i)
Note that for $x_n=(u_n,v_n)^t\in\dom(\cT)$ we have
$$\re \langle (\cB\cT-\lm \cB)x_n,x_n\rangle = - (\re\langle T_1 u_n,u_n\rangle -\re\,\lm \|u_n\|^2)+\re\langle T_2 v_n,v_n\rangle -\re\,\lm \|v_n\|^2.$$
Now the first claim follows in a similar way as in Proposition~\ref{prop.decomp}~i).

For a selfadjoint $\cT$ we have $w(\cB\cT,\cB)\subset\R$.
Let $\lm\in\R$. Then the assertion follows immediately from
\begin{align*}
\lm\in W(\cB\cT,\cB)\quad &\Longleftrightarrow\quad 0\in \overline{W(\cB\cT-\lm \cB)}={\rm conv}\,\sigma(\cB\cT-\lm \cB), \\
\lm\in W_e(\cB\cT, \cB)\quad &\Longleftrightarrow\quad 0\in W_e(\cB\cT-\lm \cB)={\rm conv}(\widehat\sigma_e(\cB\cT-\lm \cB))\backslash\{\infty\},
\end{align*}
where we used Theorem \ref{Thm2.11} in the last equality.

ii)
If $W_e(T_1)=W_e(T_2)=\emptyset$, we
 use that the sectoriality assumptions imply $W_e(\cB\cT)={\rm conv}(W_e(-T_1)\cup W_e(T_2))=\emptyset$; then the claim follows from Proposition~\ref{prop.We.empty}.

If $W_e(T_1)=\emptyset$ and $W_e(T_2)\neq \emptyset$, we proceed as in the proof of Proposition~\ref{prop.decomp}, part ii), see \eqref{eq.sum.conv}. Since the numerical ranges of both $-T_1,T_2$ have real parts bounded from below, and one of the operators  is sectorial, we obtain that both sequences
$\big(-(\langle Tu_n,u_n\rangle-\re\,\lm\|u_n\|^2)\big)_{n\in\N}$ and $\big(\langle Tv_n,v_n\rangle-\re\,\lm\|v_n\|^2\big)_{n\in\N}$ are bounded and hence admit a convergence subsequence. 
Now the claim follows from 
$$\|u_n\|\leq 1,\quad u_n\stackrel{w}{\tolong}0, \quad -\langle T_1 u_n,u_n\rangle \tolong \widetilde c\in\C \quad \Longrightarrow u_n\tolong 0, \quad \arg(\widetilde c)\in [\theta_-,\theta_+].$$

If $W_e(T_1)\neq \emptyset$ and $W_e(T_2)\neq \emptyset$, we proceed as in Proposition~\ref{prop.decomp}, part iii), using again the sectoriality assumption to prove that $\big(-(\langle Tu_n,u_n\rangle-\re\,\lm\|u_n\|^2)\big)_{n\in\N}$ and $\big(\langle Tv_n,v_n\rangle-\re\,\lm\|v_n\|^2\big)_{n\in\N}$ are bounded. 

iii)
The operator  $\cB\mathcal U$ is selfadjoint with $\cB\mathcal U \geq \min\{|a|,b\}>0$. 
Theorem~\ref{thm.pert.rel.comp}~(b) implies $w_e(\cB\cT,\cB)=w_e(\cB(\cT+\mathcal K),\cB)$. Since $0\notin W_e(\cB\cT)=W_e(\cB(\cT+\mathcal K))$ by Theorem \ref{thmSqrt}, the remaining identities follow from Proposition~\ref{prop.ess.num.range} ii).
\end{proof}

\subsection{Application to Schr\"odinger operators}

As an application of Theorem~\ref{thm.2x2}~iii), we study perturbed periodic Schr\"odinger operators.

Let $T=-\Delta+V_{\rm per}$ be a selfadjoint Schr\"odinger operator in $L^2(\R^d)$ with a real-valued  periodic potential $V_{\rm per}$.
Let $W$ be another function.
In~\cite[Theorem 2.3]{lewin-sere}, conditions on $V_{\rm per}$ and $W$ were established that guarantee that for a spectral gap $(a,b)\subset (0,\infty)$ with centre $\gamma=\frac{a+b}{2}$ and spectral projection $P_{\gamma}:=\chi_{(-\infty,\gamma]}(T)$,
no spectral pollution occurs in $(a,b)$ for a projection method $(A_n)_{n\in\N}$ where $A_n$ is the compression of $T+W$ to a subspace 
\beq\label{eq.decomp.spaces}
M_n^-\oplus M_n^+\subset \ran(P_{\gamma})\oplus (\ran(P_{\gamma})^{\perp}\cap\dom(T+W)).
\eeq
Now we consider $W$ that may be complex-valued. 

\begin{Theorem}
Define $B:=I-2P_{\gamma}=-P_{\gamma}+(I-P_{\gamma})$ and $T_{\gamma}:=T-\gamma$. We assume that $W$ is such that
$\dom(T)\subseteq\dom(W)=\{f\in L^2(\R^d):\,Wf\in L^2(\R^d)\}$ and
 $(BT_{\gamma})^{-\frac 1 2}BW (BT_{\gamma})^{-\frac 1 2}$ is compact.
Then  
\beq\label{eq.schroed.W}
W_e(B(T_{\gamma}+W),B)=W_e(BT_{\gamma},B)=(-\infty,a-\gamma]\,\dot\cup\, [b-\gamma,\infty).
\eeq
Hence no spectral pollution occurs in $(a,b)\cup \C\backslash\R$ if we compress $T+W$ to subspaces satisfying~\eqref{eq.decomp.spaces}.
\end{Theorem}

\begin{proof}
The identities in~\eqref{eq.schroed.W} follow from Theorem~\ref{thm.2x2}, claims~iii) and~i).
By Theorem~\ref{thm.pencil.sigma}~ii), a projection method of the pencil $\lm\mapsto B(T+W-\lm)$ does not pollute in $(a,b)\cup \C\backslash\R$. Note that the eigenvalues of the truncated pencil coincide with those of the operator $T+W$ compressed to a subspace in~\eqref{eq.decomp.spaces}.
\end{proof}

Next we consider Schr\"odinger operators with diverging potentials.
 
\begin{Theorem}\label{thm.schroedinger}
Consider the differential expression $\tau:=-\rd^2/\rd x^2+V$ with a potential $V:\R\to\C$ satisfying $V\in L_{\rm loc}^2(\R)$ and $$|V(x)|\tolong \infty, \quad |x|\to\infty.$$
For some $\theta\in [0,\pi/2)$ define the sector $\cS_{\theta}:=\{\lm\in\C:\,|\arg(\lm)|\leq\theta\}$.
Assume that there exist $-\infty<a< b<\infty$ and $\varphi_{\pm}\in (-\pi/2,\pi/2)$, $r> 0$ such that, for almost all $x\in\R$,
$$V(x)\in\begin{cases} \e^{\I\,\varphi_-}\cS_{\theta},&x\in(-\infty,a],\\ B_r(0), &x\in (a,b),\\ \e^{\I\,\varphi_+}\cS_{\theta}, &x\in [b,\infty).\end{cases}.$$
Let $T$ be the closure of the operator $T_0$, the minimal realisation of $\tau$ with $\dom(T_0):=C_0^{\infty}(\R)$.
Let $B$ be the bounded and continuous function
$$B(x):=\begin{cases}
\e^{-\I\,\varphi_-}, & x\in (-\infty,a],\\
\e^{-\I\,(t \varphi_-+(1-t)\varphi_+)}, & x\in (a,b),\,t=\frac{x-a}{b-a},\\
\e^{-\I\,\varphi_+}, & x\in [b,\infty).
\end{cases}$$
Then $W_e(BT)=\emptyset$ and thus $w_e(BT,B)=W_e(BT,B)=\emptyset$.
\end{Theorem}

\begin{proof}
First note that $|B|\equiv 1$ and $|\arg(B)|\leq\max\{|\varphi_-|,|\varphi_+|\}<\pi/2$.
Let $f\in\dom(T_0)=C_0^{\infty}(\R)$. Then, using integration by parts, we obtain
\begin{align*}
\langle BT f,f\rangle &=
\langle -B f''+BV f,f\rangle
=\langle B f',f'\rangle + \langle B' f',f\rangle +\langle B V f,f\rangle
=\sum_{i=1}^6 s_i[f]
\end{align*}
with
\begin{alignat*}{3}
s_1[f]&:=\e^{-\I\,\varphi_-}\int_{-\infty}^a |f'|^2 \,\rd x,\quad& 
s_2[f]&:=\int_{-\infty}^a \e^{-\I\,\varphi_-}  V |f|^2 \,\rd x,\\
s_3[f]&:= \e^{-\I\,\varphi_+}\int_b^{\infty} |f'|^2\,\rd x,\quad&
s_4[f]&:=\int_b^{\infty}  \e^{-\I\,\varphi_+} V |f|^2 \,\rd x , \\
s_5[f]&:=\int _a^b B V |f|^2\,\rd x, \quad&
s_6[f]&:=\int _a^b B |f'|^2 + B' f' \overline{f}\,\rd x.
\end{alignat*}
Notice that $|\arg(s_i[f])|\leq \max\{|\varphi_-|,|\varphi_+|,\theta\}<\pi/2$ for $i=1,\dots,4$, and $|s_5[f]|\leq r \|f\|^2$.
Moreover, for an arbitrary $\eps>0$,
$$\bigg| \int _a^b B' f' \overline{f}\,\rd x\bigg|
\leq \frac{\|B'\|_{\infty}}{2}\left(\eps\int_a^b |f'|^2\,\rd x + \frac{\|f\|^2}{\eps}\right).$$
By choosing $\eps>0$ sufficiently small, we see that $s_6$ is sectorial as well.

Now let $(f_n)_{n\in\N}\subset\dom(T_0)$ with $\|f_n\|=1$ and such that $(|\langle BT f_n,f_n\rangle|)_{n\in\N}$ is bounded.
Since $s_i$, $i=1,\dots,6$, are sectorial, we conclude that $(s_i[f_n])_{n\in\N}$, $i=1,\dots,6$, are bounded, and hence $(\|f_n'\|^2)_{n\in\N}$ and $(\langle |V|f_n,f_n\rangle)_{n\in\N}$ are bounded;
to prove the latter, we use 
\begin{align*}
\langle |V|f_n,f_n\rangle
&\leq r +\int_{-\infty}^a  \frac{1}{\cos\theta} \, \underbrace{\re\big(\e^{-\I\,\varphi_-}  V\big)}_{\geq 0} |f_n|^2 \,\rd x\\
&\quad+\int_b^{\infty}   \frac{1}{\cos\theta} \, \underbrace{\re\big(\e^{-\I\,\varphi_+} V\big)}_{\geq 0} |f_n|^2 \,\rd x\\
&= r + \frac{1}{\cos\theta} \, (\re\,s_2[f_n]+\re\,s_4[f_n]).
\end{align*}
By Rellich's criterion~\cite[Theorem XIII.65]{reedsimon}, there exists a subsequence of $(f_n)_{n\in\N}$ that is convergent in $L^2(\R)$.
Hence, if $f_n\w f$, then $f_n\to f$. Since the $f_n$ are normalised, we obtain $f\neq 0$. Therefore $W_e(BT_0)=\emptyset$, and using that
$T$ is the closure of $T_0$, we arrive at
 $W_e(BT)=\emptyset$. Then Proposition~\ref{prop.We.empty} yields $w_e(BT,B)=W_e(BT,B)=\emptyset$.
\end{proof}

\begin{Example}\label{ex.schroedinger}
One may verify that the assumptions of Theorem~\ref{thm.schroedinger} are satisfied for instance
\begin{enumerate}
\item[\rm i)]
for a $\mathcal{PT}$-symmetric Schr\"odinger operator with non-real potential 
$$V(x):=\sum_{k=0}^m\left( \alpha_k x^{2k}+\I\,\beta_k x^{2k+1}\right), \quad x\in\R,$$ where $\alpha_k,\,\beta_k\in\R$ for all $k$, and $\alpha_m>0$ or $\beta_m\neq 0$.

\item[\rm ii)]
for a Schr\"odinger operator with potential $V(x):=\e^{\I\,\vartheta_-}|x|$ on $(-\infty,0]$ and $V(x):=\e^{\I\,\vartheta_+}x$ on $[0,\infty)$ for angles $\vartheta_{\pm}\in (-\pi,\pi)$.
\end{enumerate}
\end{Example}

\begin{Remark}
Using the proof of Theorem~\ref{thm.schroedinger}, one can show that $B(T-\lm)$ has compact resolvent and hence its spectrum is discrete.
By $W_e(BT,B)=\emptyset$ and 
Theorem~\ref{thm.trunc}~ii),
 interval truncation of the pencil $\lm\mapsto B(T-\lm)$ with Dirichlet boundary conditions at the endpoints 
is spectrally exact.
Note that the eigenvalues of the approximations are the same as of the truncations of $T$ (since $B$ is bounded and boundedly invertible).
In this way one can prove spectral exactness
 for the interval truncation process of some Schr\"odinger operators that are not covered by~\cite{BST};
for instance in Example~\ref{ex.schroedinger}~ii),~\cite{BST} can only be applied for  angles $\vartheta_{\pm}\in (-3\pi/4,3\pi/4)$ since the negative real part of the potential needs to be bounded by the imaginary part, with relative bound $<1$.
\end{Remark}

\subsection{Abstract results for non-diagonal $2\times 2$ block operator matrices}

In the following we study a block operator matrix in $H_1\oplus H_2$,
$$\mathcal T:=\begin{pmatrix} A & B \\ C & D\end{pmatrix}, \quad \dom(\mathcal T):=(\dom(A)\cap\dom(C))\oplus(\dom(B)\cap\dom(D)).$$
Throughout this subsection, $A$, $B$, $C$, $D$ refers to the above entries of $\cT$.

If $\cT$ is closable, then Theorem~\ref{thm.intersection}~i) yields
\begin{equation}\label{eq.enclosure}
\begin{aligned}
\sigma_{\rm app}(\overline{\cT})
&=\sigma_{\rm app}(\cT)\subseteq
 \underset{a,d\in\C}{\bigcap}W({\rm diag}(a,d)\cT,{\rm diag}(a,d)),\\
\sigma_e(\overline{\cT})
&=\sigma_e(\cT)\subseteq 
\underset{a,d\in\C}{\bigcap} W_e({\rm diag}(a,d)\cT,{\rm diag}(a,d)).
\end{aligned}
\end{equation}
We compare this (in general non-convex) spectral enclosures for $\cT$  with the quadratic numerical range which was introduced in~\cite{Langer-Tretter-98} (see also~\cite{tretter}).
We use the notation
$$\cT_{x,y}:=\bmat \langle Ax,x\rangle & \langle By,x\rangle \\ \langle Cx,y\rangle & \langle T y,y\rangle\emat, \quad (x,y)^t\in\dom(\cT).$$

\begin{Theorem}\label{thm.2x2.intersec}
\begin{enumerate}
\item[\rm i)]
The quadratic numerical range \cite{Langer-Tretter-98} is contained in the above spectral enclosure,
\begin{align*}
W^2(\cT)&:=\underset{{(x,y)^t\in{\rm dom}(\cT)}\atop {\|x\|=\|y\|=1}}{\bigcup}\,\sigma(\cT_{x,y})
\subseteq\underset{a,d\in\C}{\bigcap}W({\rm diag}(a,d)\cT,{\rm diag}(a,d)).
\end{align*}

\item[\rm ii)]
For $\Lambda\subseteq\C^2$ let
$$\Omega\subseteq\C\backslash\underset{(a,d)^t\in\Lambda}{\bigcap} W({\rm diag}(a,d)\cT,{\rm diag}(a,d))$$
be a connected set with $\Omega\cap\rho(\cT)\neq\emptyset$.
Then $\Omega\subseteq\rho(\cT)$ and, for any $\lm\in\Omega$,
\[
\|(\cT-\lm)^{-1}\|
\leq \inf_{(a,d)^t\in\Lambda}\frac{\max\{|a|,|d|\}}{{\rm dist}\big(0,{\rm conv}\{a,d\}\big)\,{\rm dist}\big(\lm,W({\rm diag}(a,d)\cT,{\rm diag}(a,d)\big)\Big)}.
\]
\end{enumerate}
\end{Theorem}

\begin{proof}
i)  Let  $(x,y)^t\in\dom(\cT)$ with $\|x\|=\|y\|=1$. Then, by Theorem~\ref{thm.intersection}~i),
$$\sigma(\cT_{x,y})
= \underset{a,b,c,d\in\C}{\bigcap} W\left(\bmat a & b\\ c &d\emat \cT_{x,y},\bmat a & b \\ c & d\emat\right)
\subseteq  \underset{a,d\in\C}{\bigcap}\, W\big({\rm diag}(a,d)\cT_{x,y},{\rm diag}(a,d)\big).$$
Using that the numerical range of a finite matrix is closed and with 
$$W\left({\rm diag}(a,d)(\cT_{x,y}-\lm)\right)=W\left({\rm diag}(a,d)(\cT-\lm))_{x,y}\right)\subseteq W({\rm diag}(a,d)(\cT-\lm)),$$
 we obtain
\begin{align*}
\underset{{(x,y)^t\in{\rm dom}(\cT)}\atop {\|x\|=\|y\|=1}}{\bigcup}\,\sigma(\cT_{x,y})
&\subseteq \underset{a,d\in\C}{\bigcap}\,\underset{{(x,y)^t\in{\rm dom}(\cT)}\atop {\|x\|=\|y\|=1}}{\bigcup}\, W\big({\rm diag}(a,d)\cT_{x,y},{\rm diag}(a,d)\big)\\
&= \underset{a,d\in\C}{\bigcap}\,\underset{{(x,y)^t\in{\rm dom}(\cT)}\atop {\|x\|=\|y\|=1}}{\bigcup}\, \{\lm\in\C:\,0\in W(({\rm diag}(a,d)(\cT-\lm))_{x,y})\}\\
&\subseteq \underset{a,d\in\C}{\bigcap}\,\left\{\lm\in\C:\,0\in W({\rm diag}(a,d)(\cT-\lm))\right\},
\end{align*}
which implies the claim.

ii) The claim follows from Theorem~\ref{thm.intersection}~ii).
\end{proof}

\begin{Remark}\label{rem.2x2}
\begin{enumerate}
\item[\rm i)]
The inclusion in claim~i) may be strict. As in~\cite[Example~2.5.14]{tretter}, let $A=C=D=0$ and let $B$ be bijective (and hence closed) with dense domain $\dom(B)\subsetneqq H_2$. Then $\cT$ is off-diagonally dominant of order~$0$, it is closed and $\sigma_{\rm app}(\cT)=\C$. Hence \eqref{eq.enclosure} implies 
$$ \underset{a,d\in\C}{\bigcap}W({\rm diag}(a,d)\cT,{\rm diag}(a,d))=\C.$$
However, we have $W^2(\cT)=\{0\}$; in particular, $\overline{W^2(\cT)}$ does not contain $\sigma_{\rm app}(\cT)$. 
Hence \eqref{eq.enclosure} can be used to enclose the approximate point spectrum also for $2\times 2$ operator matrices that are not diagonally dominant of order $0$ or off-diagonally dominant of order $0$ with $B$, $C$ boundedly invertible as assumed in~\cite[Theorems~2.5.10,~2.5.12]{tretter}. An application to Dirac operators is given in Theorem~\ref{thm.dirac}.

\item[\rm ii)]
Assume that $\dim H_1\geq 2$ and $\dim H_2\geq 2$.
Then~\cite[Theorem~2.5.4]{tretter} and claim~i) yield $$W(A)\cup W(D)\subseteq W^2(\cT)\subseteq  \underset{a,d\in\C}{\bigcap}W({\rm diag}(a,d)\cT,{\rm diag}(a,d)).$$
\end{enumerate}
\end{Remark}

In the following result we assume that $A$ controls the off-diagonal entries but $\mathcal T$ need not be diagonally dominant.
The motivation is to establish spectral enclosures for the Stokes-type operator in Theorem~\ref{thm.stokes} below.
In contrast to the previous result, here the multiplier is not a constant diagonal matrix, hence  Remark~\ref{rem.2x2}~ii) does not apply and the enclosure need not contain $W(D)$.

\begin{Theorem}\label{thm.distances}
Assume that $\dom(A)\cap\dom(C)\subseteq\dom(B^*)$ and that $\mathcal T$ is closable.
Let $\Lambda\subseteq \C\times [-\pi,\pi)$ contain all $(\lm,\varphi)^t$ with $\lm\in\rho(D)\backslash\overline{W(A)}$ and $\varphi\in [-\pi,\pi)$ such that 
\begin{equation}\label{eq.numrange}
r_{\lm,\varphi}:=\inf\,\re\big(\e^{\I\varphi}W(A-\lm)\big)>0,
\end{equation}
and suppose there exist $a_{\lm,\varphi},b_{\lm,\varphi},c_{\lm,\varphi},d_{\lm,\varphi}\geq 0$ satisfying, for all $x\in\dom(A)\cap\dom(C)$,
\begin{equation}\label{eq.relbounds}
\begin{aligned}
\|Cx\|^2&\leq a_{\lm,\varphi}\|x\|^2+b_{\lm,\varphi}\re\big(\e^{\I\varphi}\langle (A-\lm)x,x\rangle\big), \\
\|B^*x\|^2&\leq c_{\lm,\varphi}\|x\|^2+d_{\lm,\varphi}\re\big(\e^{\I\varphi}\langle (A-\lm)x,x\rangle\big).
\end{aligned}
\end{equation}
For $(\lm,\varphi)^t\in\Lambda$ define the bounded multiplier
$$\mathcal B_{\lm,\varphi}:=\begin{pmatrix} I & 0 \\ 0 & \e^{-\I\varphi}\eps_{\lm,\varphi} (D-\lm)^{-1}\end{pmatrix}, \quad \eps_{\lm,\varphi}:=\frac{1}{\left(b_{\lm,\varphi}+\frac{a_{\lm,\varphi}}{r_{\lm,\varphi}}\right)\|(D-\lm)^{-1}\|^{2}}.$$
Then
\begin{align*}
&\sigma_{\rm app}(\overline{\mathcal T})\subseteq \underset{(\lm,\varphi)^t\in\Lambda}{\bigcap}W(\mathcal B_{\lm,\varphi}\mathcal T, \mathcal B_{\lm,\varphi})\\
&\subseteq\C\backslash\left\{\lm\in\rho(D)\backslash\overline{W(A)}:\,\inf_{\varphi\in[-\pi,\pi)\atop (\lm,\varphi)^t\in\Lambda}\left(b_{\lm,\varphi}+\frac{a_{\lm,\varphi}}{r_{\lm,\varphi}}\right)\left(d_{\lm,\varphi}+\frac{c_{\lm,\varphi}}{r_{\lm,\varphi}}\right)< \frac{1}{\|(D-\lm)^{-1}\|^{2}}\right\}.
\end{align*}
\end{Theorem}

\begin{proof}
Let $(\lm,\varphi)^t\in\Lambda$.
We show that if 
\begin{equation}\label{eq.distances}
\left(b_{\lm,\varphi}+\frac{a_{\lm,\varphi}}{r_{\lm,\varphi}}\right)\left(d_{\lm,\varphi}+\frac{c_{\lm,\varphi}}{r_{\lm,\varphi}}\right)< \frac{1}{\|(D-\lm)^{-1}\|^{2}},
\end{equation}
then  $\lm\notin W(\mathcal B_{\lm,\varphi}\mathcal T, \mathcal B_{\lm,\varphi})$; then the claim follows using Theorem~\ref{thm.intersection}~i).

Since $\lm$ and $\varphi$ are fixed, in the following we drop the indices and simply write $r,a,b,c,d,\eps,\mathcal B$.
First we carefully choose two constants $\alpha,\beta$ and derive some preliminary estimates that will be used later on.
Define $$\beta:=
\frac{1}{2\left(b+\frac{a}{r}\right)}>0.$$
Then, using the inequality~\eqref{eq.distances}, we obtain
\begin{align*}
\frac{r-\beta\left(rb+a\right)}{rd+c}
&=\frac{1-\beta\left(b+\frac{a}{r}\right)}{d+\frac{c}{r}}
=\frac{1}{2\left(d+\frac{c}{r}\right)}
>\frac{1}{2\eps}
=\frac{1}{4\eps-\frac{\eps^2\|(D-\lm)^{-1}\|^2}{\beta}}.
\end{align*}
Choose $\alpha$ strictly in between the left hand side and the right hand side; note that $\alpha>0$.
Then we arrive at
\begin{equation}\label{eq.st}
\begin{aligned}
s&:=r(1-\alpha d-\beta b)-(\alpha c+\beta a)\\
&=r-\beta(rb+a)-\alpha(rd+c)>0, \\
t&:=\eps-\frac{1}{4\alpha}-\frac{\eps^2\|(D-\lm)^{-1}\|^2}{4\beta}
=\frac{1}{4}\left(4\eps-\frac{\eps^2\|(D-\lm)^{-1}\|^2}{\beta}-\frac{1}{\alpha}\right)>0.
\end{aligned}
\end{equation}
Note that, in particular, $1-\alpha d-\beta b>0$.

Now let $(x,y)^t\in\dom(\mathcal T)$. The relative boundedness assumption~\eqref{eq.relbounds} and the numerical range estimate~\eqref{eq.numrange} yield
\begin{align*}
&\re\left(\e^{\I\varphi}\langle \mathcal B(\mathcal T-\lm)(x,y)^t,(x,y)^t\rangle\right)\\
&=\re\left(\e^{\I\varphi}\langle (A-\lm)x,x\rangle+\e^{\I\varphi}\overline{\langle B^*x,y\rangle}+\eps\langle (D-\lm)^{-1}Cx,y\rangle+\eps\|y\|^2\right)\\
&\geq \re\left(\e^{\I\varphi}\langle (A-\lm)x,x\rangle\right)-\|B^*x\|\|y\|-\eps\|(D-\lm)^{-1}\| \|Cx\| \|y\|+\eps\|y\|^2\\
&\geq \re\left(\e^{\I\varphi}\langle (A-\lm)x,x\rangle\right)-\alpha \|B^*x\|^2-\beta\|Cx\|^2\\
&\quad +\left(\eps-\frac{1}{4\alpha}-\frac{\eps^2\|(D-\lm)^{-1}\|^2}{4\beta}\right)\|y\|^2\\
&\geq \re\left(\e^{\I\varphi}\langle (A-\lm)x,x\rangle\right)(1-\alpha d-\beta b)-(\alpha c+\beta a)\|x\|^2\\
&\quad +\left(\eps-\frac{1}{4\alpha}-\frac{\eps^2\|(D-\lm)^{-1}\|^2}{4\beta}\right)\|y\|^2\\
&\geq \big(r(1-\alpha d-\beta b)-(\alpha c+\beta a)\big)\|x\|^2+\left(\eps-\frac{1}{4\alpha}-\frac{\eps^2\|(D-\lm)^{-1}\|^2}{4\beta}\right)\|y\|^2\\
&=s\|x\|^2+t\|y\|^2.
\end{align*}
Since $s>0$ and $t>0$ by~\eqref{eq.st}, we arrive at $0\notin\overline{W(\mathcal B(\mathcal T-\lm))}$ and hence $\lm\notin W(\mathcal B\mathcal T, \mathcal B)$.
\end{proof}

In the next result we assume that $B$ is bounded and $W_e(A)=\emptyset$ and we take the intersection of the essential numerical ranges. The motivation is to study $\cT$ where $D$ is (the operator of multiplication with) a function and its resolvent therefore easily computable, and multiplication with the operator $\cB_{\lm}$ below commutes with a domain truncation process of $\cT$ as in Theorem~\ref{thm.trunc}. For an application to Hain-L\"ust-type operators see Theorem~\ref{thm.hain} below.

\begin{Theorem}\label{thm.intersection.We}
Let  $A$ be sectorial with sectoriality vertex $0$ and with $W_e(A)=\emptyset$.
Assume that $B$ is bounded and there exist $a,b> 0$ such that $$\|Cf\|^2\leq a\|f\|^2+b\, \re\,\langle Af,f\rangle, \quad f\in\dom(A)\cap\dom(C).$$
For $\lm\in\rho(D)$ define the bounded operator
$$\mathcal B_{\lm}:=\begin{pmatrix} I & 0 \\ 0 & \eps (D-\lm)^{-1} \end{pmatrix}, \quad \eps:=\frac{1}{b\|(D-\lm)^{-1}\|^2}.$$
Then $$\sigma_e(\cT)\subseteq \bigcap_{\lm\in\rho(D)}W_e(\mathcal B_{\lm}\mathcal T,\mathcal B_{\lm})\subseteq\sigma(D).$$
\end{Theorem}

\begin{proof} The first inclusion is immediate from (\ref{eq.We.pencil.spec}). 
Now let $\lm\in\rho(D)$. Take $(f,g)^t\in\dom(\mathcal T)$ with $\|f\|^2+\|g\|^2=1$.
We calculate
$$\re\,\langle \mathcal B_{\lm}(\mathcal T-\lm)(f,g)^t,(f,g)^t\rangle\\
=\re\,\langle (A-\lm)f,f\rangle +\re\,\langle Bg,f\rangle + \eps \re\, \langle (D-\lm)^{-1} Cf,g\rangle+\eps \|g\|^2.$$
Note that if $f=0$, then $\|g\|=1$ and $\re\,\langle \mathcal B_{\lm}(\mathcal T-\lm)(f,g)^t,(f,g)^t\rangle=\eps$.
 If $f\neq 0$, then, for any $\alpha,\beta>0$,
\begin{align*}
&\re\,\langle \mathcal B_{\lm}(\mathcal T-\lm)(f,g)^t,(f,g)^t\rangle\\
&\geq \re\langle Af,f\rangle -|\lm|\|f\|^2-\|B\|\|g\|\|f\|-\eps \|(D-\lm)^{-1}\|\|Cf\|\|g\|+\eps\|g\|^2\\
&\geq \re\langle Af,f\rangle -|\lm|\|f\|^2-\alpha \|B\|^2 \|g\|^2-\frac{1}{4\alpha}\|f\|^2-\beta\|Cf\|^2\\
&\quad-\frac{\eps^2}{4\beta} \|(D-\lm)^{-1}\|^2\|g\|^2+\eps\|g\|^2\\
&\geq \left((1-\beta b)\,\frac{\re\langle Af,f\rangle}{\|f \|^2}-|\lm|-\frac{1}{4\alpha}-\beta a\right)\|f\|^2\\
&\quad+\left(\eps -\alpha \|B\|^2-\frac{\eps^2}{4\beta} \|(D-\lm)^{-1}\|^2 \right)\|g\|^2.
\end{align*}
By setting $\beta=1/(2b)$ and $\alpha=\eps/(4\|B\|^2)$ (if $B\neq 0$ and $\alpha>0$ arbitrary otherwise) and using $ \|(D-\lm)^{-1}\|^2=1/(b\eps)$, we obtain
\begin{equation}\label{eq.estimatereal}
\begin{aligned}
&\re\,\langle \mathcal B_{\lm}(\mathcal T-\lm)(f,g)^t,(f,g)^t\rangle\\
&\geq \left(\frac{1}{2}\,\frac{\re\langle Af,f\rangle}{\|f \|^2}-|\lm|-\frac{1}{4\alpha}-\frac{a}{2b}\right)\|f\|^2
+\frac{\eps}{4}\|g\|^2.
\end{aligned}
\end{equation}

Now assume that there exist  $(f_n,g_n)^t\in\dom(\mathcal T)$, $n\in\N$, with $\|f_n\|^2+\|g_n\|^2=1$, $f_n\stackrel{w}{\to}0$, $g_n\stackrel{w}{\to}0$ and
\beq\label{eq.fngn}\langle \mathcal B_{\lm}(\mathcal T-\lm)(f_n,g_n)^t,(f_n,g_n)^t\rangle\tolong 0, \quad n\to\infty.\eeq
The above estimates and $\re\,W(A)\geq 0$ by the sectoriality of $A$ imply that there exist $\delta>0$ and $n_0\in\N$ such that $\|f_n\|\geq \delta$, $n\geq n_0$;
otherwise there would exist a subsequence on which $\|f_n\|\to 0$ and thus 
$$\limsup_{n\to\infty}\re\,\langle  \mathcal B_{\lm}(\mathcal T-\lm)(f_n,g_n)^t,(f_n,g_n)^t\rangle\geq \eps/4,$$
 a contradiction to~\eqref{eq.fngn}.
Since $\|f_n\|\geq \delta$, $n\geq n_0$, the normalised elements $f_n/\|f_n\|$ satisfy  $f_n/\|f_n\|\stackrel{w}{\to}0$.
The assumptions on $A$ imply that  $\re\langle Af_n,f_n\rangle/{\|f_n \|^2}\to\infty$. But then, using the estimate~\eqref{eq.estimatereal},
we arrive at the contradiction  $\re\,\langle  \mathcal B_{\lm}(\mathcal T-\lm)(f_n,g_n)^t,(f_n,g_n)^t\rangle\to \infty$.
Hence no such singular sequence  $((f_n,g_n)^t)_{n\in\N}\subset\dom(\mathcal T)$ exists, which proves $\lm\notin W_e(\mathcal B_{\lm}\mathcal T,\mathcal B_{\lm}).$
\end{proof}

\subsection{Application to Dirac operators, Stokes-type operators and Hain-L\"ust-type operators}

First we study Dirac operators in upper/lower spinor basis (compare~\cite[Section 2.3.2]{lewin-sere}).
In $L^2(\R^3,\C^2)\oplus L^2(\R^3,\C^2)$ consider
$$\cT:=\begin{pmatrix} I+V & \sigma\cdot (-\I\nabla) \\ \sigma\cdot(-\I\nabla) & -I+V \end{pmatrix}, 
\quad \dom(\cT):=W^{1,2}(\R^3,\C^2)\oplus W^{1,2}(\R^3,\C^2),$$
where $\sigma=(\sigma_1,\sigma_2,\sigma_3)^t$ and $\sigma_i\in\C^{2\times 2}$, $i=1,2,3$, are the Pauli matrices
and $V:\R^3\to\C$ is a (scalar) potential (real-valued in~\cite{lewin-sere}).
If we set $V\equiv 0$, the operator $\cT$ is the free Dirac operator $\cT_0$.
It is well known that $\cT_0$ is selfadjoint with $\sigma(\cT_0)=\sigma_e(\cT_0)=(-\infty,-1]\cup [1,\infty)$.
For a bounded potential  $V$, a Neumann series argument yields
$$\sigma(T)=\sigma(\cT_0+V)\subseteq \{\lm\in\C:\,{\rm dist}(\lm,\sigma(\cT_0))\leq \|V\|_{\infty}\}.$$
We want to improve this bound by taking into account the shape of ${\rm essran}(V)$.
To this end,
 we  study the pencil $\lm\mapsto \cB_{\theta}(\cT-\lm)$ where $\theta\in (-\pi/2,\pi/2)$ and $$\cB_{\theta}:={\rm diag}(\e^{-\I\,\theta},\e^{\I\,\theta})\quad\text{in}\quad L^2(\R^3,\C^2)\oplus L^2(\R^3,\C^2).$$
For $z\in\C$ and $-\pi\leq \theta_1<\theta_2\leq\pi$ denote the open sector 
$\mathcal S_{\theta_1,\theta_2}(z):=\{z+w:\,w\in\C\backslash\{0\},\,\arg(w)\in (\theta_1,\theta_2)\}$.

\begin{Theorem}\label{thm.dirac}
Let $V\in L^{\infty}(\R^3)$. 
Define 
\begin{align*}
\Sigma^- :=(-\infty,-1]+{\rm conv}({\rm essran}(V)),  & \;\;\;
\Sigma^+:=[1,\infty)+{\rm conv}({\rm essran}(V)),\\
  \Sigma:= \Sigma^-\, \cup \, \Sigma^+. &
\end{align*}
\begin{enumerate}
\item[\rm i)]
Let $z\in\C$ be such that there exists $\varphi\in [0,\frac{\pi}{2})$ for which either the sector
$\mathcal S:= \mathcal S_{\varphi,\pi-\varphi}(z)$ or $\mathcal S:= \mathcal S_{-\pi+\varphi,-\varphi}(z)$ satisfies the following:
\begin{quote}
$\mathcal S\subset \C\setminus\Sigma$, with  left boundary tangential to $\partial \Sigma^-$ and 
right boundary tangential to $\partial\Sigma^+$;
\end{quote}
let $\mathcal B:=\mathcal B_{\varphi}$ or $\mathcal B:=\mathcal B_{-\varphi}$, respectively.
Then $ \mathcal S \subseteq \big(\C\backslash W(\cB\cT, \cB\big)\cap \rho(\cT)$ and
$$\|(\cT-\lm)^{-1}\|\leq \frac{1}{\cos(\varphi)\,{\rm dist}(\lm,\partial \mathcal S)}, \quad\lm\in\mathcal S.$$
\item[\rm ii)] Let $\mathcal F$ be the set of all sectors $\mathcal S$ having the properties in part {\rm i)}. Let
\[ \widetilde{\Sigma} := \C\setminus\bigcup_{\mathcal S\in{\mathcal F}}\mathcal S. \]
Then the spectrum satisfies $\sigma(\cT)\subseteq\widetilde \Sigma$. \end{enumerate}
\end{Theorem}

\begin{Remark}\label{rem.dirac}
If ${\rm esssup}\re\,V-{\rm essinf}\re\,V<2$, then $\Sigma^-$ and $\Sigma^+$ are disjoint.
For a better understanding of the set $\widetilde\Sigma$ see Figure~\ref{fig.dirac}. 
\end{Remark}

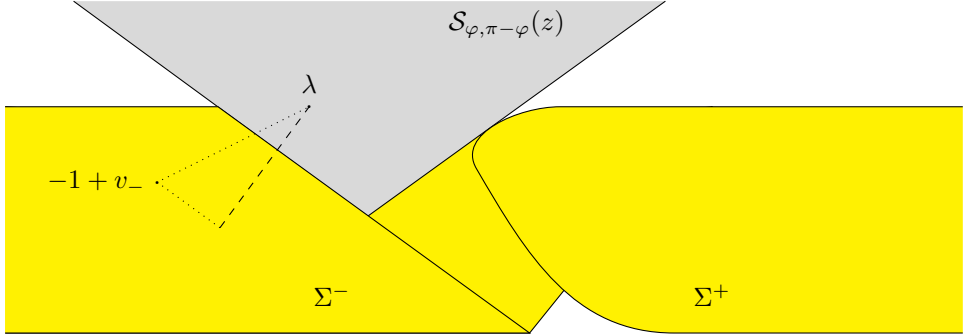
\begin{figure}[htb]
\begin{center}
\begin{tikzpicture}
\filldraw[fill=yellow,draw=black] (-6.3,-1.5) to  (0.6,-1.5) to (3,1.5) to (-0.22,1.1) to (-6.3,1.5);
\filldraw[fill=gray!30!white,draw=black] (-5.4,2.9) to (-1.533,0.0463) to (2.3868,2.9);
\draw (0.3,2.6) node{$\mathcal S_{\varphi,\pi-\varphi}(z)$};
\draw (-2.3,1.5) node {.};
\draw (-2.3,1.8)  node{$\lambda$};
\filldraw[fill=yellow,draw=black] (6.3,1.5) to (1,1.5)  to[out=180, in=120] (-0.1,0.68) to[out=300,in=180] (2.5,-1.5) to (6.3,-1.5);
\filldraw[fill=yellow,draw=black] (-6.3,-1.5) to  (0.6,-1.5) to (-3.5,1.5) to (-6.3,1.5);
\draw (-2,-1) node{$\Sigma^-$};
\draw (3,-1) node{$\Sigma^+$};
\draw (-4.3,0.5) node {.};
\draw (-5.1,0.5)  node{$-1+v_-$}; 
\draw[dotted] (-3.4740,-0.1044) to (-4.3,0.5) to  (-2.3,1.5);
\draw[dashed] (-3.4740,-0.1044) to (-2.3,1.5);
\end{tikzpicture}
\end{center}
\caption{The set $\widetilde\Sigma\supset\Sigma^-\cup\Sigma^+$ (yellow) and sector $\mathcal S_{\varphi,\pi-\varphi}$ (grey) containing $\lambda$. The dashed line is perpendicular to the sector's left boundary and measures $\big|\im\big({\rm e}^{\I\,\varphi}(-1+v_--\lm)\big)\big|$ used in the proof.}
\label{fig.dirac}
\end{figure}

\begin{proof}[Proof of Theorem~{\rm\ref{thm.dirac}}]
%
i)
We prove the claim for $\mathcal S=\mathcal S_{\varphi,\pi-\varphi}(z)$; the proof is analogous for $\mathcal S=\mathcal S_{-\pi+\varphi,-\varphi}(z)$.
First let $\lm\in\C$ be arbitrary.
Let $(f,g)^t\in\dom(\cB_{\varphi}\cT)=\dom(\cT)$ with $\|f\|^2+\|g\|^2=1$. Define 
$$u:=\langle \sigma\cdot(-\I\nabla)g,f\rangle=\langle g,\sigma\cdot(-\I\nabla)f\rangle, \quad t:=\|f\|^2\in [0,1].$$
Then there exist $v_-,v_+\in \overline{W(V)}={\rm conv}({\rm essran}(V))$ such that
$\langle Vf,f\rangle=v_- \|f\|^2$ and $\langle Vg,g\rangle=v_+\|g\|^2$. This implies
\begin{align*}
&\langle (\cB_{\varphi}(\cT-\lm)(f,g)^t,(f,g)^t\rangle\\
&=t \,{\rm e}^{-\I\,\varphi}(1+v_+-\lm)+(1-t){\rm e}^{\I\,\varphi}(-1+v_--\lm)+ {\rm e}^{-\I\,\varphi} u+{\rm e}^{\I\,\varphi}\overline{u}
\end{align*}
and hence
\be
\begin{aligned}
&\im\,\langle (\cB_{\varphi}(\cT-\lm)(f,g)^t,(f,g)^t\rangle\\
&=t\, \im\big({\rm e}^{-\I\,\varphi}(1+v_+-\lm)\big)+(1-t)\im\big({\rm e}^{\I\,\varphi}(-1+v_--\lm)\big).
\end{aligned}
\ee
Note that $-1+v_-\in\Sigma^-$ and $1+v_+\in \Sigma^+$.
For $\lm\in\mathcal S$, one may show that the assumptions on $\mathcal S$ and the convexity of $\Sigma^-$, $\Sigma^+$ imply
 (see Figure \ref{fig.dirac}) 
$$\max\big\{\im\big({\rm e}^{\I\,\varphi}(-1+v_--\lm)\big),\im\big({\rm e}^{-\I\,\varphi}(1+v_+-\lm)\big)\big\}\leq  -{\rm dist}(\lm,\partial \mathcal S)<0.$$
Hence $0\notin \overline{W(\cB_{\varphi}(\cT-\lm))}$, i.e.\ $\lm\notin  W(\cB_{\varphi}\cT, \cB_{\varphi})$.
As $V\in L^{\infty}(\R^3)$, a Neumann series argument yields that $\mathcal S\cap\rho(\cT)\neq\emptyset$. Now the rest of the claim follows from Theorem~\ref{thm.2x2.intersec}~ii);
note that ${\rm dist}(0,W(\mathcal B_{\varphi}))=\cos(\varphi)$.

ii) For all $\lambda\in\C\backslash\widetilde \Sigma$ there exists 
$\mathcal S\in\mathcal F$ such that $\lambda\in\mathcal S$. Now the claim follows from~i).
\end{proof}

\begin{Remark}
\begin{enumerate}
\item[\rm i)]
In \cite[Theorem~2.4]{lewin-sere} it has been shown for real-valued $V$ decaying at infinity that if we approximate $\cT$ by a projection method that respects the decomposition in upper/lower spinor basis, then spectral pollution is confined to $\Sigma$, i.e.\ every eigenvalue accumulation point outside the latter set is a true eigenvalue of $\cT$. However, we have just proved that the eigenvalues of $\cT$ are all contained in $\Sigma$.

\item[\rm ii)]
We have set the dimension $d=3$ in order to compare our results to the ones in \cite{lewin-sere} (see Remark i)). However, one may apply the multiplier trick to Dirac operators in other dimensions.
In addition, it is possible to allow for unbounded potentials $V$ although the results are only non-trivial if $\re V$ is bounded and $\im V$ is semibounded.
If for instance $\im V$ is unbounded from above but bounded from below, then the imaginary part of $\widetilde\Sigma$ is unbounded from above and it is impossible to find a  $\varphi\in \big[0,\frac\pi 2\big)$ and $z\in \C$ such that $\mathcal S_{\varphi,\pi-\varphi}(z)\subset\C\backslash\widetilde \Sigma$. However, for every $\lm\in\C\backslash\widetilde \Sigma$ (which is a connected set) there exist $z\in\C\backslash\widetilde \Sigma$ and $\varphi\in \big[0,\frac\pi 2\big)$ such that $\lm\in\mathcal S_{-\pi+\varphi,-\varphi}(z)\subset\C\backslash\widetilde \Sigma$. This proves $\sigma_p(\cT)\subseteq \widetilde \Sigma$. If, in addition, we know that $(\C\backslash\widetilde\Sigma)\cap\rho(\cT)\neq\emptyset$, we can conclude $\sigma(\cT)\subseteq \widetilde \Sigma$ and the resolvent norm estimates in claim i) of Theorem~\ref{thm.dirac} hold.

\item[\rm iii)]
Spectral enclosures for non-selfadjoint Dirac operators were proved in \cite{cuenin-laptev-tretter} 
and in~\cite[Section~5.1]{cuenin-tretter}.
However, the enclosures are given in terms of $L^{1}$-norms of $V$ and do not take into account the shape of ${\rm essran}(V)$.
\end{enumerate}
\end{Remark}

We illustrate Theorem~\ref{thm.distances} for Stokes-type operators.

\begin{Theorem}\label{thm.stokes}
For $U\in L^2_{\rm loc}(\R)$ and $\gamma,\delta\in\C$ with $|\gamma|=|\delta|=1$ define
$$\mathcal T:=\bmat -\frac{\rd^2}{\rd x^2} & \gamma \frac{\rd}{\rd x}\\  \delta \frac{\rd}{\rd x} & U\emat,\quad
\dom(\cT):=W^{2,2}(\R)\oplus \{f\in W^{1,2}(\R):\,Uf\in L^2(\R)\}.$$
The operator matrix $\mathcal T$ is not closed but closable, with $\sigma_{\rm app}(\overline{\mathcal T})$ contained in
\begin{equation}\label{eq.enclosureStokes}
\begin{aligned}
&\big\{\lm\in\C:\,\re\,\lm<0,\,{\rm dist}(\lm,{\rm essran}(U))\leq 1\big\}\cup [0,\infty)\\ 
&\cup\left\{\lm\in\C:\,\re\,\lm\geq 0, \,\im\,\lm\neq 0,\,\,{\rm dist}(\lm,{\rm essran}(U))\leq \frac{|\lm|}{|\im\,\lm|}\right\}.
\end{aligned}
\end{equation}
The set in particular contains the $1$-neighbourhood of ${\rm essran}(U)$.
\end{Theorem}

\begin{proof}
We apply Theorem~\ref{thm.distances} to $\cT$ with operator entries $A$, $B$, $C$, $D$.
Note that $\overline{W(A)}=[0,\infty)$ and $\sigma(D)={\rm essran}(U)$.
Let $\lm\in\rho(D)\backslash\overline{W(A)}$
and choose $\varphi\in (-\pi/2,\pi/2)$ such that 
\begin{equation}\label{eq.rphi}
r_{\lm,\varphi}=\inf\,\re\big(\e^{\I\varphi}W(A-\lm)\big)=\re\left(-\e^{\I\varphi}\lm\right)>0.
\end{equation}
Then the relative boundedness assumptions~\eqref{eq.relbounds} are satisfied with
$$a_{\lm,\varphi}=c_{\lm,\varphi}=0, \quad b_{\lm,\varphi}=d_{\lm,\varphi}=\frac{1}{\cos\varphi}.$$
So the  inequality~\eqref{eq.distances} holds if
\begin{equation}\label{eq.distu}
{\rm dist}(\lm,{\rm essran}(U))=\|(D-\lm)^{-1}\|^{-1}>\frac{1}{\cos\varphi}.
\end{equation}
For $\lm\in\C$ with $\re\,\lm<0$, the condition~\eqref{eq.rphi} holds for $\varphi=0$ and correspondingly~\eqref{eq.distu} is satisfied if  ${\rm dist}(\lm,{\rm essran}(U))>1$.
Now let $\re\,\lm\geq 0$.  If $\im\,\lm>0$, we restrict our attention to angles $\varphi\in (\pi/2-{\rm arg}\,\lm, \pi/2)$ to guarantee~\eqref{eq.rphi}.
Then the condition~\eqref{eq.distu} is satisfies for all sufficiently small such $\varphi$ provided that 
$${\rm dist}(\lm,{\rm essran}(U))>\inf_{\varphi\in  \left(\frac{\pi}{2}-{\rm arg}\,\lm, \frac{\pi}{2}\right)}\frac{1}{\cos\varphi}=\frac{1}{\cos\left(\frac{\pi}{2}-{\rm arg}\,\lm\right)}=\frac{|\lm|}{\im\,\lm}.$$
In the case $\re\,\lm\geq 0$, $\im\,\lm<0$ we arrive at the condition ${\rm dist}(\lm,{\rm essran}(U))>|\lm|/|\im\,\lm|$.
\end{proof}

\begin{Example}
For a first example let $U$ be constant. Then the spectrum of $\overline{\mathcal T}$ is easy to calculate using Fourier transforms; it is given by
$$\sigma(\overline{\mathcal T})=\sigma_e(\overline{\mathcal T})=\overline{\left\{\frac{k^2+U}{2}\pm\sqrt{\left(\frac{k^2-U}{2}\right)^2-\gamma\delta k^2}:\,k\in\R\right\}}.$$
In Figure~\ref{figStokesConst} the spectrum is plotted for $U=-1+\I$ and different values of $\gamma\delta=\e^{\I\phi}$ (different colours) inside the enclosure (white) given by~\eqref{eq.enclosureStokes}.
We see that the enclosure is sharp in the sense that for all point $z$ in it there exists a $\phi$ such that $z\in\sigma(\overline{\mathcal T})$.
\begin{figure}[tbh]
\begin{center}
\includegraphics[width=0.7\textwidth]{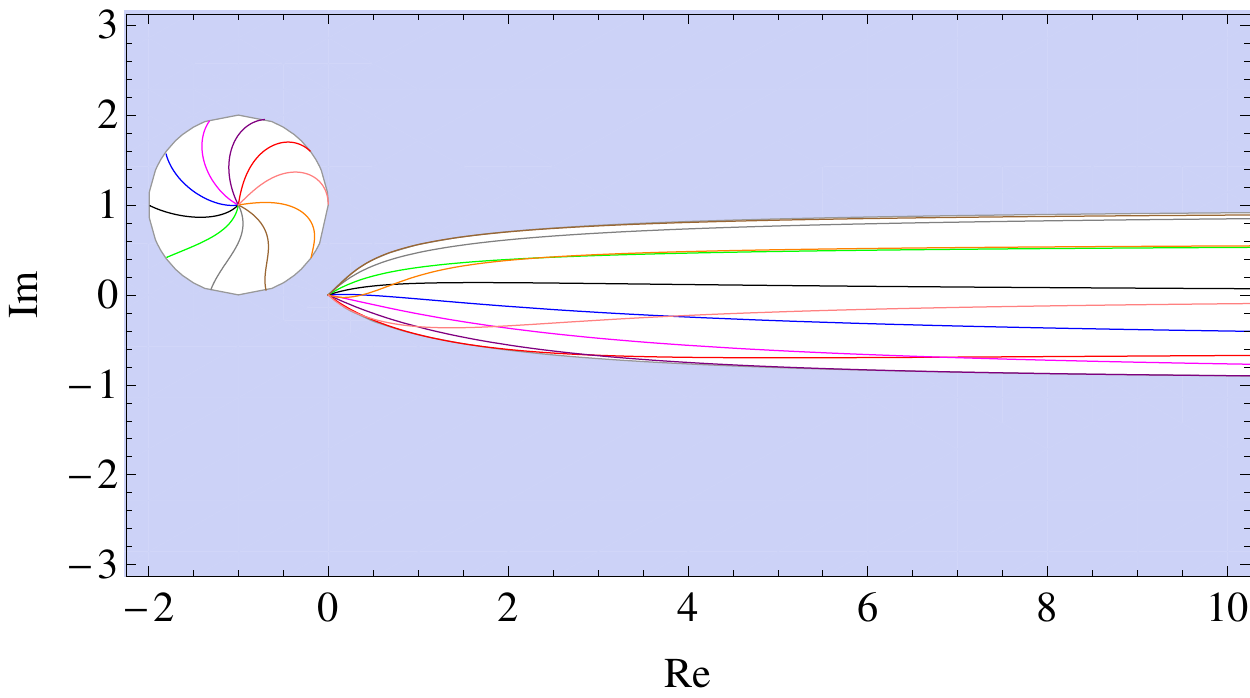}
\end{center}
\caption{Enclosure (white) in~\eqref{eq.enclosureStokes} for $U=-1+\I$ and spectrum (different colours) for different $\gamma\delta=\e^{\I\phi}$.}
\label{figStokesConst}
\end{figure}

\begin{figure}[htb]
\begin{center}
\includegraphics[width=0.32\textwidth]{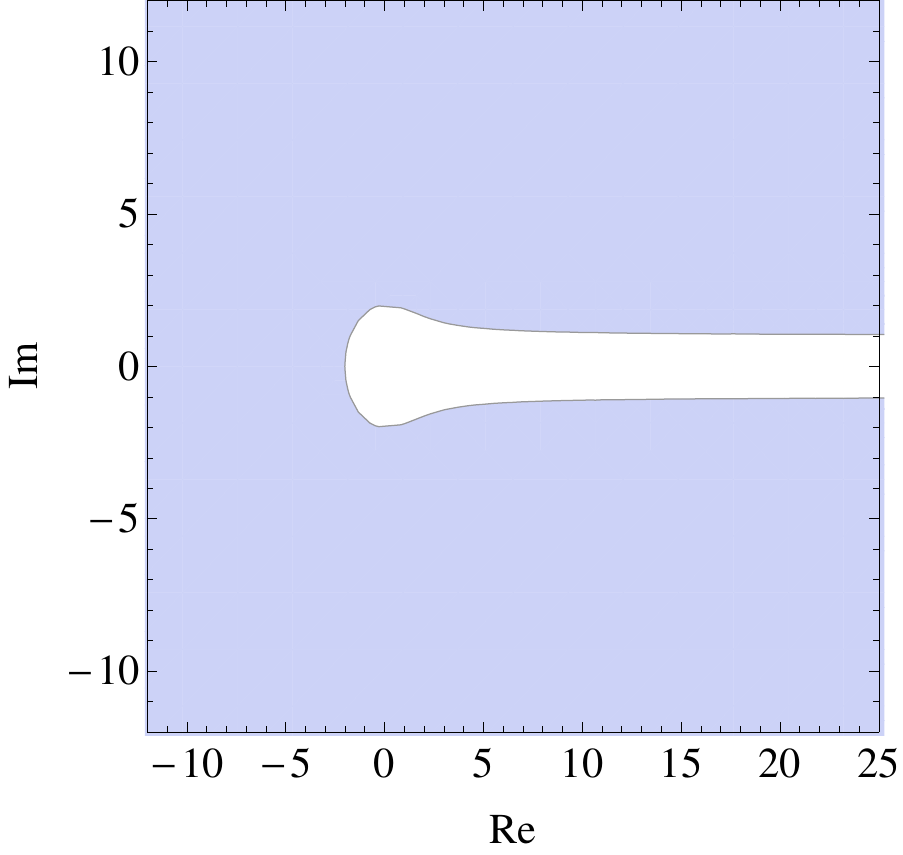}
\includegraphics[width=0.32\textwidth]{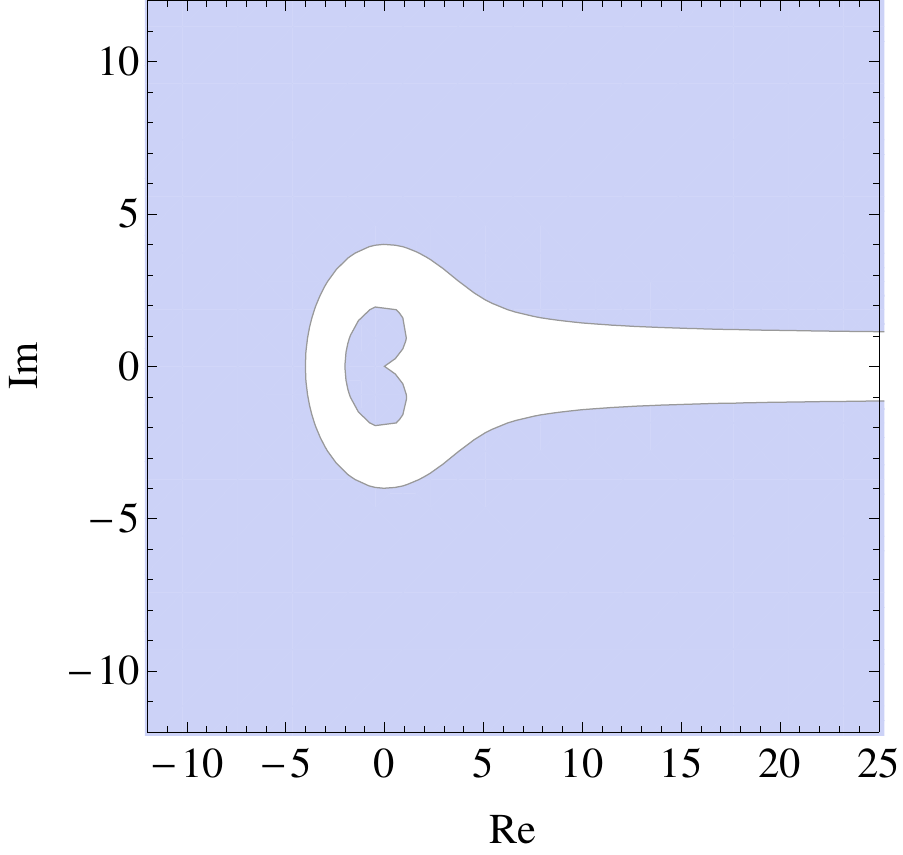}
\includegraphics[width=0.32\textwidth]{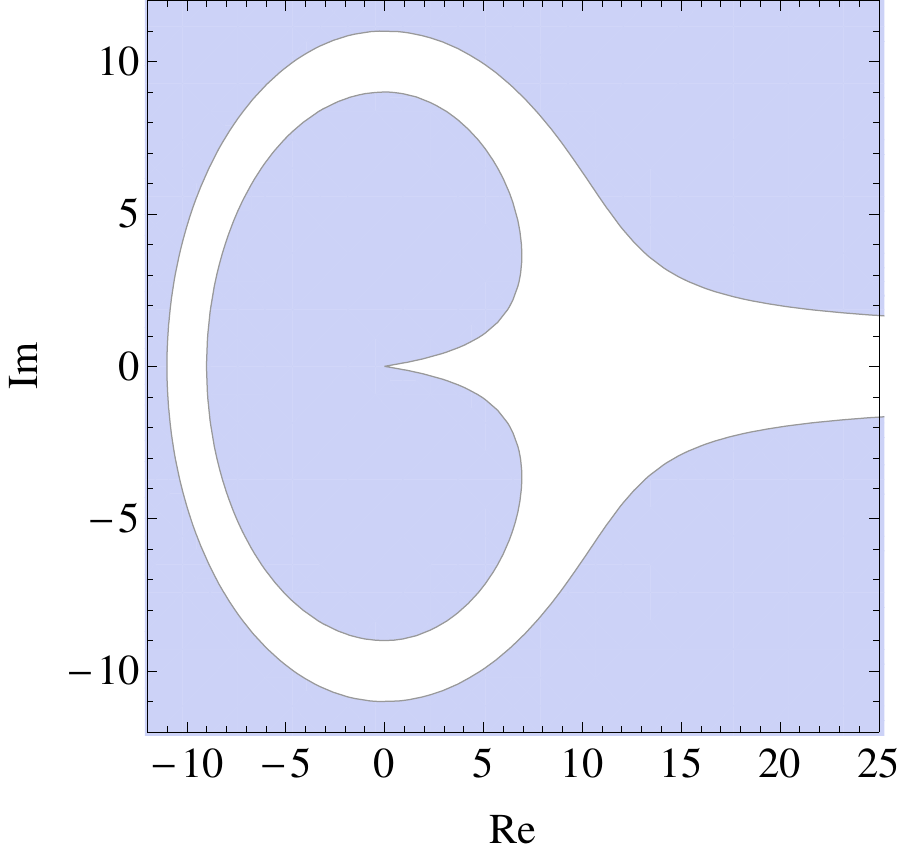}
\end{center}
\caption{Enclosure (white) in~\eqref{eq.enclosureStokes} for ${\rm essran}(U)=\{z\in\C:\,|z|=R\}$ and $R=1$ (left), $R=3$ (middle) and $R=10$ (right).}
\label{figStokes}
\end{figure}

As a second example let ${\rm essran}(U)=\{z\in\C:\,|z|=R\}$ be the circle of some radius $R>0$.
Figure~\ref{figStokes} illustrates the enclosure~\eqref{eq.enclosureStokes} for three different values of $R$. For small $R>0$ the enclosure is simply connected and for increasing $R$ a hole opens up near the origin (with origin on the boundary).
Note that $[R,\infty)+\I\,[-1,1]$ is contained in the enclosure.
\end{Example}

Finally we apply Theorems~\ref{thm.intersection.We},~\ref{thm.trunc} to Hain-L\"ust-type operators on the whole real line.

\begin{Theorem}\label{thm.hain}
Let $Q,V,U\in L_{\rm loc}^{\infty}(\R)$, $W\in L^{\infty}(\R)$ with
\begin{alignat*}{3}
&\exists\, \theta\in [0,\pi/2):\quad &{\rm essran}(Q) &\subseteq\{\lm\in\C:\,|{\rm arg}\lm|\leq\theta\}, \quad \lim_{|x|\to\infty}|Q(x)|=\infty,\\
&\exists\, b\geq 0:\quad &|V(x)|^2&\leq b |Q(x)| \quad\text{for a.e.\ } x\in\R.
\end{alignat*}
Define the operator
\begin{align*}
\cT&:=\bmat -\frac{\rd^2}{\rd x^2}+Q & W \\ V & U\emat, \\
\dom(\cT)&:=\{f\in W^{2,2}(\R):\,Qf\in L^2(\R)\}\oplus\{f\in L^2(\R):\,Uf\in L^2(\R)\}.
\end{align*}
\begin{enumerate}
\item[\rm i)]
If $\cB_{\lm}$, $\lm\notin {\rm essran}(U)$, are defined as in Theorem{\rm~\ref{thm.intersection.We}}, then
\beq\label{eq.intersection.hain}
\sigma_e(\cT)=\underset{\lm\notin{\rm essran}(U)}{\bigcap}W_e(\cB_{\lm}\cT,\cB_{\lm})={\rm essran}(U).
\eeq

\item[\rm ii)]
If we truncate $\cT$ to $\cT_n$, $n\in\N$, with 
$$\dom(\cT_n):=\{f\in W^{2,2}(-n,n):\,f(\pm n)=0\}\oplus L^2(-n,n), \quad n\in\N,$$ 
then no spectral pollution occurs and all discrete eigenvalues are approximated.
\end{enumerate}
\end{Theorem}

\begin{proof}
i) 
Define $A:=-{\rd}^2/{\rd x}^2+Q$ with $\dom(A):=\{f\in W^{2,2}(\R):\,Qf\in L^2(\R)\}$.
By~\cite[Corollary~VII.2.7]{edmundsevans}, $A$ is closed and $C_0^{\infty}(\R)$ is a core of $A$,  and $W_e(A)=\emptyset$ by Rellich's criterion~\cite[Theorem XIII.65]{reedsimon}. 
The operator $V$ is $A$-bounded with relative bound $0$. This follows since, for every $f\in\dom(A)$,
$$\|Vf\|^2\leq b|\langle Qf,f\rangle|\leq b|\langle Af,f\rangle\leq b\|Af\| \|f\|\leq \frac{b}{2}(\eps\|Af\|^2+\eps^{-1}\|f\|^2)$$
for any $\eps>0$.
Together with $W\in L^{\infty}(\R^d)$ we conclude that the operator matrix $\cT$ is diagonally dominant of order $0$;
 it is closed by~\cite[Corollary~2.2.9~i)]{tretter}.
 Theorem~\ref{thm.intersection.We} yields the sequence of inclusions 
$$\sigma_e(\cT)\subseteq \underset{\lm\notin{\rm essran}(U)}{\bigcap}W_e(\cB_{\lm}\cT,\cB_{\lm})\subseteq {\rm essran}(U).$$
Then the equality~\eqref{eq.intersection.hain} follows from ${\rm essran}(U)=\sigma_e(U)=\sigma_e(\cT)$ by~\cite[Theorem~2.4.8]{tretter}; note that a different definition of the essential spectrum was used in~\cite{tretter} but for this example it coincides with the definition used here.

ii) First note that $0$ is an eigenvalue of $\cT_n-\mu$ if and only if it is an eigenvalue of $\cB_{\lm}(\cT-\mu)$ truncated to $\dom(\cT_n)$. Now Theorem~\ref{thm.trunc}  implies that spectral pollution is confined to the set in~\eqref{eq.intersection.hain} and each isolated point of $\sigma(\cT)$ outside this set is approximated;
note that $\Phi:=C_0^{\infty}(\R)\oplus C_0^{\infty}(\R)$ is a core of $\cB_{\lm}\cT$ and of its adjoint operator, and, for every $n\in\N$,
$$\dom(\cT_n)\subset \{f\in W^{1,2}(\R):\,Q|f|^2\in L^1(\R)\}\oplus\{f\in L^2(\R):\,U|f|^2\in L^1(\R)\}.$$ 
Since $\sigma_e(\cT)\subseteq \sigma(\cT)$, no spectral pollution can occur.
\end{proof}

\subsection*{Acknowledgements}
The authors are very grateful for the comments and corrections of the referee, whose
detailed attention allowed us to make very worthwhile improvements.
The first author acknowledges financial support of the Swiss National Science Foundation (SNF), Early Postdoc.Mobility project P2BEP2\_159007. Some later parts of this work were completed while the first author was a Chapman Fellow at Imperial College London.

\bibliography{mybib}{}
\bibliographystyle{acm}

\end{document}